\documentclass[11pt]{article}
\usepackage{amssymb}
\usepackage{mathrsfs}
\usepackage{latexsym,amsmath,amssymb,amsfonts,epsfig,graphicx,cite,psfrag}
\usepackage{eepic,color,colordvi,amscd}
\definecolor{blue}{rgb}{0,0,0.9} 
\definecolor{red}{rgb}{0.9,0,0} 
\definecolor{green}{rgb}{0,0.9,0}

\usepackage{ebezier} 
\usepackage{verbatim} 
\usepackage{subfigure} 
\usepackage{enumitem} 
\usepackage{float}    
\usepackage{placeins} 
\usepackage{capt-of}
\usepackage{algorithm} 
\usepackage{longtable}
\usepackage{algorithmic} 
\usepackage{amsthm} 
\usepackage{comment}
\usepackage{blkarray}
\usepackage[hypertexnames=false]{hyperref}
\makeatletter
\newcommand{\myitem}[2][]{%
  \item[#1]%
  \protected@edef\@currentlabel{#1}
  \label{#2}%
}
\makeatother
\usepackage{color}
\theoremstyle{plain}
\newtheorem{theo}{Theorem}[section]

\newtheorem{lem}[theo]{Lemma}
\newtheorem{coro}[theo]{Corollary}
\newtheorem{assump}[theo]{Assumption}

\newtheorem{prop}[theo]{Proposition}

\usepackage{tikz}
\usetikzlibrary{positioning}
\usetikzlibrary{calc}
\usepackage{enumitem}

\theoremstyle{definition}

\theoremstyle{remark}
\newtheorem{rem}[theo]{Remark}
\def\qed{\hfill \rule{4pt}{7pt}}
\def\pf{\noindent {\it Proof.\quad}}
\def\Tr{\mathrm{Tr}}

\def\<{\left\langle}
\def\>{\right\rangle}
\def\E{\mathcal{E}}
\def\M{\mathcal{M}}

\def\DD{{\rm Diag}}
\def\dd{{\rm d}}
\def\A{\mathcal{A}}
\def\B{\mathcal{B}}

\def\I{\mathcal{I}}

\def\R{\mathbb{R}}

\def\X{\mathcal{X}}
\def\Y{\mathcal{Y}}
\def\Z{\mathcal{Z}}

\def\G{\mathcal{G}}

\def\N{\mathbb{N}}

\def\S{\mathbb{S}}

\def\rr{{\rm rank}}

\def\P{\mathcal{P}}

\def\O{\mathcal{O}}

\def\({\left(}
\def\){\right)}

\def\C{\mathcal{C}}

\newcommand{\mubar}{\bar\mu}
\newcommand{\nubar}{\bar\nu}

\newcommand{\Lip}{\operatorname{Lip}}

\let\svthefootnote\thefootnote
\newcommand\blankfootnote[1]{%
	\let\thefootnote\relax\footnotetext{#1}%
	\let\thefootnote\svthefootnote%
}

\begin{document}
	\title{Convex relaxation approaches for high-dimensional optimal transport}
	\author{Yuehaw Khoo
	\thanks{Department of Statistics, University of Chicago, 
	({\tt ykhoo@uchicago.edu}). The research of this author is partially funded by NSF DMS-2339439, DOE DE-SC0022232, DARPA The Right Space HR0011-25-9-0031, and a Sloan research fellowship.}, \quad
	Tianyun Tang
\thanks{Department of Statistics, University of Chicago, ({\tt ttang@u.nus.edu}).
         }
	 }
	\date{\today} 
	\maketitle
	

\begin{abstract}
Optimal transport (OT) is a powerful tool in mathematics and data science but faces 
severe computational and statistical challenges in high dimensions. We propose convex 
relaxation approaches based on marginal and cluster moment relaxations that exploit 
locality in the distributions. These methods approximate high-dimensional couplings using 
low-order marginals and sparse moment statistics, yielding semidefinite programs that 
provide lower bounds on the OT cost with greatly reduced complexity. For Gaussian measures with sparse correlations, we prove an exponential convergence rate for the cluster moment relaxation and an improved statistical error bound. We also establish approximation error bounds for the marginal relaxation when the reference measures are local perturbations of mean-field measures. In addition, we demonstrate how to extract transport maps from our relaxations, offering a simpler and interpretable alternative to neural networks in generative modeling. Extensive numerical experiments demonstrate strong empirical performance across a range of distributions. Our results suggest that convex relaxations can provide a promising path for dimension reduction in high-dimensional OT.
\end{abstract}

\section{Introduction}

\subsection{Optimal transport}

In this paper, we consider the following optimal transport (OT) problem:  
\begin{equation}\label{OT}
\inf_\pi\left\{ \int_{\X\times \Y} c(x,y) {\rm d}\pi(x,y):\ \pi\in \Pi(\mu,\nu)\right\}, \tag{{\rm OT}}
\end{equation}
where $d \in \N^+$, $\mu$ and $\nu$ are probability measures on Borel sets $\X,\Y \subset \R^d$, $c:\X\times\Y \to \R^+$ is a lower semi-continuous cost function, and $\Pi(\mu,\nu)$ denotes the set of joint probability measures on $\X \times \Y$ with marginals $\mu$ and $\nu$ (the so-called transport plans).

Originating in the works of Monge and Kantorovich \cite{monge1781memoire,kantorovich1942translocation}, this problem has grown into a powerful mathematical framework with deep ties to analysis, geometry, partial differential equations, and optimization \cite{friesecke2024optimal,villani2021topics,villani2008optimal,cloninger2025linearized,calder2022improved,scarvelis2023riemannian}. Beyond pure mathematics, OT has become a central tool in applications across the sciences. In machine learning and data science, it plays a key role through the Wasserstein distance, which measures similarity between probability distributions while respecting their underlying geometry. This geometric viewpoint has enabled advances in generative modeling \cite{tolstikhin2017wasserstein,arjovsky2017wasserstein,genevay2018learning,balaji2020robust,cheng2024convergence}. Compared to divergences such as the Kullback–Leibler divergence \cite{kullback1951information}, OT-based methods often lead to more stable training and are better at handling distributions with disjoint supports \cite{peyre2019computational}.

\subsection{Curse of dimensionality}

Despite its promise, applying OT in high dimensions remains difficult. The computational cost grows quickly with sample size, and the sample size needed to estimate OT scales exponentially with dimension \cite{fournier2015rate,dudley1969speed,weed2019sharp}. This creates a major obstacle for large-scale machine learning and data science, where data often lie in very high-dimensional spaces. To mitigate this, several dimension-reduction strategies have been proposed.

One popular approach is to use \emph{neural network parameterizations} of transport maps or of the Kantorovich potentials (the dual variables in OT) \cite{makkuva2020optimal,korotin2022neural,mokrov2021large}. These methods are widely applied in generative modeling, but their training involves non-convex optimization and thus lacks strong theoretical guarantees.

Another strategy is the \emph{sliced Wasserstein} distance, which computes one-dimensional OT along random (or learned) projections and then averages the results \cite{rabin2011wasserstein,bonneel2015sliced,paty2019subspace,lin2020projection}. This substantially reduces computational cost, though the resulting distance is generally different from the true Wasserstein distance.

A further line of work \cite{weed2019estimation,vacher2021dimension,hutter2021minimax} leverages the smoothness of the distributions to obtain statistical rates for Wasserstein distance and transport map estimation that avoid exponential dependence on dimension. However, these improvements require strong regularity assumptions and often involve high computational complexity, limiting their use in large-scale applications.

\subsection{Our contributions}

In this paper, to address the high dimensionality of \eqref{OT}, we make the following contributions:

\begin{itemize}
\item We introduce convex relaxations of \eqref{OT} that approximate high-dimensional distributions using only sparse collections of low-order marginals or cluster moments. The resulting semidefinite programs provide computable lower bounds on the OT cost.

\item {We provide a theoretical analysis of the proposed relaxations under structured models. For Gaussian models with sparse precision structure, we prove that the cluster moment relaxation achieves an exponential approximation rate and gives an improved statistical error bound based on empirical means and sparse covariance entries. For local perturbations of mean-field product measures, we establish approximation error bounds for the marginal relaxation. In addition, we show that these relaxations lead to significant reductions in both computational and sample complexity.}

\item  Numerical experiments further suggest that the approach behaves robustly in several non-Gaussian settings, where we observe linear scaling in dimension for both error and running time, and constant scaling with respect to sample size.
\end{itemize}

Marginal relaxation has a long history in areas such as graphical models, density functional theory, and statistical physics \cite{wainwright2008graphical,an1988note,pelizzola2005cluster}, where it is commonly applied to complex many-body systems. In this work, we adapt the idea to optimal transport by applying it to the joint distribution $\pi(x,y)$, with both marginals $\mu$ and $\nu$ prescribed. Related work by Khoo et al.~\cite{khoo2019convex,khoo2020semidefinite} introduced marginal relaxation techniques for \emph{multi-marginal} OT, where the difficulty comes from coupling $N$ low-dimensional variables. In contrast, our setting involves only two variables, but each lies in a very high-dimensional space. While the sliced Wasserstein distance also relies on low-order marginals to define a tractable distance, our approach is designed to closely approximate the original Wasserstein distance.

In \cite{mula2024moment}, Mula and Nouy proposed sum-of-squares (SOS) moment relaxations \cite{lasserre2001global,lasserre2008semidefinite} of \eqref{OT}, which approximate high-dimensional distributions through low-order statistics. Their method avoids spatial discretization, making it particularly suitable for \eqref{OT} between continuous distributions. Our cluster moment relaxation is inspired by their approach. However, we realize that their method {does} not scale well in high-dimensional setting, because one has to solve a semidefinite program (SDP) of size $\binom{2d+n}{n}$, where $n$ is the relaxation degree—a combinatorial growth that quickly becomes memory-prohibitive for large $d$ and $n$. In contrast, our cluster moment relaxation leverages the local structure of \eqref{OT} and requires only a sparse collection of moments, resulting in much smaller SDP blocks and computations that remain tractable even in high dimensions.

Another related work is by Vacher and Bach \cite{vacher2021dimension}, who use sum-of-squares representations of kernel functions in the dual (Kantorovich) formulation of \eqref{OT} to model smooth nonnegative functions. Their dimension-free guarantees are proved under smoothness assumptions on the underlying distributions. Our theoretical results below also impose model-specific regularity assumptions in the regimes where quantitative rates are proved; see, for instance, the local perturbation result in Theorem~\ref{thm:mfloc}. The main difference is structural: our cluster moment relaxation is designed to exploit locality, so that only low-dimensional marginal information and sparse collections of moments are used. This perspective is reflected in the local approximation result of Theorem~\ref{thm:mfloc} and Remark~\ref{rem:mf-relative}, and in the statistical bounds of Proposition~\ref{prop:samplecomplex} and Theorem~\ref{thm:sdpvalue}.

\subsection{Organization}

The rest of the paper is organized as follows. Section~\ref{Sec:cvxrl} introduces our convex relaxation approaches for \eqref{OT}. {Section~\ref{Sec:analysis} provides theoretical analysis of the approximation error bounds as well as statistical error of our convex relaxation approaches.} In Section~\ref{Sec:extmap}, we describe how to extract transport maps from the relaxation. {In Section~\ref{Sec:prepro}, we discuss preprocessing methods for our convex relaxation approaches.} Section~\ref{Sec:numer} presents numerical experiments that illustrate the effectiveness of our methods. Section~\ref{Sec:conc} concludes with a brief summary and discussion. {Useful technical lemmas and detailed proofs are given in Appendix~\ref{sec:useful-lemmas} and Appendix~\ref{sec:proof}.}

\section{Convex relaxations}\label{Sec:cvxrl}

In this section, we propose a convex relaxation framework for \eqref{OT}. 
To formally introduce the relaxation approaches, we first present the relevant notations, 
definitions, and assumptions in the following subsection. 
For intuition, readers may also refer to Figures~\ref{fig:MF}–\ref{fig:HG}.

\subsection{Preliminaries}

In what follows, we introduce the basic notation and structures that support our convex relaxation framework. Definition~\ref{itm:first} sets up index partitions and the associated product spaces. Definition~\ref{itm:second} introduces moments of probability measures. Definition~\ref{itm:third} uses graphs to encode sparsity patterns and correlation structure. Finally, Definition~\ref{itm:fourth} defines projection operators for extracting marginals and moments.

\begin{itemize}
\myitem[D1]{itm:first}(indices, partitions, marginals)
For any integer $n\in \N^+,$ define sets $[n]:=\{1,2,\ldots,n\}$ and 
\begin{equation}
[n]_2:=\left\{ij:\ i,j\in \N^+,\  1\leq i<j\leq n\right\},
\end{equation} 
{where we use $ij$ to denote the ordered pair $(i,j).$}
Fix $K\in \N^+.$ Partition {the coordinates of} $x$ and $y$ {in $\R^d$} into $K$ clusters $(x_1;x_2;\dots;x_K)$ and $(y_1;y_2;\dots;y_K),$ where each cluster may contain multiple variables.  
We {require} $x_k$ and $y_k$ to share the same coordinates in $x$ and $y$.  
For each $k\in [K],$ let $\X_k$ and $\Y_k$ denote the Borel subsets of $\X$ and $\Y$ corresponding to clusters $x_k$ and $y_k,$ respectively{; equivalently, they are the coordinate subspaces determined by the corresponding coordinate hyperplanes}. Then
\begin{equation}\label{XYdec}
\X=\X_{1}\times \X_{2}\times \cdots\times \X_{K},\qquad 
\Y=\Y_{1}\times \Y_{2}\times \cdots\times \Y_{K}.
\end{equation}
For simplicity, we write
\begin{equation}\label{decvar}
z=(x,y),\quad\mathcal{Z}=\mathcal{X}\times\mathcal{Y},\quad z_k=(x_k,y_k),\quad \Z_k=\X_k\times \Y_k
\end{equation}
The marginals of $\mu$ and $\nu$ on $\mathcal{X}_k$ and $\mathcal{Y}_k$ are denoted by $\mu_k$ and $\nu_k$.
The marginals of $\mu$ and $\nu$ on $\mathcal{X}_{i}\times \mathcal{X}_{j}$ and $\mathcal{Y}_{i}\times \mathcal{Y}_{j}$ are denoted by $\mu_{ij}$ and $\nu_{ij}.$

\myitem[D2]{itm:second}(measures, moments, densities)

Let $\mathcal{Z}$ be a Borel set in a Euclidean space. Denote by $\M(\mathcal{Z})$ the space of signed Borel measures on $\mathcal{Z}$ satisfying $\eta(\mathcal{Z})=1$ for any $\eta\in \M(\Z)$, and by $\P(\mathcal{Z})$ the subset of probability measures on $\mathcal{Z}$, i.e.,
\begin{equation}
\P(\mathcal{Z}) := \{\, \eta \in \M(\mathcal{Z}) : \eta \geq 0 \,\}.
\end{equation}
For $\eta \in \P(\mathcal{Z})$ and a measurable function $\Xi:\mathcal{Z}\to\mathbb{R}^{m\times n}$, we define the corresponding moment as
\begin{equation}\label{defiint}
\eta(\Xi) := \int_{\mathcal{Z}} \Xi(z)\,{\rm d}\eta(z).
\end{equation}
When $\Xi$ is vector- or matrix-valued, the integral in \eqref{defiint} is understood to be taken elementwise.

\myitem[D3]{itm:third}(graphs)

For a graph $G$, let $V(G)$ and $E(G)$ denote its vertex and edge sets, respectively. 
Throughout this paper, all graphs are assumed to be undirected and self-loop-free. {We also use integers to denote the vertices of a graph.} 
Whenever we write $ij \in E(G)$, we implicitly assume that $i < j$. 

We introduce a reference graph $\mathcal{G}$ with vertex set $V(\mathcal{G}) = [K]$, 
which serves as the underlying structure for the convex relaxation methods developed in this work. 
Without loss of generality, we further assume that $\mathcal{G}$ is connected, since otherwise our models can be decomposed into independent subproblems, each associated with a connected underlying structure.

Define $\S^d$ to be the set of $d$ by $d$ symmetric real matrices. {Define $\S^d_+$ and $\S^d_{++}$ to be the set of $d$ by $d$ symmetric positive semidefinite and positive definite matrices respectively. For any graph $G=([d],\E)$, define the sets}
\begin{equation}\label{defspset}
{\S_G} := \{X \in \S^d : X_{ij}=0 \ \text{for all } i \neq j \text{ with } ij \notin \E \},
\end{equation}
\begin{equation}\label{defspsetc}
{\S_G^\perp} := \{X \in \S^d : X_{ij}=0 \ \text{whenever } i=j \text{ or } ij \in \E \}.
\end{equation}
{Here $\S_G^\perp$ is the orthogonal complement of $\S_G$ in $\S^d$ with respect to the Frobenius inner product.}
See the following example:
\begin{equation}\label{illG}
G:\tikz[baseline=-0.5ex]{
  \node[circle,fill,inner sep=2pt] (a) {};
  \node[circle,fill,inner sep=2pt,right=0.5cm of a] (b) {};
  \node[circle,fill,inner sep=2pt,right=0.5cm of b] (c) {};
  \draw (a)--(b)--(c);
} \qquad {\S_G}:=\left\{\begin{bmatrix}
a&b&0\\
b&c&d\\
0&d&e
\end{bmatrix}\right\}, \quad {\S_G^\perp}:=\left\{\begin{bmatrix}
0&0&f\\
0&0&0\\
f&0&0
\end{bmatrix}\right\}
\end{equation}

We say that $G$ is a \emph{sparsity pattern} of $A$ if $A \in {\S_G}$. 
For any $A \in \S^d$, we write 
{$[A]_G := {\rm Proj}_{\S_G}(A)$ and 
$[A]_G^\perp := {\rm Proj}_{\S_G^\perp}(A)$, where the projections are taken with respect to the Frobenius inner product}.

For any $h\in \N,$ we define graph $G^h$ connecting nodes within graph distance $h$ in $G,$ that is
\begin{equation}\label{defipower}
\forall ij\in [d]_2,\ ij\in E(G^h),\ {\rm if \ and\ only\ if}\ {\rm dist}_G(i,j)\leq h.
\end{equation} 
We will use it to encode {\bf connectivity radius}. See the following example:
\begin{equation}\label{illpower}
G:\tikz[baseline=-0.5ex]{
  \node[circle,fill,inner sep=2pt] (a) {};
  \node[circle,fill,inner sep=2pt,right=0.5cm of a] (b) {};
  \node[circle,fill,inner sep=2pt,right=0.5cm of b] (c) {};
  \node[circle,fill,inner sep=2pt,right=0.5cm of c] (d) {};
  \node[circle,fill,inner sep=2pt,right=0.5cm of d] (e) {};
  \draw (a)--(b)--(c)--(d)--(e);
} \qquad 
G^2:\tikz[baseline=-0.5ex]{
  \node[circle,fill,inner sep=2pt] (a) {};
  \node[circle,fill,inner sep=2pt,right=0.5cm of a] (b) {};
  \node[circle,fill,inner sep=2pt,right=0.5cm of b] (c) {};
  \node[circle,fill,inner sep=2pt,right=0.5cm of c] (d) {};
  \node[circle,fill,inner sep=2pt,right=0.5cm of d] (e) {};
  \draw (a)--(b)--(c)--(d)--(e);
  \draw (a)to[out=45,in=135](c);
  \draw (c)to[out=45,in=135](e);
  \draw (b)to[out=-45,in=-135](d);
}
\end{equation}

\myitem[D4]{itm:fourth}(projection operators)

{Throughout this paper, we use two projection operators ${\rm P}$ and ${\rm R}.$ Let $\eta$ be a probability measure on variables indexed by a set $S,$ and let $U$ be another set of variables. We define ${\rm P}_U(\eta)$ to be the marginal of $\eta$ on the common variables $S\cap U,$ or equivalently the pushforward of $\eta$ under the coordinate projection onto $S\cap U.$ In particular, if $U\subset S,$ then ${\rm P}_U(\eta)$ is the usual marginal on $U.$ For example, if $(u_2,u_3)\sim \eta,$ then ${\rm P}_{(u_1,u_2)}(\eta)$ is the marginal of $\eta$ on $u_2.$}

{Next let $\Phi$ be a finite vector or matrix whose entries are functions of several variables. We define ${\rm R}_u(\Phi)$ to be the vector obtained by keeping exactly those entries of $\Phi$ that depend only on the variables in $u,$ with the inherited ordering. Equivalently, an entry $\phi$ is kept if there exists a function $\tilde\phi$ such that $\phi(z)=\tilde\phi(u).$ For polynomial bases, this means keeping the monomials whose variable support is contained in $u.$ For example,}
\begin{equation}\label{exP}
{
{\rm R}_x\!\left(\begin{bmatrix}xy & 1 \\ x^2 & y\end{bmatrix}\right)
=\begin{bmatrix}1 \\ x^2\end{bmatrix}, \qquad
{\rm R}_y\!\left(\begin{bmatrix}xy & 1 \\ x^2 & y\end{bmatrix}\right)
= \begin{bmatrix}1 \\ y\end{bmatrix}.
}
\end{equation}
\end{itemize}

\subsection{Marginal relaxation}

In this subsection, we present the marginal relaxation of \eqref{OT}, which approximates the high-dimensional transport plan in two steps, illustrated in Figures~\ref{fig:MF}–\ref{fig:HG}. Step 1 groups strongly correlated variables and constructs local couplings within each group. Step 2 introduces pairwise couplings across groups.  {Figure~\ref{fig:MF} depicts the first, mean-field step: instead of working with the full coupling of $(z_1,\ldots,z_K),$ we keep only the within-cluster couplings $\pi_k$ between $x_k$ and $y_k.$ Figure~\ref{fig:HG} depicts the second step, where pairwise couplings $\pi_{ij}$ are added along the reference graph to recover selected dependencies between clusters.}

\begin{figure}
\centering
\resizebox{0.75\linewidth}{!}{%
\begin{tikzpicture}[line cap=round,line join=round,
  dot/.style={circle,fill,inner sep=2.6pt},
  box/.style={draw,dashed,very thick}]

\newcommand{\panelCompact}[4]{%
  \draw[box] (-2.8,1.8) rectangle (2.0,-1.8);
  \node[dot,label=above:$#2$] (a) at (-0.5,1.0) {};
  \node[dot,label=below:$#3$] (b) at (-0.5,-0.5) {};
  \draw[thick] (a)--(b);
  \path (a) -- (b) node[midway,yshift=-1.6cm] {$#4$};
}

\begin{scope}[shift={(0,0)}]
  \panelCompact{}{x_1}{y_1}{z_1=(x_1,y_1)\sim\pi_1}
\end{scope}

\begin{scope}[shift={(6,0)}]
  \panelCompact{}{x_2}{y_2}{z_2=(x_2,y_2)\sim\pi_2}
\end{scope}

\begin{scope}[shift={(12,0)}]
  \panelCompact{}{x_3}{y_3}{z_3=(x_3,y_3)\sim\pi_3}
\end{scope}

\end{tikzpicture}
}
\caption{\small Step 1. Partition $x$ and $y$ into $K$ clusters $(x_1;x_2;\cdots;x_K)$ and $(y_1;y_2;\cdots;y_K)$. Construct a local coupling within each cluster: $z_k=(x_k,y_k)\sim\pi_k$. The mean-field approximation $\otimes_{k=1}^K \pi_k$ of $\pi$ is exact if the $z_k$'s are independent.}
\label{fig:MF}
\end{figure}

\begin{figure}
\centering
\resizebox{0.7\linewidth}{!}{%
\begin{tikzpicture}
[line cap=round,line join=round,
  dot/.style={circle,fill,inner sep=2.6pt},
  box/.style={draw,dashed,very thick}
]

\newcommand{\panel}[8]{%
  \draw[box] (-2,1.5) rectangle (2,-1.5);
  \node[dot,label=above:$#4$] (c) at (0,0.8) {};
  \node[dot,label=below:$#5$] (d) at ( 0,-0.4) {};
  \draw[thick] (c)--(d) node[midway,below] {$#8$};
  \path (c) -- (d) node[midway,yshift=-1.4cm] {$#6$};
}

\begin{scope}[shift={(0,5)}]
  \panel{}{}{}{x_1}{y_1}{z_1 \sim \pi_{1}}{}{}
\end{scope}

\begin{scope}[shift={(-5,0)}]
  \panel{}{}{}{x_2}{y_2}{z_2 \sim \pi_{2}}{}{}
\end{scope}

\begin{scope}[shift={(5,0)}]
  \panel{}{}{}{x_3}{y_3}{z_3 \sim \pi_{3}}{}{}
\end{scope}

\draw[dashed,thick] (-2,5) -- (-5,1.5) node[midway,sloped,above] {$(z_1,z_2) \sim \pi_{12}$};
\draw[dashed,thick] ( 2,5) -- ( 5,1.5) node[midway,sloped,above] {$(z_1,z_3) \sim \pi_{13}$};
\draw[dashed,thick] (-3,0) -- ( 3,0) node[midway,above] {$(z_2,z_3) \sim \pi_{23}$};
\end{tikzpicture}
}
\caption{\small Step 2. Add pairwise couplings between correlated clusters: $(z_i,z_j)\sim \pi_{ij},$ consistent with marginals ${\rm P}_{z_i}(\pi_{ij})=\pi_i,{\rm P}_{z_j}(\pi_{ij})=\pi_j$ and satisfying the PSD constraint \eqref{PSDcons}. In this example, the reference graph $\G$ (Definition~\ref{itm:third}) is a triangle. }
\label{fig:HG}
\end{figure}

The problem \eqref{OT} can be equivalently written as follows
\begin{align}\label{OT1}
&\inf_{\pi}\ \pi(c) \tag{OT} \\
{\rm s.t.}&\ {\rm P}_x(\pi)=\mu,\ {\rm P}_y(\pi)=\nu \tag{1a} \label{consrmarg} \\
&\ \pi\in \P(\Z). \tag{1b} \label{consrPg}
\end{align}
We will relax the above conditions \eqref{consrmarg} and \eqref{consrPg} as some conditions on the marginals of $\pi.$ 

\begin{itemize}
\item {\bf OT marginal constraints:} We relax the marginal constraint \eqref{consrmarg} as the following conditions on $\pi_{ij}$
\begin{equation}\label{marmar}
{\rm P}_x(\pi_{ij})=\mu_{ij},\qquad {\rm P}_y(\pi_{ij})=\nu_{ij}.
\end{equation}
We will later restrict the marginal constraints in a reference graph $\G$ defined in Definition~\ref{itm:third}.

\item {\bf Local positivity:} We relax the positivity of $\pi$ as the positivity of its marginals: 
\begin{equation}\label{marcons11}
\pi_{ij}\in \P(\Z_i\times \Z_j).
\end{equation}
Note that the positivity of $\pi_i$ and $\pi_j$ will be {implicitly} implied by \eqref{marcons11} and the consistency condition discussed later.

\item {\bf Global positivity:} We relax the nonnegativity constraint \eqref{consrPg} as the positive semidefiniteness (PSD) condition on the marginals $(\pi_k,\pi_{ij})_{[K],[K]_2}$. This means that for any family of square-integrable functions $(f_k)_{[K]} \in (L^2(\pi_k))_{[K]}$ the following inequality holds
\begin{equation}\label{PSDcons}
\sum_{k\in [K]}\pi_k\(f_k^2\)+
\sum_{ij\in [K]_2}2\pi_{ij}(f_if_j)= \pi\Big( \big( \sum_{k\in [K]} f_k(z_k) \big)^2 \Big)\geq 0,
\end{equation}
which comes from the linearity of integration as well as the fact that $(\pi_k,\pi_{ij})_{[K],[K]_2}$ are 1 and 2 marginals of the probability measure $\pi.$ We use ``$(\pi_k,\pi_{ij})_{[K],[K]_2} \succeq 0$" to denote the condition \eqref{PSDcons}. When the sets $\mathcal{Z}_k$ are finite, condition \eqref{PSDcons} is equivalent to requiring the block matrix
\begin{equation}\label{matPSDm}
\begin{bmatrix}
\DD(\pi_1)      & \pi_{12}   & \cdots & \pi_{1K} \\
\pi_{12}^\top & \DD(\pi_2)  & \ddots & \pi_{2K} \\
\vdots   & \ddots  & \ddots & \vdots \\
\pi_{1K}^\top & \pi_{2K}^\top & \cdots & \DD(\pi_K)
\end{bmatrix} \succeq 0,
\end{equation}
where $\DD(\pi_k)$ denotes a diagonal matrix with diagonal entries being $\pi_k$. The PSD condition is often used to strengthen marginal relaxations, with applications in density functional theory 
\cite{chen2025convex,chen2024multiscale,khoo2019convex,peng2012approximate}. 

\item {\bf Consistency:} We relax the condition that $\pi_k,\pi_{ij}$ are one- and two-marginals of $\pi$ into some local consistency conditions. For every $ij \in [K]_2$, the marginals of $\pi_{ij}$ agree with those of $\pi_i$ and $\pi_j.$ The conditions are summarized as follows:
\begin{equation}\label{marcons}
{\rm P}_{z_i}(\pi_{ij}) = \pi_i, 
\qquad 
{\rm P}_{z_j}(\pi_{ij}) = \pi_j,
\end{equation}
where ${\rm P}_{z_i},{\rm P}_{z_j}$ are defined in Definition~\ref{itm:fourth}.

\end{itemize}
With the above conditions on the one- and two-marginals $\pi_k,\pi_{ij}$ of $\pi$ in \eqref{OT1}, we are now able to present the marginal relaxation of \eqref{OT1}.  Suppose $c(z)$ has the following decomposition
\begin{equation}\label{cdec}
c(z)=\sum_{k\in [K]} c_k(z_k).
\end{equation}
{This decomposition holds for the standard squared Euclidean cost $c(x,y)=\|x-y\|^2,$ by taking $c_k(z_k)=\|x_k-y_k\|^2.$ More generally, it holds for additive coordinate costs $c(x,y)=\sum_i \ell_i(x_i,y_i),$ including $\ell_p$ losses $c(x,y)=\sum_i |x_i-y_i|^p.$}
The marginal relaxation of \eqref{OT1} is as follows:
\begin{align}\label{OTmr}
\inf_{(\pi_k,\pi_{ij})_{[K],[K]_2}} &\ \sum_{k\in [K]} \pi_k(c_k) 
\tag{${\rm OT}_{\rm mar}$} \\
{\rm s.t.}\quad & {\rm P}_x(\pi_{ij})=\mu_{ij},\quad {\rm P}_y(\pi_{ij})=\nu_{ij},\quad \forall ij \in E(\G) \label{cons1} \tag{2a} \\
& {\rm P}_{z_i}(\pi_{ij}) = \pi_i,\quad  
   {\rm P}_{z_j}(\pi_{ij}) = \pi_j, \quad \forall ij \in [K]_2 \label{cons2} \tag{2.1b} \\
& \pi_{ij}\in \P(\Z_i\times \Z_j),\quad \forall ij\in [K]_2 \label{cons3} \tag{2.2b} \\
& (\pi_k,\pi_{ij})_{[K],[K]_2} \succeq 0, \label{cons4} \tag{2.3b}
\end{align}
where the constraint \eqref{cons1} is a relaxation of \eqref{consrmarg} and \eqref{cons2}--\eqref{cons4} are relaxations of \eqref{consrPg}. Note that the condition $\pi_k\in \P(\Z_k)$ is implied by \eqref{cons2} so we omit it to avoid redundancy. The reference graph $\G,$ first mentioned in Definition~\ref{itm:third} (it is a triangle in Figure~\ref{fig:HG}), controls which pairwise marginals are retained. Although the complete graph ($E(\G)=[K]_2$) gives the tightest relaxation, sparser graphs can reduce computational complexity for solving the SDP problem, which will be elaborated later.

Compared to \eqref{OT}, whose decision variable is a full $d$-dimensional measure, this formulation works only with low-dimensional marginals, drastically reducing complexity. For example, when $K=d$ and $\X_k=\Y_k=[r]$, \eqref{OT} involves $r^{2d}$ variables, while \eqref{OTmr} involves only $r^4 d(d-1)/2+r^2d$.

From \eqref{cons3} and \eqref{cons4}, the matrix variable in \eqref{matPSDm} (in the finite setting) is both PSD and nonnegative. In this case, the problem \eqref{OTmr} is a doubly-nonnegative (DNN) programming problem. Although it is computationally tractable with several solvers \cite{hou2025low,sun2020sdpnal+,hou2025rinnal+} available, high-dimensional DNN problems are still challenging due to its large number of variables and constraints. To reduce the dimensionality of \eqref{OTmr}, we further relax its constraints in the following two ways:

The first way is to drop the PSD condition \eqref{cons4} of \eqref{OTmr}, obtaining the following problem:
\begin{align}\label{OTmr-1}
\inf_{(\pi_k,\pi_{ij})_{[K],E(\G)}} &\ \sum_{k\in [K]} \pi_k(c_k) 
\tag{${\rm OT}_{\rm mar}^1$} \\
{\rm s.t.}\quad & {\rm P}_x(\pi_{ij})=\mu_{ij},\quad {\rm P}_y(\pi_{ij})=\nu_{ij},\quad \forall ij \in E(\G) \label{cons5} \tag{3a} \\
& {\rm P}_{z_i}(\pi_{ij}) = \pi_i,\quad  
   {\rm P}_{z_j}(\pi_{ij}) = \pi_j,\quad \forall ij \in E(\G) \label{cons6} \tag{3.1b} \\
   & \pi_{ij}\in \P(\Z_i\times \Z_j),\quad \forall ij\in E(\G). \label{cons7} \tag{3.2b} 
\end{align}
In \eqref{OTmr-1}, we only include $\pi_{ij}$ on the edge set of $E(\G).$ This is because for any $ij\in [K]_2\setminus E(\G),$ we may define $\pi_{ij}$ to be the product measure $\pi_i\otimes \pi_j$ and it is easy to check that $(\pi_k,\pi_{ij})_{[K],[K]_2}$ satisfy all the constraints in \eqref{OTmr} except the PSD constraint \eqref{cons4}. Therefore, by dropping the PSD, we also dropped a large number of unrelated decision variables if $\G$ is sparse.  In addition, the problem \eqref{OTmr-1} is a linear programming (LP) problem, which is much easier than SDP. 

The second way is to remove the nonnegativity constraint \eqref{cons3} from \eqref{OTmr}, obtaining the following problem:
\begin{align}\label{OTmr-2}
\inf_{(\pi_k,\pi_{ij})_{[K],[K]_2}} &\ \sum_{k\in [K]} \pi_k(c_k) 
\tag{${\rm OT}_{\rm mar}^2$} \\
{\rm s.t.}\quad & {\rm P}_x(\pi_{ij})=\mu_{ij},\quad {\rm P}_y(\pi_{ij})=\nu_{ij},\quad \forall ij \in E(\G) \label{cons8} \tag{4a} \\
& {\rm P}_{z_i}(\pi_{ij}) = \pi_i,\quad  
   {\rm P}_{z_j}(\pi_{ij}) = \pi_j,\quad \forall ij \in [K]_2 \label{cons9} \tag{4.1b} \\
      & \pi_{ij}\in \M(\Z_i\times \Z_j),\quad \forall ij\in [K]_2 \label{cons10} \tag{4.2b}  \\
   & (\pi_k,\pi_{ij})_{[K],[K]_2} \succeq 0. \label{cons11} \tag{4.3b} 
\end{align}

\begin{rem}\label{remmr}
The problem \eqref{OTmr-2} is an SDP without nonnegativity constraints. Its dimension can be further reduced via the \emph{chordal conversion} method in \cite{tang2024exploring}.  
For example, consider the setting $K=d$ and $\X_k=\Y_k=[r].$ Then $\pi_k\in \R^{r^2}$ and $\pi_{ij}\in \R^{r^2\times r^2}.$ The constraint \eqref{cons9} becomes
\begin{equation}\label{disOTcons}
\pi_{ij}{\bf 1}_{r^2}=\pi_i,\qquad \pi_{ij}^\top{\bf 1}_{r^2}=\pi_j.
\end{equation}
These conditions can be compactly written as
\begin{equation}\label{affconsf}
X(I_d\otimes {\bf 1}_{r^2}-J_d\otimes {\bf 1}_{r^2})=0,
\end{equation}
where $X$ is the matrix in \eqref{matPSDm}, $I_d$ is the identity matrix, and $J_d$ is the cyclic shift matrix with $J_{2,1}=J_{3,2}=\cdots=J_{d,d-1}=J_{1,d}=1$ and all other entries zero.  
Since $X\succeq 0,$ \eqref{affconsf} is equivalent to
\begin{equation}\label{affconsf1}
\<X,HH^\top\>=0,
\end{equation}
with $H=(I_d\otimes {\bf 1}_{r^2}-J_d\otimes {\bf 1}_{r^2}).$ Thus, \eqref{affconsf} reduces to a single affine constraint on $X.$ Noting that $\rr(HH^\top)=\rr(H)=d-1,$ this constraint involves a low-rank coefficient matrix.  

The remaining linear constraints in \eqref{OTmr-2} and the linear objective involve only the diagonal blocks $\DD(\pi_k)$ or the off-diagonal blocks $\pi_{ij},\pi_{ij}^\top$ for $ij\in E(\G).$ This yields a \emph{sparse plus low-rank} structure \cite[Definition~1.1]{tang2024exploring}, which enables significant dimensionality reduction. For instance, if $\G$ is a tree, \cite[Theorem~1.4]{tang2024exploring} shows that \eqref{OTmr-2} decomposes into a multi-block SDP with block size at most $r^2+2d,$ far smaller than the ambient dimension $dr^2.$ For general $\G,$ the block size is related to the \emph{tree-width} of $\G$. Please see \cite[Section 2]{tang2024exploring} for more details about tree-width.
\end{rem}

While the marginal relaxation already reduces dimensionality significantly, it still works directly with measures, which becomes costly in continuous domains. To address this, we next introduce the cluster moment relaxation, which works with moments instead of full measures.

\subsection{Cluster moment relaxation}\label{Sec:mom}

In this subsection, we introduce the \emph{cluster moment relaxation}, 
which applies moment relaxation to the marginal relaxations \eqref{OTmr}, \eqref{OTmr-1} and \eqref{OTmr-2} of \eqref{OT}. 
Following the methodology of \cite{chen2025convex}, we begin by defining the cluster basis. 
{The idea is to keep the same local structure as the marginal relaxation, but to replace each local measure variable by finitely many moments. The choice of which moments to keep is determined by the cluster basis defined below.}

Fix a relaxation degree $n \in \N^+,$ specifying the maximum polynomial degree retained.  
Let $\{\phi_j : \R \to \R\}_{j=0}^n$ be basis functions with $\phi_0 \equiv 1$ (e.g., $\phi_j(s)=s^j$).  
For any multi-index $\alpha=(\alpha_1,\ldots,\alpha_{2K}) \in \N^{2d},$ define
\begin{equation}\label{defPhi}
\phi_{\alpha} := \prod_{k \in [K]} \phi_{\alpha_k}(x_k)\,\phi_{\alpha_{k+K}}(y_k).
\end{equation}
If $x_k$ (or $y_k$) contains multiple coordinates, then $\alpha_k$ is itself a multi-index, and  
$\phi_{\alpha_k}(x_k)$ denotes the product $\prod_i \phi_{\alpha_k[i]}(x_k[i]),$  
where $\alpha_k[i]$ and $x_k[i]$ are the respective coordinates. We define the \emph{cluster basis} for $z_k,$ as:
\begin{equation}
\Phi_k := \left\{ \phi_{\alpha} : 
  \  \phi_{\alpha}\ {\rm only\ has\ variables\ in\  }z_k,\ 
  |\alpha| \leq n \right\}, \label{defFk}\\
\end{equation}
which can be viewed as vectors of basis functions supported on the variables $z_k,$ up to degree $n$. If $c_k$ lies in the span of $\Phi_k \Phi_k^\top$, it can be written as 
\begin{equation}\label{eqck}
c_k(z_k) = \langle C_k, \Phi_k \Phi_k^\top \rangle
\end{equation}
for some symmetric matrix $C_k$. 
{For the monomial basis $\phi_j(s)=s^j,$ this condition is satisfied by any polynomial local cost $c_k$ of degree at most $2n$ in the variables $z_k.$ This includes the usual squared Euclidean cost  and $\ell_p$ losses with even $p\leq 2n.$ }
The objective of the relaxation then becomes
\begin{equation}\label{objmomk}
\sum_{k \in [K]} \langle C_k, M_k \rangle,\quad {\rm where}\quad M_k := \pi_k(\Phi_k \Phi_k^\top), \quad M_{ij} := \pi_{ij}(\Phi_i \Phi_j^\top).
\end{equation}
Here $M_k$ is a Gram-type moment matrix and therefore always symmetric 
positive semidefinite, while $M_{ij}$ collects cross-moments between cluster bases 
and need not be symmetric or positive semidefinite.
{Thus $M_k$ and $M_{ij}$ are finite statistics of the unknown couplings $\pi_k$ and $\pi_{ij}$ instead of the couplings themselves. Increasing the degree $n$ enriches the relaxation; for small $n,$ the program is cheaper but only enforces the OT constraints through low-order statistics.}

We next describe the constraints that these moment matrices must satisfy. They mirror the marginal/local and global positivity/consistency constraints from \eqref{OTmr}. {The prescribed moment constraints encode the source and target marginals, the local and global positivity constraints encodes the nonnegativity of local and global coupling measures at the level of the chosen basis.}

\begin{itemize}

\item {\bf OT marginal constraints:}
We relax the marginal constraints in \eqref{marmar} using moments of the $x$ and $y$ marginal. In detail, let 
\begin{equation}\label{objmomkn}
M_k^x := \pi_k\({\rm R}_x\(\Phi_k \Phi_k^\top\)\), \qquad M_k^y := \pi_k\({\rm R}_y\(\Phi_k \Phi_k^\top\)\).
\end{equation}
\begin{equation}\label{objmomijn}
M_{ij}^x := \pi_{ij}\({\rm R}_x\(\Phi_i \Phi_j^\top\)\), \qquad M_{ij}^y := \pi_{ij}\({\rm R}_y\(\Phi_i \Phi_j^\top\)\).
\end{equation}
We have that
\begin{equation}\label{consmomk}
M_k^x =  \mu\({\rm R}_x\(\Phi_{k}\Phi_{k}^\top\)\),\qquad M_k^y =  \nu\({\rm R}_y\(\Phi_{k}\Phi_{k}^\top\)\),
\end{equation}
\begin{equation}\label{consmomij}
M_{ij}^x =  \mu\({\rm R}_x\(\Phi_{i}\Phi_{j}^\top\)\),\qquad M_{ij}^y =  \nu\({\rm R}_y\(\Phi_{i}\Phi_{j}^\top\)\),
\end{equation}
which means that some entries of the moment matrices $M_k,M_{ij}$ are prescribed as the moments of $\mu$ and $\nu.$

\item {\bf Local positivity:}  The local positivity condition \eqref{marcons11} in marginal relaxation is relaxed as the following condition on the moment matrices
\begin{equation}\label{locpos}
(M_i,M_j,M_{ij}) \in \C_{ij},
\end{equation}
where $\C_{ij}$ is a convex set encoding necessary conditions for $(M_i,M_j,M_{ij})$ to represent the moments of some probability measure $\pi_{ij}$. The conditions can be summarized as that: $(M_i ,M_j, M_{ij})$ is embedded in a PSD matrix defined by sum-of-squares (SOS) hierarchy $\pi_{ij}({\rm SOS\ for}\ (z_i,z_j))\geq0.$ In the marginal, this is enforced by point-wise nonnegativity.

\item {\bf Global positivity:} 
We relax the global positivity condition \eqref{PSDcons} in \eqref{OTmr} by restricting the test functions $f_k$ to the form $f_k = v_k^\top \Phi_k$ for some vector $v_k$. Substituting into \eqref{PSDcons} gives
\begin{align}\label{gPSDp}
0 &\leq \sum_{k \in [K]} \pi_k\!\left((v_k^\top \Phi_k)^2\right) 
   + \sum_{ij \in [K]_2} 2 \pi_{ij}\!\left((v_i^\top \Phi_i)(v_j^\top \Phi_j)\right) \notag \\
  &= \sum_{k \in [K]} v_k^\top M_k v_k 
   + \sum_{ij \in [K]_2} 2 v_i^\top M_{ij} v_j.
\end{align}
Since this inequality must hold for all choices of $(v_k)_{k \in [K]}$, it is equivalent to requiring the block moment matrix
\begin{equation}\label{matPSD}
M :=
\begin{bmatrix}
M_1      & M_{12}   & \cdots & M_{1K} \\
M_{12}^\top & M_2  & \ddots & M_{2K} \\
\vdots   & \ddots  & \ddots & \vdots \\
M_{1K}^\top & M_{2K}^\top & \cdots & M_K
\end{bmatrix} \succeq 0.
\end{equation}

\item {\bf Consistency:} {Let $\Phi:=[\Phi_1;\Phi_2;\ldots;\Phi_K].$ Then the matrix $M$ represents the moment of the Gram matrix $\Phi\Phi^\top.$} If $\Phi \Phi^\top$ has repeated monomials at different positions, the corresponding entries of $M$ must coincide. We say ``$M$ is consistent”. 

\end{itemize}

With these conditions, the \emph{cluster moment relaxation} of \eqref{OT} can be formulated as
\begin{align}\label{OTmom}
\inf_{M}\quad &\sum_{k \in [K]} \langle C_k, M_k \rangle 
\tag{${\rm OT}_{\rm mom}$} \\
{\rm s.t.}\quad 
& M_k^x =  \mu\({\rm R}_x\(\Phi_{k}\Phi_{k}^\top\)\),\ M_k^y =  \nu\({\rm R}_y\(\Phi_{k}\Phi_{k}^\top\)\),\ \forall k\in [K] \label{cmomk} \tag{5.1a} \\
& M_{ij}^x = \mu\({\rm R}_x\(\Phi_i \Phi_j^\top\)\),\ M_{ij}^y = \nu\({\rm R}_y\(\Phi_i \Phi_j^\top\)\),\ \forall ij\in E(\G) \label{cmomij} \tag{5.2a} \\
& (M_i,M_j,M_{ij})\in \C_{ij}, \ \forall ij \in [K]_2,\ M \succeq 0 \ \text{and consistent}, \label{gpsd} \tag{5b}
\end{align}
where $M$ is the block moment matrix defined in \eqref{matPSD} and the conditions \eqref{cmomk}--\eqref{cmomij} refer to the OT marginal constraints \eqref{consmomk}, \eqref{consmomij}, which stem from the OT marginal constraints \eqref{consrmarg} of \eqref{OT1}.  
Constraints \eqref{gpsd} relax \eqref{cons2}–\eqref{cons4} in \eqref{OTmr}, reflecting the representability of $\pi_k$ and $\pi_{ij}$ as one- and two-marginals of the probability measure $\pi$ in \eqref{consrPg}.

\begin{rem}\label{remsupp}
{Consider the optimal transport in continuous setting. Suppose, for example, that the local supports are semialgebraic sets}
\begin{equation}
{
\X_{k}=\{x_{k}:\ a^x_{k,\ell}(x_{k})\geq 0,\ \ell\in [L_k^x]\},\qquad 
\Y_{k}=\{y_{k}:\ a^y_{k,\ell}(y_{k})\geq 0,\ \ell\in [L_k^y]\}.
}
\end{equation}
{Here the polynomials $a^x_{k,\ell}$ and $a^y_{k,\ell}$ describe the support constraints. In the moment SOS relaxation, these constraints are incorporated by localizing matrices \cite[Sec.~3.2]{mula2024moment} \cite[Sec.~3.2.1]{lasserre2009moments}. For instance, the condition that a measure is supported on $a^x_{k,\ell}\geq0$ is relaxed by requiring}
\begin{equation}
{
\pi_k(q^2 a^x_{k,\ell})\geq 0
}
\notag
\end{equation}
{for all test polynomials $q$ of the allowed degree, which gives a linear matrix inequality in the moment of $\pi_k$. This remark only concerns the cluster moment relaxation for the continuous optimal transport. In the finite-state or discretized setting used in our marginal relaxation, the support is already encoded by the chosen grid or finite state space, so these additional localizing matrix constraints are not needed.}
\end{rem}

Problem \eqref{OTmom} is a finite-dimensional SDP solvable by standard solvers. However, the presence of multiple PSD blocks in $\C_{ij}$ together with the large PSD variable $M$ leads to complicated implementation and high computational cost. Similar to \eqref{OTmr-1} and \eqref{OTmr-2}, we therefore consider further relaxations of \eqref{OTmom} by dropping either the global or local positivity conditions.

\begin{align}\label{OTmom-1}
\inf_{(M_k,M_{ij})_{[K],E(\G)}}\ &\sum_{k \in [K]} \langle C_k, M_k \rangle 
\tag{${\rm OT}_{\rm mom}^1$} \\
{\rm s.t.}\quad 
& M_k^x =  \mu\({\rm R}_x\(\Phi_{k}\Phi_{k}^\top\)\),\ M_k^y =  \nu\({\rm R}_y\(\Phi_{k}\Phi_{k}^\top\)\),\ \forall k\in [K] \label{cmomspk} \tag{6.1a} \\
&M_{ij}^x = \mu\({\rm R}_x\(\Phi_i \Phi_j^\top\)\),\ M_{ij}^y = \nu\({\rm R}_y\(\Phi_i \Phi_j^\top\)\),\ \forall ij\in E(\G) \label{cmomspij} \tag{6.2a} \\
& (M_i,M_j,M_{ij})\in \C_{ij}\ \text{and consistent},\ \forall ij \in E(\G). \label{spconsA} \tag{6b}
\end{align}

Problem \eqref{OTmom-1} is obtained from \eqref{OTmom} by dropping the global positivity constraint and removing the variables $M_{ij}$ for $ij\notin E(\G).$  
It is easier to solve than \eqref{OTmom}, as it involves fewer variables and constraints.

\begin{align}\label{OTmom-2}
\inf_{M}\quad & \sum_{k \in [K]} \langle C_k, M_k \rangle 
  \tag{${\rm OT}_{\rm mom}^2$} \\
  {\rm s.t.}\quad 
& M_k^x =  \mu\({\rm R}_x\(\Phi_{k}\Phi_{k}^\top\)\),\ M_k^y =  \nu\({\rm R}_y\(\Phi_{k}\Phi_{k}^\top\)\),\ \forall k\in [K] \label{cmomsimk} \tag{7.1a} \\
& M_{ij}^x = \mu\({\rm R}_x\(\Phi_i \Phi_j^\top\)\),\ M_{ij}^y = \nu\({\rm R}_y\(\Phi_i \Phi_j^\top\)\),\ \forall ij\in E(\G)\label{cmomsimij} \tag{7.2a} \\
& M \succeq 0 \ \text{and consistent}. \label{conBsim} \tag{7b}
\end{align}
Problem \eqref{OTmom-2} is obtained from \eqref{OTmom} by dropping the local positivity constraints.  

\begin{rem}\label{remmom}
{The basis vector $\Phi$ may contain redundant entries. For example, when using a monomial basis, the constant term $1$ appears in each cluster basis $\Phi_k$, and hence occurs multiple times in $\Phi$. In practice, we remove such redundant entries from $\Phi$ beforehand.}

From \eqref{objmomk} and \eqref{matPSD}, the consistency condition \eqref{conBsim} affects only the diagonal blocks $M_k$'s ({after removing redundant rows and columns}).  
Thus, all affine constraints and the objective function involve only the blocks $(M_k,M_{ij})_{[K],E(\G)}.$  
Consequently, as in Remark~\ref{remmr}, the chordal conversion method of \cite{tang2024exploring} can be applied to exploit the chordal sparsity of $\G$ and decompose \eqref{OTmom-2} into a multi-block SDP with small matrix variables. {In addition, when the tree-width of $\G,$ the cluster size and relaxation degree are constant-size, fast interior point solver \cite{gu2022faster} can solve the multi-block problem in linear-time complexity $\O(d).$}
\end{rem}

{
\subsection{Comparison with existing sum-of-squares methods}\label{subseccomp}

Our cluster moment relaxation is related to the moment-SOS hierarchy of
Mula and Nouy~\cite{mula2024moment} and to the RKHS-based SOS relaxation of
Vacher, Muzellec, Rudi, Bach, and Vialard~\cite{vacher2021dimension}. We compare
these approaches from the viewpoint of sparsity, statistical error, and SDP size.

{\bf Comparison with \cite{mula2024moment}.} Mula and Nouy~\cite{mula2024moment} proposed a moment-SOS hierarchy for OT based on all monomials up to a prescribed degree $n$. While this yields a systematic sequence of relaxations, the number of moments grows as $\O(d^{2n})$, leading to substantial sample and computational complexity in high dimensions. In contrast, our cluster moment relaxation exploits the underlying local structure and requires only a small collection of moments, resulting in significantly reduced sample and computational costs.

Our cluster moment relaxation reduces the sample error by keeping the same moment-SOS principle but restricting the basis to local clusters and edges in a reference graph $\G$. If $r$ denotes the maximum cluster size, then the number of prescribed moments is $\O(r^{2n}(K+|E(\G)|)),$ which is much better than $\O(d^{2n})$ when the cluster size $r$ is small. Thus the statistical error is smaller since the number of moments we need to estimate is smaller. This is quantified in Proposition~\ref{prop:samplecomplex} and Theorem~\ref{thm:sdpvalue}.

Our method also reduces the computational cost. The SDP problem in \cite{mula2024moment} has matrix dimension $\O(d^n)$ and number of linear constraints $\O(d^{2n})$. On the other hand, the cluster moment SDP only has matrix dimension $\O(dr^n)$ and number of constraints $\O(r^{2n}(K+|E(\G)|))$. Moreover, as discussed in Remark~\ref{remmom}, when the tree-width of $\G,$ the relaxation degree $n$ and the cluster size are fixed as constants, the SDP problem can be solved in linear time $\O(d).$ This scalability is also visible
in the second panels of Figures~\ref{Fig:GGMS}--\ref{Fig:block-ising}.

\vspace{0.5em}

{\bf Comparison with \cite{vacher2021dimension}.}
Vacher et al.~\cite{vacher2021dimension} develop an SOS relaxation in a
reproducing kernel Hilbert space and provide guarantees under smoothness assumptions
on the source and target distributions; see \cite[Assumption~1]{vacher2021dimension}. While sample complexity for achieving $\epsilon$ accuracy is $\O(\epsilon^{-2}),$ the hidden constant can depend exponentially on $d$ as stated in \cite[Theorem~2]{vacher2021dimension}.

Our approach improves sample complexity by restricting the relaxation to a small collection of local moments. For sparse Gaussian models, Theorem~\ref{thm:gauss-sample} yields a sample complexity of $\O(d\epsilon^{-2}+|E(\G)|\epsilon^{-1})$, which is linear in $d$ when $\G$ is sparse. More generally, although convergence rates for the moment-SOS hierarchy beyond the Gaussian setting remain unknown, Theorem~\ref{thm:sdpvalue} shows that the sample complexity of the cluster moment relaxation is governed by the sparsity of the reference graph $\G$ rather than the ambient dimension. In particular, when the relaxation degree and cluster size are fixed independently of $d$, the sample complexity is of order $\O((K+|E(\G)|)^2\epsilon^{-2})$, exhibiting only polynomial dependence on the dimension.

Our approach also reduces the computational cost. The SDP in~\cite{vacher2021dimension} has dimension proportional to the sample size $N$, leading to an interior-point complexity of $\O(N^{3.5})$. In contrast, once the empirical moments have been estimated, the size of the cluster moment SDP depends only on the number of local basis functions \eqref{matPSD} and is independent of $N$. For fixed-degree relaxations with bounded cluster size, the SDP size scales linearly with $d$. Moreover, under a bounded tree-width assumption, Remark~\ref{remmom} yields linear-time complexity in $d$. Since evaluating a degree-$n$ monomial requires only $\O(n)$ arithmetic operations, the cost of estimating all prescribed moments is $\O((K+|E(\G)|)N)$ when the degree and cluster size are fixed. Thus, the moment estimation cost is linear in the sample size $N$.}

\section{Theoretical analysis}\label{Sec:analysis}

{In this section, we provide theoretical analysis of the convex relaxation methods proposed in the previous section. We focus on the squared cost setting
\begin{equation}
c(x,y)=\|x-y\|^2,\qquad \X=\Y=\R^d.
\end{equation}
We assume that all measures considered in this section are absolutely continuous with respect to Lebesgue measure. For a measure $\eta$, we denote its density by $\rho_\eta$, i.e.,
$$
\eta({\rm d}z) = \rho_\eta(z)\,{\rm d}z.
$$
The dual problem of \eqref{OT} is:

\begin{equation}\label{OTd0}
\sup_{f,g} \Bigg\{ 
   \int_{\R^d} f(x)\,{\rm d}\mu(x) 
 + \int_{\R^d} g(y)\,{\rm d}\nu(y) 
 : \ \|x-y\|^2 - f(x) - g(y) \geq 0 
 \Bigg\}.
\end{equation}

Our convex relaxation methods introduce three types of errors. To organize them,
fix a partition $\I=\{I_a\subset [d]:\ a\in [K]\}$ (see Subsubsection~\ref{subsubsec:mf-dual}), a reference graph $\G$, and a relaxation degree $n$. We use
the following notation:
\begin{itemize}
\item $\operatorname{MR}_{\I,\G}$ is the value of the chosen marginal relaxation among
\eqref{OTmr}--\eqref{OTmr-2};
\item $\operatorname{OPT}_{n}$ is the value of the corresponding degree-$n$ moment SDP
with ground truth moments;
\item $\widehat{\operatorname{OPT}}_{n}$ is the value of the cluster moment SDP whose
moments are estimated from samples.
\end{itemize}
With this notation, the total error can be decomposed as
\begin{equation}
\big|\widehat{\operatorname{OPT}}_{n}-W_2^2(\mu,\nu)\big|
\le
E_{\rm app}+E_{\rm trunc}+E_{\rm stat},
\label{eq:three-errors}
\end{equation}
where the errors are defined as
\begin{itemize}
\item $E_{\rm app}:=|W_2^2(\mu,\nu)-\operatorname{MR}_{\I,\G}|$ is the \emph{approximation error}, namely the difference between the original OT value and the value of the underlying marginal relaxations \eqref{OTmr}--\eqref{OTmr-2}.
\item $E_{\rm trunc}:=|\operatorname{MR}_{\I,\G}-\operatorname{OPT}_{n}|$ is the \emph{truncation error} of the moment SOS hierarchy applied to the marginal relaxations.
\item $E_{\rm stat}:=|\operatorname{OPT}_{n}-\widehat{\operatorname{OPT}}_{n}|$ is the \emph{statistical error} from estimating the prescribed moments of $\mu$ and $\nu$ using samples.
\end{itemize}

In Subsection~\ref{Sec:ThmGauss}, we analyse the above errors for Gaussian measures. Theorem~\ref{thmexp} shows that for degree-1 moment relaxation, the error $E_{\rm app}+E_{\rm trunc}$ goes to zero exponentially fast as the connectivity radius $h$ in the reference graph $\G=G^h$ (see \eqref{defipower} for definition of $G^h.$) increases, where $G$ is the sparsity pattern of the Gaussian precision matrix. This implies that our method achieves good accuracy while preserving the sparse structure. In addition, Theorem~\ref{thm:gauss-sample} establishes the statistical error bound $\O(\sqrt{d/N}+|E(\G)|/N)$, where the constant in the second term depends on the density of the reference graph $\G$. Consequently, compared with the full relaxation, for which $\G$ is the complete graph and the statistical error is $\O(\sqrt{d/N}+d^2/N)$, a sparse graph with $|E(\G)|=\O(d)$ yields the improved bound $\O(\sqrt{d/N}+d/N)$. This demonstrates that, even in the Gaussian setting, exploiting sparsity can substantially improve sample complexity.

In Subsection~\ref{subsec:mf}, we move beyond the Gaussian case and study local perturbations of mean-field product measures. Subsubsection~\ref{subsubsec:mf-results} states the main approximation result, Theorem~\ref{thm:mfloc}, which bounds $E_{\rm app}$ for the marginal relaxation \eqref{OTmr-1} and shows that it captures the local linearization of the quadratic Wasserstein distance~\cite{peyre2018comparison}, also known as the weighted negative homogeneous Sobolev norm~\cite[equation (2.5)]{peyre2018comparison}. For this non-Gaussian setting, we do not analyze $E_{\rm trunc}$, since obtaining a quantitative convergence rate for the moment-SOS hierarchy approximating the continuous marginal relaxation is substantially more involved. Nevertheless, under fixed relaxation degree and bounded relaxation cluster size, Corollary~\ref{coro:mf-stat} in Subsubsection~\ref{subsubsec:mf-results} gives a statistical error bound of order $\O(dN^{-1/2})$, corresponding to a sample complexity of $\O(d^2\epsilon^{-2})$. 

In Subsection~\ref{subsec:samplecomplex}, we analyse the statistical error $E_{\rm stat}$ for general cluster moment relaxation. Our main result Theorem~\ref{thm:sdpvalue} provides the error bound of order $\O((K+|E(\G)|)N^{-1/2})$ in the setting where both relaxation degree and cluster size are fixed as constants independent of the dimension $d.$ This results shows that our method has smaller statistical error for sparse graph $\G.$ }

\subsection{Gaussian distribution}\label{Sec:ThmGauss}

{In this section, we analyze the approximation quality of the cluster moment relaxation in the case where both marginals are Gaussian. Theorem~\ref{thmexp} provides a bound on the combined error $E_{\rm app}+E_{\rm trunc}$ for the degree-$1$ cluster moment relaxation. It shows that the error decays exponentially as the connectivity radius $h$ in the reference graph $\G=G^h$ (see \eqref{defipower}) increases, where $G$ is the sparsity pattern of the Gaussian precision matrix. Thus, the cluster moment SDP can achieve high accuracy while preserving a good level of sparsity. Theorem~\ref{thm:gauss-sample} further establishes a statistical error of order $\O(\sqrt{d/N}+(d+|E(G^h)|)/N)$. It shows that the statistical error for a sparse reference graph $\G=G^h$ is $\O(\sqrt{d/N}+d/N),$ which is better than that of the full degree-1 moment relaxation $\O(\sqrt{d/N}+d^2/N).$ Thus, by utilizing the sparsity, our cluster moment relaxation improves the OT sample complexity even in the well-studied Gaussian setting.}

Let $\mu=\mathcal{N}(m_1,\Sigma_1),\nu=\mathcal{N}(m_2,\Sigma_2)$ with means $m_1,m_2 \in \R^d$ and {covariances} $\Sigma_1,\Sigma_2 \in \S^d_{++}$. For the quadratic cost $c(x,y)=\|x-y\|^2$, problem \eqref{OT} reduces to the squared
Wasserstein-2 distance, which admits the closed form
\begin{equation}\label{W2GGM}
W_2^2(\mu,\nu)
= \|m_1-m_2\|^2 
+ \Tr\!\Big(\Sigma_1+\Sigma_2 - 2(\Sigma_1^{1/2}\Sigma_2\Sigma_1^{1/2})^{1/2}\Big).
\end{equation}
Since Gaussians are fully characterized by first and second moments, it suffices to take clusters of single variables and a degree-1 monomial basis $\{1,s\}.$ With these choices, the relaxation \eqref{OTmom-2} reduces to the SDP
\begin{align}\label{GSmom}
\min_X\ &\ \Tr(Z_1)+\Tr(Z_2)-2\Tr(Y) \\
{\rm s.t.}\ &\ {[Z_1]_{\G}}={[\Sigma_1+m_1m_1^\top]_{\G}},\quad 
              {[Z_2]_{\G}}={[\Sigma_2+m_2m_2^\top]_{\G}}, \notag\\
&\ X=\begin{bmatrix}
1& m_1^\top & m_2^\top \\
m_1 & Z_1 & Y \\
m_2 & Y^\top & Z_2
\end{bmatrix} \in \S^{2d+1}_+,
\end{align}
where $Z_1,Z_2\in \S^d,Y\in \R^{d\times d}$ are decision variables. Here, $Y$ represents the cross-covariance between $x$ and $y$ and the operator {$[\,\cdot\,]_{\G}$} projects a matrix onto the sparsity pattern $\G$ (Definition~\ref{itm:third}). Note that, $\Sigma_i+m_im_i^\top$ is simply the second moment matrix of $\mu$ or $\nu.$ 

We now discuss how to reduce the dimension of \eqref{GSmom} using chordal conversion. For this, we rely on the following lemma, which is an immediate corollary of the classical results of Grone and Agler~\cite{agler1988positive,grone1984positive}.

\begin{lem}\label{lemGS}
Consider the linear SDP
\begin{equation}\label{lSDP}
\min\left\{ \langle A_0,X\rangle:\ \langle A_i,X\rangle=b_i\ \forall i\in[m],\ X\in\S_+^n \right\},
\end{equation}
where each $A_i\in {\S_G}$ for $i\in[m]\sqcup\{0\}$ (Definition~\ref{itm:third}).  
If $G$ is chordal with maximal cliques $V_1,\ldots,V_p\subset[n]$, then \eqref{lSDP} is equivalent to
\begin{equation}\label{lSDPmt}
\min\left\{ \langle A_0,X\rangle:\ \langle A_i,X\rangle=b_i\ \forall i\in[m],\ X\in {\S_G},\ X_{V_t,V_t}\succeq 0\ \forall t\in[p] \right\},
\end{equation}
which is a multi-block SDP with block sizes $|V_1|,|V_2|,\ldots,|V_p|$.
\end{lem}

Using Lemma~\ref{lemGS}, we obtain the following result for \eqref{GSmom}.

\begin{prop}\label{propGS}
If the reference graph $\G$ (Definition~\ref{itm:third}) in \eqref{GSmom} is chordal with maximal cliques $V_1,\ldots,V_p\subset[d]$, then \eqref{GSmom} is equivalent to a multi-block SDP with block sizes $2|V_1|+1,\ 2|V_2|+1,\ \ldots,\ 2|V_p|+1.$
\end{prop}

\begin{proof}
Define a graph $G$ on vertices $[2d+1]$ whose edges consist of those in the cliques  
\begin{equation}\label{deficlique}
V_i' := \{1\}\,\sqcup\, (V_i+1)\,\sqcup\, (V_i+1+d),\qquad i\in[p].
\end{equation}
This $G$ is precisely the sparsity pattern of \eqref{GSmom}, and it is chordal with maximal cliques $V_1',\ldots,V_p'$.  
Applying Lemma~\ref{lemGS} then yields the desired multi-block structure for \eqref{GSmom}, with block sizes $2|V_i|+1$ for $i\in[p]$.  
\end{proof}

\begin{rem}\label{remGS}
When $\G$ is a tree, its maximal cliques are exactly the edges of $\G$. In this case, Proposition~\ref{propGS} implies that \eqref{GSmom} can be reformulated as a multi-block SDP in which each block has size~5. This reduces the dense matrix dimension from $2d+1$ to blocks of constant size. Moreover, in this setting, \eqref{GSmom} coincides with the sparse cluster moment relaxation \eqref{OTmom-1}, where the constraint~\eqref{spconsA} corresponds to the PSD conditions on these small matrix blocks.
\end{rem}

Proposition~\ref{propGS} shows how the sparsity of the reference graph $\G$ (Definition~\ref{itm:third}) can be leveraged to reduce the dimension of the SDP~\eqref{GSmom}. We next establish a result regarding the tightness of our cluster moment relaxation for OT between Gaussian distributions.

\begin{theo}\label{thmexp}
Suppose $\Sigma_1,\Sigma_2 \in \S^d_{++}$ satisfy 
$a I_d \preceq \Sigma_1,\Sigma_2 \preceq b I_d$ for some $a,b>0$, 
and the precision matrices $\Sigma_1^{-1},\Sigma_2^{-1} \in {\S_G}$ for a graph $G=([d],\E)$. 
Let ${\rm opt}_{\G}$ denote the optimal value of \eqref{GSmom} (special case of \eqref{OTmom-2}). We have that:
\begin{itemize}
\item[(i)] If $\G$ is a complete graph, then the relaxation is exact, that is, ${\rm opt}_{\G}=W_2^2(\mu,\nu).$
\item[(ii)] If $\G=G^h$ for some $h\in \N$ (Definition~\ref{itm:third}), then there exist constants $C>0$ and $\rho>1$, depending only on $a,b$, such that
\begin{equation}\label{thmine}
\big|\,{\rm opt}_{G^h} - W_2^2(\mu,\nu)\,\big| \;<\; C d \rho^{-h}.
\end{equation}
\end{itemize}
\end{theo}

In the above theorem, the assumption $\Sigma_1^{-1}, \Sigma_2^{-1} \in {\S_G}$ implies that the two Gaussians are Markovian with respect to the graph $G$, meaning that variables not connected by an edge in $G$ are conditionally independent given the others. This property, known as \emph{correlative sparsity}, arises in Gaussian graphical models \cite{wainwright2008graphical}. The full proof of Theorem~\ref{thmexp} is given in Appendix~\ref{sec:proof}.

Theorem~\ref{thmexp} (i) establishes the exactness of our cluster moment relaxation \eqref{GSmom} (equivalently, \eqref{OTmom-2}) for optimal transport between Gaussian distributions when $\G$ is complete. This matches the classical fact that a Gaussian distribution is fully determined by its first and second moments. In practice, these low-order moments can be estimated from a moderate number of samples. In contrast, the standard OT solver that substitutes $\mu,\nu$ in~\eqref{OT} with their empirical measures suffers from the curse of dimensionality and requires exponentially many samples in~$d$ to achieve comparable accuracy.

Theorem~\ref{thmexp} (ii) further shows that even in the Gaussian case our convex relaxation improves computational efficiency. When $G$ is sparse, its neighborhood extensions $G^h$ (Definition~\ref{itm:third}) also remain sparse for small $h$, so the relaxation imposes only a limited number of moment constraints. The exponential approximation rate in~\eqref{thmine} indicates that a small $h$ already yields an accurate approximation of the true Wasserstein distance. This sparsity enables the use of chordal conversion (see Proposition~\ref{propGS}) to further reduce the dimension and computational cost of solving the SDP~\eqref{GSmom}.

{
We now discuss the sampling error in the Gaussian setting. Let
\begin{equation}
{\rm opt}_{\G}(m_1,m_2,\Sigma_1,\Sigma_2)
\notag
\end{equation}
denote the value of the Gaussian SDP \eqref{GSmom}. Equivalently, after taking the Schur
complement with respect to the leading scalar entry in \eqref{GSmom}, this value can be
written as
\begin{equation}\label{gauss-reduced-value}
{\rm opt}_{\G}(m_1,m_2,\Sigma_1,\Sigma_2)
=
\|m_1-m_2\|^2+\phi_{\G}(\Sigma_1,\Sigma_2),
\end{equation}
where the covariance part has the dual representation
\begin{equation}\label{gauss-cov-dual}
\phi_{\G}(\Sigma_1,\Sigma_2)
=
\sup_{\Lambda_1,\Lambda_2\in\S_{\G}}
\left\{
\Tr(\Sigma_1)+\Tr(\Sigma_2)
-\langle\Sigma_1,\Lambda_1\rangle
-\langle\Sigma_2,\Lambda_2\rangle:
\begin{bmatrix}
\Lambda_1&-I_d\\
-I_d&\Lambda_2
\end{bmatrix}\succeq0
\right\}.
\end{equation}
Here $\S_{\G}$ is defined in \eqref{defspset}.

\begin{theo}[Gaussian approximation and sampling error]\label{thm:gauss-sample}
Assume the setting of Theorem~\ref{thmexp}. In addition, assume
$\|m_1-m_2\|=O(\sqrt d)$. Let $\G=G^h$, and let $\Delta_h$ be the maximum degree of $G^h$.
Given $N$ independent samples from each Gaussian, let $\widehat m_i$ and
$\widehat\Sigma_i$ be the empirical mean and empirical covariance, and let
\begin{equation}
\widehat{\operatorname{opt}}_{G^h}:=
{\rm opt}_{G^h}(\widehat m_1,\widehat m_2,\widehat\Sigma_1,\widehat\Sigma_2).
\notag
\end{equation}
There are constants $C,c>0$, depending only on $a,b$ and the implicit constant in
$\|m_1-m_2\|=O(\sqrt d)$, such that, if
\begin{equation}\label{smcond}
\begin{aligned}
d+|E(G^h)|+\log(1/\delta)&\le cN,\\
(1+\Delta_h)
\sqrt{\frac{\log(d+|E(G^h)|)+\log(1/\delta)}{N}}
&\le c,
\end{aligned}
\end{equation}
then with probability at least $1-\delta$,
\begin{equation}\label{gauss-sample-total}
\big|\widehat{\operatorname{opt}}_{G^h}-W_2^2(\mu,\nu)\big|
\le
C d\rho^{-h}
+C\left(
\sqrt{\frac{d\log(1/\delta)}{N}}
+
\frac{d+|E(G^h)|+\log(1/\delta)}{N}
\right),
\end{equation}
where $\rho>1$ is the constant in Theorem~\ref{thmexp}.
\end{theo}

The proof of Theorem~\ref{thm:gauss-sample} is given in Appendix~\ref{sec:proof}.

\begin{rem}\label{rem:gauss-relative}
The error bound \eqref{gauss-sample-total} implies that sparsity reduces the statistical error. If $h$ is fixed and the graphs $G^h$
have uniformly bounded degree, then the statistical part of \eqref{gauss-sample-total}
becomes
\begin{equation}\label{spSDPsterr}
O\!\left(
\sqrt{\frac{d\log(1/\delta)}{N}}
+
\frac{d+\log(1/\delta)}{N}
\right),
\end{equation}
which is $O(\sqrt{d/N}+d/N)$ when logarithmic factors are suppressed. By contrast,
the same bound applied to the dense moment relaxation, for which
$|E(G^h)|=\Omega(d^2)$, gives
\begin{equation}
O\!\left(
\sqrt{\frac{d\log(1/\delta)}{N}}
+
\frac{d^2+\log(1/\delta)}{N}
\right),
\notag
\end{equation}
This is worse than the sparse bound \eqref{spSDPsterr} because of the $d^2/N$
term. Thus, even within the Gaussian
moment relaxation \eqref{GSmom}, exploiting locality reduces both computational and sample complexity.
\end{rem}
}

{
\subsection{Local perturbations of a mean-field product measure}\label{subsec:mf}

The analysis in the previous subsection relies on the closed-form expression for the squared OT distance between Gaussian measures. In this subsection, we move beyond the Gaussian setting and show that locality still leads to an accurate sparse relaxation for local perturbations of a mean-field product measure. Subsubsection~\ref{subsubsec:mf-model} introduces the perturbative model and locality assumptions. Subsubsection~\ref{subsubsec:mf-results} states the main approximation result, Theorem~\ref{thm:mfloc}, which bounds $E_{\rm app}$ and shows that the marginal relaxation captures the local linearization of the squared Wasserstein distance, also known as the weighted negative homogeneous Sobolev norm~\cite{peyre2018comparison}. The same subsubsection also gives Corollary~\ref{coro:mf-stat}, which controls the statistical error when the relaxation degree and cluster size are fixed as constants independent of the dimension. We do not analyze the truncation error $E_{\rm trunc}$ here, since quantitative convergence rates for the moment-SOS hierarchy beyond the Gaussian setting are substantially more involved. The proof of Theorem~\ref{thm:mfloc} is then split according to its two main estimates: Subsubsection~\ref{subsubsec:mf-pde} develops the linearized PDE argument and proves the upper bound, while Subsubsection~\ref{subsubsec:mf-dual} constructs local dual certificates and proves the lower bound.

\subsubsection{Perturbative model and local structure}\label{subsubsec:mf-model}

Now, we state our models in detail. We first consider the cluster size to be 1, in which case, $x_i,y_i$ denote the $i$th coordinates of $x,y\in\R^d$. Later we will use larger clusters.  Let
\begin{equation}\label{eq:mfref}
    \mubar=\nubar=\bigotimes_{i=1}^d \mubar_i,
    \qquad
    \mubar_i(dx_i)=\rho_i(x_i)dx_i,
    \qquad
    \rho_i(x_i)=Z_i^{-1}e^{-V_i(x_i)} .
\end{equation}
We assume uniform strong convexity and smoothness of the one-site potentials:
\begin{equation}\label{eq:onesite}
    \inf_{i,x}V_i''(x)>0,
    \qquad
    \sup_{i,x}
    \left(
    V_i''(x)+|V_i^{(3)}(x)|+|V_i^{(4)}(x)|
    \right)<\infty .
\end{equation}
We perturb the reference by
\begin{equation}\label{eq:mfpert}
    \mu_\tau=(1+\tau H_\mu)\mubar,
    \qquad
    \nu_\tau=(1+\tau H_\nu)\mubar,
\end{equation}
where
\begin{equation}\label{eq:mf-local-H}
    H_\mu=\sum_{C\in\C} h_C^\mu(x_C),
    \qquad
    H_\nu=\sum_{C\in\C} h_C^\nu(x_C).
\end{equation}
Here $\C$ is a collection of nonempty subsets of $[d]$.

\begin{rem}
The form \eqref{eq:mfpert} is used for analytical convenience. It can be viewed as the first-order expansion of a normalized exponential perturbation such as $Z_{\mu,\tau}^{-1}e^{\tau H_\mu}\mubar$. Our analysis also applies to the corresponding exponential perturbation.
\end{rem}

\begin{assump}[Local perturbation structure]\label{as:mfstruct}
There are constants $s_{\C},r_{\C},M_{\rm pert}>0$, independent of $d$, such that:
\begin{enumerate}[label=(\roman*)]
    \item $|C|\le s_{\C}$ for every $C\in\C$;
    \item each coordinate appears in at most $r_{\C}$ local terms;
    \item the local perturbations satisfy
    \begin{equation}\label{eq:mfdloc}
        \max_{C\in\C}\max_{0\le |\alpha|\le 3}
        \Big(\|\partial^\alpha h_C^\mu\|_\infty
        +\|\partial^\alpha h_C^\nu\|_\infty\Big)
        \le M_{\rm pert} .
    \end{equation}
\end{enumerate}
After subtracting constants from the local terms, we assume without loss of generality that
\begin{equation}
    \int h_C^\mu\,d\mubar_C=
    \int h_C^\nu\,d\mubar_C=0,
    \qquad C\in\C.
\notag
\end{equation}
\end{assump}

\subsubsection{Main approximation and statistical results}\label{subsubsec:mf-results}

We now state the main perturbative approximation theorem. Some objects in the statement
are constructed below: $H_t$ and the smallness condition are defined in
\eqref{eq:mfeps}, $\psi$ is the solution of the weighted Poisson equation
\eqref{eq:cpoi}, equivalently the sum of the local PDE solutions $\psi_C$ in
\eqref{eq:mfpoi}, and the support condition on the partition $\I$ and tree $\G$ is
given in \eqref{eq:mf-support-condition}.

\begin{theo}[Mean-field local perturbation bound for \eqref{OTmr-1}]\label{thm:mfloc}
Assume \eqref{eq:onesite} and Assumption~\ref{as:mfstruct}. Suppose \eqref{eq:mfeps} holds and $|\tau|\|\nabla^2\psi\|_{\infty,{\rm op}}<2$. Let $\I=\{I_a:a\in[K]\}$ be a partition of $[d]$, and let $\G=([K],E(\G))$ be a tree satisfying \eqref{eq:mf-support-condition}. Then
\begin{align}\label{eq:mflocbd}
    &\frac{\tau^2}{4}
    \int\|\nabla\psi\|^2\,d\mubar
    \left[
    1-|\tau|s_{\C}r_{\C}^2M_{\rm pert}
    -\frac{|\tau|\|\nabla^2\psi\|_{\infty,{\rm op}}}
    {2(1-|\tau|\|\nabla^2\psi\|_{\infty,{\rm op}}/2)}
    (1+|\tau|s_{\C}r_{\C}^2M_{\rm pert})
    \right] \notag\\
    &\le
    \operatorname{MR}^{(1)}_{\I,\G}(\mu_\tau,\nu_\tau)
    \le
    W_2^2(\mu_\tau,\nu_\tau) \notag\\
    &\le
    \frac{\tau^2}{4}
    \int\|\nabla\psi\|^2\,d\mubar
    \left(
    1+|\tau|s_{\C}r_{\C}^2M_{\rm pert}
    +2|\tau|^2
    (d s_{\C}r_{\C}^2+s_{\C}^2r_{\C}^4)M_{\rm pert}^2
    \right).
\end{align}
\end{theo}

\begin{rem}\label{rem:mf-relative}
The leading term $\frac{\tau^2}{4}\int \|\nabla\psi\|^2\,d\mubar$ in Theorem~\ref{thm:mfloc} is exactly the linearization of the quadratic Wasserstein distance around the reference measure $\mubar$ \cite{peyre2018comparison}. Therefore, Theorem~\ref{thm:mfloc} shows that the sparse marginal relaxation captures the leading Wasserstein geometry of these local perturbations. Theorem~\ref{thm:mfloc} also gives a relative approximation statement whenever the perturbation is nontrivial. Dividing the gap estimate in Theorem~\ref{thm:mfloc} by this leading-order lower bound gives
\begin{equation}
    \frac{
    W_2^2(\mu_\tau,\nu_\tau)
    -\operatorname{MR}^{(1)}_{\I,\G}(\mu_\tau,\nu_\tau)}
    {W_2^2(\mu_\tau,\nu_\tau)}
    =
    O\!\left(
    |\tau|
    +|\tau|^2(d s_{\C}r_{\C}^2+s_{\C}^2r_{\C}^4)M_{\rm pert}^2
    \right),
\notag
\end{equation}
where the $O(|\tau|)$ term absorbs only the dimension-free first-order factors $s_{\C}r_{\C}^2M_{\rm pert}$ and $\|\nabla^2\psi\|_{\infty,{\rm op}}$ appearing in Theorem~\ref{thm:mfloc}. Since $s_{\C},r_{\C},M_{\rm pert}$ are independent of $d$, this is in particular $O(|\tau|+|\tau|^2d)$. Under the explicit scaling in \eqref{eq:mfeps}, the relative error is $O(|\tau|)$.
\end{rem}

The proof of the main Theorem~\ref{thm:mfloc} has three main steps. First, in Subsubsection~\ref{subsubsec:mf-pde}, the first-order transport direction is identified
through a weighted Poisson equation. Because both the reference measure and the
perturbation are local, this PDE decomposes into local equations on the sets $C\in\C$,
giving $\psi=\sum_C\psi_C$ with dimension-free derivative bounds. Second, in the same Subsubsection~\ref{subsubsec:mf-pde}, an explicit Benamou--Brenier path based on the
velocity $-\nabla\psi/2$ gives the upper bound on $W_2^2(\mu_\tau,\nu_\tau)$. Third, in Subsubsection~\ref{subsubsec:mf-dual}, a local approximation of the $c$-transform of $f_\tau=\tau\psi$ gives a
dual certificate for the sparse marginal relaxation. Evaluating this certificate and
using the Poisson identity yields the lower bound in Theorem~\ref{thm:mfloc}.

Theorem~\ref{thm:mfloc} and Remark~\ref{rem:mf-relative} address the approximation error $E_{\rm app}=|\operatorname{MR}_{\I,\G}(\mu_\tau,\nu_\tau)-W_2^2(\mu_\tau,\nu_\tau)|$ for the marginal relaxation. However, for the cluster moment relaxation, we do not have a quantitative rate for the truncation error $E_{\rm trunc}=\big|\operatorname{MR}_{\I,\G}-\operatorname{OPT}_{n}\big|$ because this would require a convergence rate for the moment-SOS hierarchy approximating the continuous marginal relaxation. Nevertheless, for a fixed degree $n$ and bounded relaxation cluster size, the statistical component $E_{\rm stat}$ can be controlled directly from Theorem~\ref{thm:sdpvalue}. We state it in the following corollary:

\begin{coro}[Statistical error for the mean-field perturbation model]\label{coro:mf-stat}
Assume the setting of Theorem~\ref{thm:mfloc} and Theorem~\ref{thm:sdpvalue}. Fix the relaxation degree $n$ and the maximum cluster size $\max_{a\in[K]} |I_a|\le r$ to be constants independent of $d$. Then, with probability at least $1-\delta$, the statistical error $E_{\rm stat}$ of the degree-$n$ cluster moment SDP satisfies
\begin{equation}\label{eq:mf-stat-bound-dual}
    E_{\rm stat}
    \lesssim
    d
    \left[\log\!\left(\frac{8A_{\rm tail}dN}{\delta}\right)\right]^n
    \sqrt{
    \frac{\log d+\log(1/\delta)}{N}
    } ,
\end{equation}
where $A_{\rm tail}$ is a constant in the Gaussian tail assumption~\ref{assump:sampletail}. 
\end{coro}

\begin{proof}
In the notation of Theorem~\ref{thm:sdpvalue}, $p(b)=\operatorname{OPT}_{n}$ and $p(\widehat b)=\widehat{\operatorname{OPT}}_{n}$. Since $\G$ is a tree in Theorem~\ref{thm:mfloc}, $|E(\G)|=K-1$. The bounded relaxation cluster size and fixed degree $n$ imply that the local basis size is uniformly bounded, and hence the number of moments $M_{\rm mom}$ is upper bounded by $(K+|E(\G)|)\binom{r+n}{n}^2=\O(d)$. Applying Theorem~\ref{thm:sdpvalue} gives \eqref{eq:mf-stat-bound-dual}.
\end{proof}

\subsubsection{Linearized PDE and upper bound}\label{subsubsec:mf-pde}

Now, we discuss the proof of Theorem~\ref{thm:mfloc} in detail. We first state some preliminaries.

\paragraph{Preliminaries.} For $t\in[0,1]$, define
\begin{equation}
    h_C^t:=(1-t)h_C^\mu+t h_C^\nu,
    \qquad
    H_t:=\sum_{C\in\C}h_C^t .
\notag
\end{equation}
Since each coordinate appears in at most $r_{\C}$ local terms, the number of sets in $\C$ is at most $d r_{\C}$. Hence Assumption~\ref{as:mfstruct} gives
\begin{equation}
    \|H_t\|_\infty\le d r_{\C}M_{\rm pert} .
\notag
\end{equation}
Thus positivity of $1+\tau H_t$ is dimension-dependent. We assume that, for the values of $\tau$ considered here,
\begin{equation}\label{eq:mfeps}
    1+\tau H_t(x)\ge \frac12,
    \qquad
    t\in[0,1],\ x\in\R^d,
    \qquad
    |\tau|d\le \frac{1}{2r_{\C}M_{\rm pert}} .
\end{equation}
Indeed, the last condition and the bound $\|H_t\|_\infty\le d r_{\C}M_{\rm pert}$ imply
$|\tau|\|H_t\|_\infty\le 1/2$, and hence the positivity condition above.

Write
\begin{equation}\label{eq:mftheta}
    \Theta:=H_\nu-H_\mu=
    \sum_{C\in\C}\theta_C(x_C),
    \qquad
    \theta_C:=h_C^\nu-h_C^\mu,
    \qquad
    \int\theta_C\,d\mubar_C=0.
\end{equation}
For the product mean-field measure, the weighted Poisson operator decomposes as
\begin{equation}\label{eq:mfL}
    L=\sum_{i=1}^d L_i,
    \qquad
    L_i u:=-\rho_i^{-1}\partial_i(\rho_i\partial_i u).
\end{equation}
For $U\subset[d]$, write $L_U:=\sum_{i\in U}L_i$.

\paragraph{Mean-field locality.}
The linearized transport direction is determined by a weighted Poisson equation. Indeed, if the source Kantorovich potential is $f_\tau=\tau\psi+o(\tau)$, then the corresponding map has the expansion
\begin{equation}
    T_\tau(x)=x-\frac{\tau}{2}\nabla\psi(x)+o(\tau).
\notag
\end{equation}
Since $\rho_{\nu_\tau}-\rho_{\mu_\tau}=\tau\Theta\bar\rho$, the first-order velocity field
\begin{equation}
    w:=-\frac12\nabla\psi
\notag
\end{equation}
should satisfy $\nabla\cdot(\bar\rho w)=-\Theta\bar\rho$. Equivalently,
\begin{equation}\label{eq:cpoi}
    L\psi=-2\Theta,
    \qquad
    L=\sum_{i=1}^dL_i .
    \tag{PE}
\end{equation}
\addtocounter{equation}{1}
Because both $L$ and $\Theta$ are local sums, this equation can be solved locally.

For each $C\in\C$, let $\psi_C$ solve the local equation
\begin{equation}\label{eq:mfpoi}
    L_C\psi_C=-2\theta_C,
    \qquad
    \int\psi_C\,d\mubar_C=0.
\end{equation}

\begin{lem}[Local Poisson solutions]\label{lem:mf-local-poisson}
Assume \eqref{eq:onesite} and Assumption~\ref{as:mfstruct}. For each $C\in\C$, equation \eqref{eq:mfpoi} admits a solution $\psi_C$ satisfying
\begin{equation}\label{eq:mfpoibd}
    \|\nabla\psi_C\|_\infty+
    \|\nabla^2\psi_C\|_\infty+
    \Lip(\nabla^2\psi_C)
    \le K_P .
\end{equation}
Here $K_P$ is independent of $C$ and $d$. If
\begin{equation}\label{eq:mf-psi-decomp}
    \psi(x):=\sum_{C\in\C}\psi_C(x_C).
\end{equation}
then $L\psi=-2\Theta$ and $\int\psi\,d\mubar=0.$ In particular, $\psi$ is the normalized solution of \eqref{eq:cpoi} and is exactly local.
\end{lem}

\begin{proof}
Fix $C\in\C$ and write $\mubar_C=Z_C^{-1}e^{-V_C(x_C)}dx_C$, where $V_C=\sum_{i\in C}V_i$. The corresponding Stein operator is
\begin{equation}
    A_Cu=\frac12\Delta_Cu-\frac12\nabla V_C\cdot\nabla u
    =-\frac12L_Cu .
\notag
\end{equation}
Thus \eqref{eq:mfpoi} is equivalent to $A_C\psi_C=\theta_C$. Since $\int\theta_C\,d\mubar_C=0$, the Stein factor estimate for strongly log-concave measures \cite[Theorem 2.1]{mackey2016multivariate} gives a solution with bounded first, second, and third derivative seminorms. The constant is uniform in $C$ and $d$ because $|C|\le s_{\C}$ and the one-site potentials satisfy the uniform bounds \eqref{eq:onesite}. Subtracting a constant gives the normalization $\int\psi_Cd\mubar_C=0$.

Finally, by product structure and the definition of $L_C$, the sum $\psi=\sum_C\psi_C$ satisfies
\begin{equation}
    L\psi=\sum_{C\in\C}L_C\psi_C=-2\sum_{C\in\C}\theta_C=-2\Theta.
\notag
\end{equation}
The normalization follows from $\int\psi_C\,d\mubar_C=0$ for every $C$. 
\end{proof}

For later use, define the interaction neighborhood of coordinate $i$ by
\begin{equation}\label{eq:mf-Si}
    S_i:=\bigcup_{C\in\C:\ i\in C}C .
\end{equation}
Then $\partial_i\psi$ depends only on $x_{S_i}$, and Assumption~\ref{as:mfstruct} gives $|S_i|\le s_{\C}r_{\C}$.

\begin{lem}[Weighted estimates for local sums]\label{lem:mfdim}
Assume \eqref{eq:onesite} and Assumption~\ref{as:mfstruct}. Then, uniformly for $t\in[0,1]$,
\begin{align}\label{eq:mfext}
    \left|\int\|\nabla\psi\|^2 H_t\,d\mubar\right|
    &\le
    s_{\C}r_{\C}^2M_{\rm pert}
    \int\|\nabla\psi\|^2\,d\mubar, \notag\\
    \int\|\nabla\psi\|^2 H_t^2\,d\mubar
    &\le
    (d s_{\C}r_{\C}^2+s_{\C}^2r_{\C}^4)M_{\rm pert}^2
    \int\|\nabla\psi\|^2\,d\mubar .
\end{align}
\end{lem}

\begin{proof}
Let $A_i:=|\partial_i\psi|^2$. By \eqref{eq:mf-Si}, $A_i$ depends only on $x_{S_i}$ and $|S_i|\le s_{\C}r_{\C}$. If $D\cap S_i=\emptyset$, then by product structure and centering,
\begin{equation}
    \int A_i h_D^t\,d\mubar=0.
\notag
\end{equation}
There are at most $|S_i|r_{\C}\le s_{\C}r_{\C}^2$ sets $D$ intersecting $S_i$. Since $\|h_D^t\|_\infty\le M_{\rm pert}$ and $A_i\ge0$,
\begin{equation}
    \left|\int A_iH_t\,d\mubar\right|
    \le
    s_{\C}r_{\C}^2M_{\rm pert}
    \int A_i\,d\mubar .
\notag
\end{equation}
Summing over $i$ gives the first bound in \eqref{eq:mfext}.

For the second bound, expand
\begin{equation}
    \int A_iH_t^2\,d\mubar
    =
    \sum_{D,E\in\C}
    \int A_i h_D^t h_E^t\,d\mubar .
\notag
\end{equation}
The integral vanishes unless either $D\cap E\neq\emptyset$, or both $D$ and $E$ intersect $S_i$. Indeed, if for instance $D$ is disjoint from both $E$ and $S_i$, the factor $h_D^t$ is independent of the rest of the integrand and has mean zero. The number of ordered pairs $(D,E)$ with $D\cap E\neq\emptyset$ is at most
\begin{equation}
    |\C|\,s_{\C}r_{\C}\le d s_{\C}r_{\C}^2,
\notag
\end{equation}
because $|\C|\le d r_{\C}$ and each $D$ intersects at most $s_{\C}r_{\C}$ sets $E$. The number of ordered pairs for which both $D$ and $E$ intersect $S_i$ is at most $s_{\C}^2r_{\C}^4$. Hence
\begin{equation}
    \int A_iH_t^2\,d\mubar
    \le
    (d s_{\C}r_{\C}^2+s_{\C}^2r_{\C}^4)M_{\rm pert}^2
    \int A_i\,d\mubar .
\notag
\end{equation}
Summing over $i$ gives the second bound in \eqref{eq:mfext}. \end{proof}

The first estimate in Lemma~\ref{lem:mfdim} gives a dimension-free relative error term below, while the second-order relative term contains the explicit factor
$\tau^2(d s_{\C}r_{\C}^2+s_{\C}^2r_{\C}^4)M_{\rm pert}^2$, whose dimension-dependent part is harmless under \eqref{eq:mfeps}.

\begin{lem}[Upper bound for the mean-field perturbation]\label{lem:mfexp}
Assume \eqref{eq:onesite} and Assumption~\ref{as:mfstruct}. Suppose \eqref{eq:mfeps} holds and $|\tau|\|\nabla^2\psi\|_{\infty,{\rm op}}<2$. Then
\begin{align}\label{eq:mfexp}
    W_2^2(\mu_\tau,\nu_\tau)
    \le
    \frac{\tau^2}{4}
    \int\|\nabla\psi\|^2\,d\mubar
    \left(
    1+|\tau|s_{\C}r_{\C}^2M_{\rm pert}
    +2|\tau|^2
    (d s_{\C}r_{\C}^2+s_{\C}^2r_{\C}^4)M_{\rm pert}^2
    \right).
\end{align}
\end{lem}

\begin{proof}
We use the Benamou--Brenier formula \cite{benamou2000computational}. Let
\begin{equation}
    \rho_t^\tau=(1+\tau H_t)\bar\rho,\qquad
    w=-\frac12\nabla\psi,\qquad
    v_t^\tau=\frac{\tau w}{1+\tau H_t}.
\notag
\end{equation}
Since $\nabla\cdot(\bar\rho w)=-\Theta\bar\rho$, the pair $(\rho_t^\tau,v_t^\tau)$ satisfies the continuity equation from $\mu_\tau$ to $\nu_\tau$. Therefore
\begin{equation}
    W_2^2(\mu_\tau,\nu_\tau)
    \le
    \tau^2\int_0^1\int\frac{\|w\|^2}{1+\tau H_t}\,d\mubar\,dt,
    \qquad w=-\frac12\nabla\psi.
\notag
\end{equation}
	Using
	\begin{equation}
	    \frac1{1+\tau H_t}
	    =1-\tau H_t+
	    \frac{\tau^2H_t^2}{1+\tau H_t},
	\notag
	\end{equation}
	together with \eqref{eq:mfeps}, we get
	\begin{equation}
	    W_2^2(\mu_\tau,\nu_\tau)
	    \le
	    \frac{\tau^2}{4}\int\|\nabla\psi\|^2\,d\mubar
	    +\frac{|\tau|^3}{4}
	    \sup_{t\in[0,1]}
	    \left|\int\|\nabla\psi\|^2H_t\,d\mubar\right|
	    +\frac{|\tau|^4}{2}
	    \sup_{t\in[0,1]}\int\|\nabla\psi\|^2H_t^2\,d\mubar .
	\notag
	\end{equation}
	Using \eqref{eq:mfext} gives \eqref{eq:mfexp}.
	\end{proof}

\subsubsection{Local dual certificates and final proof}\label{subsubsec:mf-dual}
We next relate the sparse marginal relaxation \eqref{OTmr-1} to local dual certificates. Let $\I=\{I_a:a\in[K]\}$ be a partition of $[d]$, and let $\G=([K],E(\G))$ be a tree. For $ab\in E(\G)$, write
\begin{equation}
    I_{ab}:=I_a\cup I_b,
    \qquad
    z_a:=(x_{I_a},y_{I_a}),
    \qquad
    z_{ab}:=(x_{I_{ab}},y_{I_{ab}}).
\notag
\end{equation}
Set
\begin{equation}
    c_a(z_a):=\sum_{i\in I_a}|x_i-y_i|^2,
    \qquad
    \sum_{a\in[K]}c_a(z_a)=\|x-y\|^2 .
\notag
\end{equation}
The sparse marginal relaxation can be written as
\begin{align}\label{eq:otmr1-tree-local}
    \operatorname{MR}^{(1)}_{\I,\G}(\mu_\tau,\nu_\tau)
    :=
    \inf_{(\pi_a,\pi_{ab})}
    &\quad
    \sum_{a\in[K]}\int c_a\,d\pi_a  \\
    {\rm s.t.}
    &\quad
    {\rm P}_{x_{I_a}}\pi_a={\rm P}_{x_{I_a}}\mu_\tau,
    \quad
    {\rm P}_{y_{I_a}}\pi_a={\rm P}_{y_{I_a}}\nu_\tau,
    \qquad a\in[K], \notag \\
    &\quad
    {\rm P}_{x_{I_{ab}}}\pi_{ab}={\rm P}_{x_{I_{ab}}}\mu_\tau,
    \quad
    {\rm P}_{y_{I_{ab}}}\pi_{ab}={\rm P}_{y_{I_{ab}}}\nu_\tau,
    \qquad ab\in E(\G), \notag\\
    &\quad
    {\rm P}_{z_a}\pi_{ab}=\pi_a,
    \quad
    {\rm P}_{z_b}\pi_{ab}=\pi_b,
    \qquad ab\in E(\G), \notag\\
    &\quad
    \pi_a\in\mathcal P(\R^{I_a}\times\R^{I_a}),
    \quad
    \pi_{ab}\in\mathcal P(\R^{I_{ab}}\times\R^{I_{ab}}). \notag
\end{align}
We also use the following auxiliary global formulation:
\begin{align}\label{eq:tree-cluster-MR}
    \widehat{\operatorname{MR}}^{(1)}_{\I,\G}(\mu_\tau,\nu_\tau)
    :=
    \inf_{\pi}
    &\quad
    \int\|x-y\|^2\,d\pi(x,y) \\
    {\rm s.t.}
    &\quad
    {\rm P}_{x_{I_a}}\pi={\rm P}_{x_{I_a}}\mu_\tau,
    \quad
    {\rm P}_{y_{I_a}}\pi={\rm P}_{y_{I_a}}\nu_\tau,
    \qquad a\in[K], \notag \\
    &\quad
    {\rm P}_{x_{I_{ab}}}\pi={\rm P}_{x_{I_{ab}}}\mu_\tau,
    \quad
    {\rm P}_{y_{I_{ab}}}\pi={\rm P}_{y_{I_{ab}}}\nu_\tau,
    \qquad ab\in E(\G), \notag \\
    &\quad
    \pi\in\mathcal P(\R^d\times\R^d). \notag
\end{align}
The above problem \eqref{eq:tree-cluster-MR} is not one of the marginal relaxations introduced in Section~\ref{Sec:cvxrl}: it keeps a single global measure $\pi$ on $(x,y)$, while prescribing only selected local $x$- and $y$-marginals. We introduce it only as an intermediate problem. When $\G$ is a tree, it is equivalent to the sparse marginal relaxation \eqref{eq:otmr1-tree-local}, as shown next.

\begin{lem}[Tree gluing for \eqref{OTmr-1}]\label{lem:tree-otmr1-gluing}
If $\G$ is a tree, then
\begin{equation}
    \operatorname{MR}^{(1)}_{\I,\G}(\mu_\tau,\nu_\tau)
    =
    \widehat{\operatorname{MR}}^{(1)}_{\I,\G}(\mu_\tau,\nu_\tau).
\notag
\end{equation}
\end{lem}

\begin{proof}
Any feasible $\pi$ in \eqref{eq:tree-cluster-MR} gives feasible local marginals $({\rm P}_{z_a}\pi,{\rm P}_{z_{ab}}\pi)$ in \eqref{eq:otmr1-tree-local}, with the same cost. Conversely, if $(\pi_a,\pi_{ab})$ is feasible for \eqref{eq:otmr1-tree-local}, repeated application of the gluing lemma \cite[App.~B, Lem.~B.5]{chewi2025statistical} gives a probability measure $\pi$ whose $z_a$- and $z_{ab}$-marginals are $\pi_a$ and $\pi_{ab}$. This $\pi$ is feasible for \eqref{eq:tree-cluster-MR}, again with the same cost. 
\end{proof}

\begin{lem}[Local dual certificate for the tree relaxation]\label{lem:tree-cluster-dual}
Suppose $f,g$ are globally dual feasible for the squared cost:
\begin{equation}
    f(x)+g(y)\le \|x-y\|^2,
    \qquad x,y\in\R^d.
\notag
\end{equation}
Assume that $f$ and $g$ decompose into node and edge terms on $\G$:
\begin{equation}
    f(x)=\sum_{a\in[K]}f_a(x_{I_a})+
    \sum_{ab\in E(\G)}f_{ab}(x_{I_{ab}}),
\notag
\end{equation}
and similarly
\begin{equation}
    g(y)=\sum_{a\in[K]}g_a(y_{I_a})+
    \sum_{ab\in E(\G)}g_{ab}(y_{I_{ab}}).
\notag
\end{equation}
Then
\begin{equation}\label{eq:tree-dual-lower}
    \operatorname{MR}^{(1)}_{\I,\G}(\mu_\tau,\nu_\tau)
    \ge
    \int f\,d\mu_\tau+
    \int g\,d\nu_\tau .
\end{equation}
\end{lem}

\begin{proof}
By Lemma~\ref{lem:tree-otmr1-gluing}, it is enough to use the global formulation \eqref{eq:tree-cluster-MR}. For any feasible $\pi$,
\begin{equation}
    \int\|x-y\|^2d\pi
    \ge
    \int f(x)d\pi+
    \int g(y)d\pi.
\notag
\end{equation}
Because $f$ and $g$ are sums of terms supported on node or edge clusters, the local marginal constraints in \eqref{eq:tree-cluster-MR} imply
\begin{equation}
    \int f(x)d\pi=\int f\,d\mu_\tau,
    \qquad
    \int g(y)d\pi=\int g\,d\nu_\tau.
\notag
\end{equation}
Taking the infimum over $\pi$ gives \eqref{eq:tree-dual-lower}. 
\end{proof}

We now build a dual certificate whose terms are local. Let
\begin{equation}\label{eq:mf-local-f}
    f_\tau(x):=\tau\psi(x),
\end{equation}
and set
	\begin{equation}\label{eq:mf-local-g}
	    g_\tau^{\rm loc}(y)
	    :=
	    -\tau\psi(y)-\frac{\tau^2}{4}\|\nabla\psi(y)\|^2
	    -\frac{|\tau|^3\|\nabla^2\psi\|_{\infty,{\rm op}}}
	    {8(1-|\tau|\|\nabla^2\psi\|_{\infty,{\rm op}}/2)}
	    \|\nabla\psi(y)\|^2.
	\end{equation}
This comes from the local approximation of the $c$-transform of $f_\tau$.
Recall the interaction neighborhoods $S_i$ defined in \eqref{eq:mf-Si}.
We assume that the partition $\I$ and the tree $\G$ contain all supports appearing in $f_\tau$ and $g_\tau^{\rm loc}$:
\begin{equation}\label{eq:mf-support-condition}
    \text{for every } A\in \C\cup\{S_i:i\in[d]\},
    \quad
    A\subset I_a \text{ for some }a,
    \quad\text{or}\quad
    A\subset I_{ab} \text{ for some }ab\in E(\G).
\end{equation}

\begin{lem}[Local corrected dual potential]\label{lem:mf-local-corrected}
Assume \eqref{eq:onesite}, Assumption~\ref{as:mfstruct}, and \eqref{eq:mf-support-condition}. Suppose \eqref{eq:mfeps} holds and $|\tau|\|\nabla^2\psi\|_{\infty,{\rm op}}<2$. Then
\begin{equation}
    f_\tau(x)+g_\tau^{\rm loc}(y)\le \|x-y\|^2,
    \qquad x,y\in\R^d.
\notag
\end{equation}
Moreover, $f_\tau$ and $g_\tau^{\rm loc}$ decompose into node and edge terms on $\G$.
\end{lem}

\pf
For the squared cost, the exact $c$-transform of $f_\tau$ is
\begin{equation}\label{eq:mf-c-transform}
    g_\tau(y):=f_\tau^c(y)
    =
    \inf_{x\in\R^d}
    \{\|x-y\|^2-f_\tau(x)\}
    =
    \inf_{q\in\R^d}
    \{\|q\|^2-\tau\psi(y+q)\}.
\end{equation}
By definition, $(f_\tau,g_\tau)$ is dual feasible. We now lower bound $g_\tau$ by a local expression. For fixed $y$, Taylor's formula gives
\begin{equation}\label{eq:mf-taylor-ct}
    -\tau\psi(y+q)
    \ge
    -\tau\psi(y)-\tau\nabla\psi(y)\cdot q
    -\frac{|\tau|\|\nabla^2\psi\|_{\infty,{\rm op}}}{2}\|q\|^2 .
\end{equation}
Set $\alpha:=|\tau|\|\nabla^2\psi\|_{\infty,{\rm op}}/2$. Combining \eqref{eq:mf-c-transform} and \eqref{eq:mf-taylor-ct}, we obtain
\begin{equation}
    g_\tau(y)
    \ge
    -\tau\psi(y)
    +
    \inf_{q\in\R^d}
    \left\{(1-\alpha)\|q\|^2-\tau\nabla\psi(y)\cdot q\right\}.
\notag
\end{equation}
Since $|\tau|\|\nabla^2\psi\|_{\infty,{\rm op}}<2$, we have $\alpha<1$, and the quadratic minimization gives
\begin{equation}
    \inf_{q\in\R^d}
    \left\{(1-\alpha)\|q\|^2-\tau\nabla\psi(y)\cdot q\right\}
    =
    -\frac{\tau^2}{4(1-\alpha)}\|\nabla\psi(y)\|^2.
\notag
\end{equation}
Therefore,
	\begin{equation}\label{estgmf}
	    g_\tau(y)
	    \ge
	    -\tau\psi(y)
	    -\frac{\tau^2}{4}\|\nabla\psi(y)\|^2
	    -\frac{|\tau|^3\|\nabla^2\psi\|_{\infty,{\rm op}}}
	    {8(1-|\tau|\|\nabla^2\psi\|_{\infty,{\rm op}}/2)}
	    \|\nabla\psi(y)\|^2.
	\end{equation}
	Thus $g_\tau^{\rm loc}\le g_\tau$, and dual feasibility follows from the feasibility of $(f_\tau,g_\tau)$.

It remains to check locality. Since $\psi=\sum_C\psi_C$, the function $f_\tau$ is a sum of terms supported on sets $C\in\C$. Also
	\begin{equation}
	    \|\nabla\psi\|^2=\sum_{i=1}^d |\partial_i\psi|^2,
	\notag
	\end{equation}
	and $|\partial_i\psi|^2$ is supported on $S_i$. By \eqref{eq:mf-support-condition}, every such support is contained in a node or edge cluster. \qed

\noindent{\it Proof of Theorem~\ref{thm:mfloc}.\quad}
The middle inequality in \eqref{eq:mflocbd} holds because every feasible coupling for the full OT problem induces feasible local marginals for \eqref{eq:otmr1-tree-local}. Recall the definitions of $f_\tau$ and $g_\tau^{\rm loc}$ in \eqref{eq:mf-local-f} and \eqref{eq:mf-local-g}. For the lower bound, Lemmas~\ref{lem:mf-local-corrected} and~\ref{lem:tree-cluster-dual} give
\begin{equation}\label{eq:mf-MR-dual-lower}
    \operatorname{MR}^{(1)}_{\I,\G}(\mu_\tau,\nu_\tau)
    \ge
    \int f_\tau\,d\mu_\tau+
    \int g_\tau^{\rm loc}\,d\nu_\tau .
\end{equation}
Using \eqref{eq:mfpert}, \eqref{eq:mftheta}, \eqref{eq:mf-local-f}, \eqref{eq:mf-local-g}, and $\int\psi\,d\mubar=0$, the right-hand side has the exact expansion
	\begin{align}
	    \int f_\tau\,d\mu_\tau+
	    \int g_\tau^{\rm loc}\,d\nu_\tau
	    &=
	    -\tau^2\int\psi\Theta\,d\mubar
	    -\frac{\tau^2}{4}\int\|\nabla\psi\|^2\,d\mubar
	    -\frac{\tau^3}{4}\int\|\nabla\psi\|^2H_\nu\,d\mubar \notag\\
	    &\quad
	    -\frac{|\tau|^3\|\nabla^2\psi\|_{\infty,{\rm op}}}
	    {8(1-|\tau|\|\nabla^2\psi\|_{\infty,{\rm op}}/2)}
	    \int\|\nabla\psi\|^2(1+\tau H_\nu)\,d\mubar .
	    \label{eq:mf-local-dual-expansion}
	\end{align}
The first term is controlled by the Poisson equation. Indeed, since $\psi$ solves $L\psi=-2\Theta$, integration by parts for the weighted operator $L$ gives
\begin{equation}
    \int\|\nabla\psi\|^2\,d\mubar
    =
    \int \psi L\psi\,d\mubar
    =
    -2\int\psi\Theta\,d\mubar .
\notag
\end{equation}
Also, since $H_1=H_\nu$, \eqref{eq:mfext} with $t=1$ gives
\begin{equation}
    \left|\int\|\nabla\psi\|^2H_\nu\,d\mubar\right|
    \le
    s_{\C}r_{\C}^2M_{\rm pert}
    \int\|\nabla\psi\|^2\,d\mubar
\notag
\end{equation}
and therefore
\begin{equation}
    \int\|\nabla\psi\|^2(1+\tau H_\nu)\,d\mubar
    \le
    (1+|\tau|s_{\C}r_{\C}^2M_{\rm pert})
    \int\|\nabla\psi\|^2\,d\mubar .
\notag
\end{equation}
Combining these estimates in \eqref{eq:mf-local-dual-expansion} yields
	\begin{align}
	    \int f_\tau\,d\mu_\tau+
	    \int g_\tau^{\rm loc}\,d\nu_\tau
	    &\ge
	    \frac{\tau^2}{4}
	    \int\|\nabla\psi\|^2\,d\mubar
	    \bigg[
	    1
	    -|\tau|s_{\C}r_{\C}^2M_{\rm pert} \notag\\
	    &\qquad
	    -\frac{|\tau|\|\nabla^2\psi\|_{\infty,{\rm op}}}
	    {2(1-|\tau|\|\nabla^2\psi\|_{\infty,{\rm op}}/2)}
	    (1+|\tau|s_{\C}r_{\C}^2M_{\rm pert})
	    \bigg].
	    \label{eq:mf-local-dual-value}
	\end{align}
Together with \eqref{eq:mf-MR-dual-lower}, this proves the lower bound. The upper bound follows from Lemma~\ref{lem:mfexp}. \qed

}

{
\subsection{Sample complexity of cluster moment relaxation}\label{subsec:samplecomplex}

In this subsection, we study the statistical error $E_{\rm stat}$ of the cluster moment relaxation. Our main result, Theorem~\ref{thm:sdpvalue}, establishes the bound $E_{\rm stat}=\O((d+|E(\G)|)N^{-1/2})$ when the relaxation degree and cluster size are fixed independently of $d$. This result highlights how the sparsity of $\G$ reduces the statistical error of the moment-SOS relaxation. Throughout this subsection, we use the monomial basis up
to degree $n$. Let

\begin{equation}\label{sampledef}
r := \max_{k\in[K]} \dim(z_k),
\qquad
q := \max_{k\in[K]} |\Phi_k| \leq \binom{r+n}{n},
\end{equation}
and define
\begin{equation}\label{sampleMmom}
M_{\rm mom}
:=
\sum_{k\in[K]} |\Phi_k|^2
+
\sum_{ij\in E(\G)} |\Phi_i|\,|\Phi_j|.
\end{equation}
Then $M_{\rm mom}$ is an upper bound, up to a factor of two of the the number of sample-estimated moments used in the cluster moment relaxation.
We will refer to the case where $r$ and $n$ are bounded independently of $d$ and
$|E(\G)|=O(K)$, with $K=O(d)$, as the \emph{sparse fixed-degree regime}. In this regime,
using $q\le \binom{r+n}{n}$, \eqref{sampleMmom} gives
\begin{equation}
M_{\rm mom}
\le
(K+|E(\G)|)q^2
\le
(K+|E(\G)|)\binom{r+n}{n}^2
=O(d).
\notag
\end{equation}
Because the approximation results in Subsections~\ref{Sec:ThmGauss} and~\ref{subsec:mf}
are stated on the whole space $\R^d$, we impose the following uniform Gaussian-tail condition.

\begin{assump}[Uniform Gaussian tails]\label{assump:sampletail}
There exist constants $a_{\rm tail},A_{\rm tail}>0$, independent of $d$, such that for every
coordinate $i\in[d]$ and every $R\geq 0$,
\begin{equation}\label{sampletail}
\mu\bigl(|x_i|\geq R\bigr)
\leq A_{\rm tail}e^{-a_{\rm tail}R^2},
\qquad
\nu\bigl(|y_i|\geq R\bigr)
\leq A_{\rm tail}e^{-a_{\rm tail}R^2}.
\end{equation}
\end{assump}

This assumption is satisfied by the Gaussian model in Subsection~\ref{Sec:ThmGauss}. It is also
satisfied in the local perturbation model of Subsection~\ref{subsec:mf} under smoothness
conditions \eqref{eq:onesite} and Assumption~\ref{as:mfstruct}. The proof below uses
Hoeffding's inequality and a tail-truncation estimate for local monomials; these are
stated as Lemmas~\ref{lem:hoeffding} and~\ref{lem:tailtrunc} in
Appendix~\ref{sec:useful-lemmas}.

We now state the whole-space sample complexity bound.

\begin{prop}[Moment sample complexity]\label{prop:samplecomplex}
Assume Assumption~\ref{assump:sampletail}. Suppose that we observe $N$ i.i.d.\ samples
$x^{(1)},\ldots,x^{(N)}\sim \mu$ and $N$ i.i.d.\ samples
$y^{(1)},\ldots,y^{(N)}\sim \nu$, and form all empirical moments appearing in the cluster
moment relaxation. For $\delta\in(0,1)$, with probability at least $1-\delta$,
\begin{multline}\label{samplemain}
\max\!\left\{
\max_{k\in[K]}\|\widehat{M}_k^x-M_k^x\|_\infty,\,
\max_{ij\in E(\G)}\|\widehat{M}_{ij}^x-M_{ij}^x\|_\infty,\,
\max_{k\in[K]}\|\widehat{M}_k^y-M_k^y\|_\infty,\,
\max_{ij\in E(\G)}\|\widehat{M}_{ij}^y-M_{ij}^y\|_\infty
\right\}
\\
\lesssim
\left[\log\!\left(\frac{8A_{\rm tail}dN}{\delta}\right)\right]^n
\sqrt{
\frac{\log M_{\rm mom}+\log(1/\delta)}{N}
}.
\end{multline}
Here $\|\cdot\|_\infty$ for moment matrices denotes the entrywise maximum norm, and the
hidden constant in \eqref{samplemain} depends only on
$n,r,a_{\rm tail},A_{\rm tail}$.
In particular, in the sparse fixed-degree regime defined after \eqref{sampleMmom},
\eqref{samplemain} gives
\begin{multline}\label{samplemain-sparse}
\max\!\left\{
\max_{k\in[K]}\|\widehat{M}_k^x-M_k^x\|_\infty,\,
\max_{ij\in E(\G)}\|\widehat{M}_{ij}^x-M_{ij}^x\|_\infty,\,
\max_{k\in[K]}\|\widehat{M}_k^y-M_k^y\|_\infty,\,
\max_{ij\in E(\G)}\|\widehat{M}_{ij}^y-M_{ij}^y\|_\infty
\right\}
\\
\lesssim
\left[\log\!\left(\frac{8A_{\rm tail}dN}{\delta}\right)\right]^n
\sqrt{
\frac{\log d+2\log\binom{r+n}{n}+\log(1/\delta)}{N}
}.
\end{multline}
In \eqref{samplemain-sparse}, the hidden constant also absorbs the fixed constants in the
sparse fixed-degree regime.
\end{prop}

The proof of Proposition~\ref{prop:samplecomplex} is given in
Appendix~\ref{sec:proof}. 

\paragraph{A generic SDP form.}
We will use a generic stability estimate for SDP values in the moment
sample-complexity bound. Let $b$ collect the affine constraints whose
right-hand side is estimated from samples. The SDP can be written as
\begin{equation}\label{sdpstab-primal}
p(b):=
\inf_M\Big\{
\langle C,M\rangle:
\A(M)=b,\ \B(M)=0,\ M\succeq 0
\Big\},
\end{equation}
where $\B(M)=0$ denotes consistency constraints and any other affine constraints whose
right-hand side is not estimated from samples. Its dual is
\begin{equation}\label{sdpstab-dual}
\sup_{\lambda,\gamma}\Big\{
\langle b,\lambda\rangle:
C-\A^*(\lambda)-\B^*(\gamma)=S,\ S\succeq 0
\Big\}.
\end{equation}
The SDP perturbation estimate used below is stated and proved as
Lemma~\ref{lem:sdp-perturb} in Appendix~\ref{sec:useful-lemmas}.

We now translate the moment concentration bound into a concentration bound for the value
of the SDP relaxation. Let $b$ collect all prescribed $x$- and $y$-moment constraints,
let $\widehat b$ denote the empirical version, and write the cluster moment SDP in the
generic form \eqref{sdpstab-primal}. The deterministic step is exactly the perturbation
bound in Lemma~\ref{lem:sdp-perturb}.

\begin{theo}[Sample complexity for the SDP value]\label{thm:sdpvalue}
Assume the setting of Proposition~\ref{prop:samplecomplex}. Set
\begin{equation}\label{sdpstab-etaN}
\eta_N:=
\left[\log\!\left(\frac{8A_{\rm tail}dN}{\delta}\right)\right]^n
\sqrt{
\frac{\log M_{\rm mom}+\log(1/\delta)}{N}
}.
\end{equation}
Let $\eta>0$ and suppose that strong duality holds and that the dual optimum is attained for
\eqref{sdpstab-primal}--\eqref{sdpstab-dual}
for every right-hand side $\widetilde b$ satisfying $\|\widetilde b-b\|_\infty\le \eta$.
For such $\widetilde b$, let $\Lambda^\star(\widetilde b)$ be the set of optimal dual
multipliers $\lambda$ associated with the prescribed moment constraint
$\A(M)=\widetilde b$, and assume that
\begin{equation}\label{sdpstab-Rdual}
R_{\rm dual}:=
\sup_{\|\widetilde b-b\|_\infty\le \eta}
\inf_{\lambda\in\Lambda^\star(\widetilde b)}
\|\lambda\|_1
<\infty.
\end{equation}
If $\eta_N$ is sufficiently small compared with $\eta$, where the sufficient factor depends
only on the constants hidden in \eqref{samplemain}, then, with probability at least
$1-\delta$,
\begin{equation}\label{sdpstab-main}
|p(\widehat b)-p(b)|
\lesssim
R_{\rm dual}\eta_N.
\end{equation}
Consequently, up to the dual-sensitivity constant $R_{\rm dual}$, the empirical SDP value
has the same statistical rate as the empirical moment constraints.
More generally, suppose the local dual multipliers are uniformly bounded entrywise,
so that $R_{\rm dual}=O(K+|E(\G)|)$ when the local basis size is fixed. Then
\eqref{sdpstab-main}, \eqref{samplemain}, and \eqref{sampleMmom} yield
\begin{multline}\label{sdpstab-main-graph}
|p(\widehat b)-p(b)|
\lesssim
(K+|E(\G)|)\,
\left[\log\!\left(\frac{8A_{\rm tail}dN}{\delta}\right)\right]^n
\\
\times
\sqrt{\frac{\log(K+|E(\G)|)+2\log\binom{r+n}{n}+\log(1/\delta)}{N}}.
\end{multline}
Thus the statistical error depends explicitly on the number of local blocks and
graph edges. In the sparse fixed-degree regime this reduces to the previous
linear-in-$d$ bound.
\end{theo}

\begin{proof}
By Proposition~\ref{prop:samplecomplex}, with probability at least $1-\delta$,
\begin{equation}\label{sdpstab-pf0}
\|\widehat b-b\|_\infty\lesssim \eta_N.
\end{equation}
On this event, the smallness assumption on $\eta_N$ implies that $\widehat b$ lies in the
neighborhood where strong duality holds, dual optima are attained, and the dual multiplier
bound \eqref{sdpstab-Rdual} applies. Lemma~\ref{lem:sdp-perturb}, with
$b_1=\widehat b$, $b_2=b$, and $R=R_{\rm dual}$, gives
\begin{equation}
|p(\widehat b)-p(b)|
\le
R_{\rm dual}\|\widehat b-b\|_\infty.
\notag
\end{equation}
Combining this with \eqref{sdpstab-pf0} proves \eqref{sdpstab-main}.
\end{proof}

\begin{rem}[Interpreting the graph factor]\label{rem:sdp-relative-W2}
When $R_{\rm dual}=O(K+|E(\G)|)$ and the quadratic OT value satisfies $p(b)=\Theta(d)$,
Theorem~\ref{thm:sdpvalue} gives the relative estimate
\begin{equation}
\frac{|p(\widehat b)-p(b)|}{p(b)}
\lesssim
\frac{K+|E(\G)|}{d}\,\eta_N.
\notag
\end{equation}
Together with \eqref{samplemain}, this gives an explicit dependence on the
edge density of the reference graph. In particular, if $K=O(d)$ and
$|E(\G)|=O(d)$, then $N=O(\epsilon^{-2})$ for relative accuracy $\epsilon$ of
$W_2^2$, up to logarithmic factors of $d$.
\end{rem}

}

\section{Extracting a transport map.}\label{Sec:extmap}
We now explain how to obtain an approximate transport map from the cluster moment relaxation. 
Consider $\X=\Y=\R^d$ and the quadratic cost $c(x,y)=\|x-y\|^2$. The classical dual of \eqref{OT} is
\begin{equation}\label{OTd}
\sup_{f,g} \Bigg\{ 
   \int_\X f(x)\,{\rm d}\mu(x) 
 + \int_\Y g(y)\,{\rm d}\nu(y) 
 : \ \|x-y\|^2 - f(x) - g(y) \geq 0 
 \Bigg\}.
\end{equation}
By Brenier's theorem \cite{brenier1991polar} (see also \cite[Theorem~10.28]{villani2008optimal}), 
if $\mu$ and $\nu$ are absolutely continuous, the optimal transport plan is induced by a transport map
\begin{equation}\label{optmap}
T(x) = x - \tfrac{1}{2}\nabla f(x),
\end{equation}
where $(f,g)$ are the optimal solutions of \eqref{OTd}, which are known as Kantorovich potentials. To connect this with our relaxation, recall that \eqref{OTmom-2} can be written in standard SDP form as
\begin{equation}\label{OTmom-2r}
\min_{M} \Big\{ \langle C,M \rangle :
   \ \A^x(M)=b^x,\ 
   \A^y(M)=b^y,\ 
   \A^{\rm c}(M)=0,\ 
   M\succeq 0 \Big\}.
\end{equation}
Here, $C$ is block diagonal with $[C]_{k,k}=C_k$ for each $k\in[K]$. 
The constraints $\A^x(M)=b^x$ and $\A^y(M)=b^y$ represent the affine OT marginal 
constraints \eqref{cmomsimk} and \eqref{cmomsimij} in the $x$- and $y$-coordinates, respectively. Finally, $\A^{\rm c}(M)=0$ encodes the consistency constraints, 
{requiring certain groups of entries in $M$ to coincide; recall the definition of consistency in Section~\ref{Sec:mom}.} The dual of \eqref{OTmom-2r} is
\begin{equation}\label{OTmom-2d}
\max_{\lambda}\Big\{ 
   \langle b^x,\lambda^x \rangle 
 + \langle b^y,\lambda^y \rangle : \ 
   C - \A^{x*}(\lambda^x) - \A^{y*}(\lambda^y) - \A^{{\rm c}*}(\lambda^{\rm c}) = S,\ 
   S \succeq 0 \Big\},
\end{equation}
where $\lambda^x,\lambda^y,\lambda^{\rm c}$ are the Lagrange multipliers of the affine constraints. We now reformulate \eqref{OTmom-2d} as an approximation of the dual formulation 
\eqref{OTd}. Let $\Phi=[\Phi_1;\Phi_2;\cdots;\Phi_K].$ From \eqref{consmomk}, \eqref{consmomij}, we have
\begin{equation}\label{Sep_14_1}
b^x = \mu\(\A^x(\Phi\Phi^\top)\), 
\qquad 
b^y =  \nu\(\A^y(\Phi\Phi^\top)\).
\end{equation}
Moreover, from \eqref{cdec} and \eqref{eqck}, we have that
\begin{equation}\label{Sep_14_2}
\langle C,\Phi\Phi^\top \rangle 
= \sum_{k\in[K]} \langle C_k,\Phi_k\Phi_k^\top \rangle
= \sum_{k\in[K]} c_k(z_k) 
= c(z) = \|x-y\|^2.
\end{equation}
Finally,
\begin{equation}\label{Sep_14_3}
\langle \A^{x*}(\lambda^x) + \A^{y*}(\lambda^y) + \A^{{\rm c}*}(\lambda^{\rm c}), 
        \Phi\Phi^\top \rangle
= \langle \lambda^x,\A^x(\Phi\Phi^\top)\rangle
+ \langle \lambda^y,\A^y(\Phi\Phi^\top)\rangle,
\end{equation}
where we used the fact that $\A^{\rm c}(\Phi\Phi^\top)=0$. From \eqref{Sep_14_1}–\eqref{Sep_14_3}, the dual problem \eqref{OTmom-2d} can be rewritten as
\begin{align}\label{OTmom-2d1}
\max_{\lambda}\ & \langle \lambda^x,\mu\(\A^x(\Phi\Phi^\top)\)\rangle\,
+ \langle \lambda^y,\nu\(\A^y(\Phi\Phi^\top)\)\rangle\, \\
{\rm s.t.}\ & \|x-y\|^2 
- \langle \lambda^x,\A^x(\Phi\Phi^\top)\rangle
- \langle \lambda^y,\A^y(\Phi\Phi^\top)\rangle
= \langle S,\Phi\Phi^\top\rangle,\quad S\succeq 0. \notag
\end{align}
The term $\{\langle S,\Phi\Phi^\top\rangle : S\succeq 0\}$ is precisely the cone of 
\emph{sum-of-squares (SOS)} polynomials in the basis $\Phi$. 
Defining 
$\hat{f}(x):=\langle \lambda^x,\A^x(\Phi\Phi^\top)\rangle$ and 
$\hat{g}(y):=\langle \lambda^y,\A^y(\Phi\Phi^\top)\rangle$, 
we obtain the equivalent formulation
\begin{equation}\label{OTmom-2d2}
\sup_{\hat{f},\hat{g}}
\Bigg\{ \int_\X \hat{f}(x)\,{\rm d}\mu(x) + \int_\Y \hat{g}(y)\,{\rm d}\nu(y) :
\ \|x-y\|^2 - \hat{f}(x) - \hat{g}(y)\ \text{is SOS} \Bigg\}.
\end{equation}
Comparing \eqref{OTd} and \eqref{OTmom-2d2}, we see that the dual SDP is an SOS relaxation of the OT dual: 
the potentials are restricted to the span of ${\rm R}_x(\Phi\Phi^\top)$ and ${\rm R}_y(\Phi\Phi^\top)$, 
and the nonnegativity constraint is replaced by an SOS certificate. 
From the approximate potential $\hat{f}$, we can construct the transport map
\begin{equation}\label{appmap}
\widehat{T}(x) = x - \tfrac{1}{2}\nabla \hat{f}(x),
\end{equation}
which serves as an approximation of the true map \eqref{optmap}. 

Although we focus here on the simplified cluster moment relaxation \eqref{OTmom-2}, 
the same procedure applies to the full and sparse versions \eqref{OTmom} and \eqref{OTmom-1}.

\section{Preprocessing}\label{Sec:prepro}

{In practice, the coordinates of $\mu$ and $\nu$ may not be aligned a priori, and the cluster structure used in our marginal and cluster moment relaxations may also be unknown. We therefore introduce a data-driven preprocessing stage before solving \eqref{OTmr}--\eqref{OTmr-2} or \eqref{OTmom}--\eqref{OTmom-2}. These steps are heuristic and are intended to improve practical performance. When structural information is available from the application, this preprocessing can be omitted.

\subsection{Coordinate alignment}\label{subsec:prematch}

For continuous distributions, when the coordinates of $\mu$ and $\nu$ are not aligned, we apply a linear transformation to match their first and second moments. Let $m_\mu, \Sigma_\mu$ and $m_\nu, \Sigma_\nu$ denote the empirical means and covariances. We transform samples $x \sim \mu$ via
\begin{equation}\label{eq:pre-linear}
T(x) = m_\nu + \Sigma_\nu^{1/2} \Sigma_\mu^{-1/2} (x - m_\mu),
\end{equation}
so that the transformed distribution $T_\# \mu$ has the same mean and covariance as $\nu$. While this alignment only matches second-order statistics, it provides a simple and effective preprocessing step that reduces mismatch between $\mu$ and $\nu$.

For discrete models such as Ising distributions, the variables already have coordinate-wise meaning, and the appropriate operation is relabeling rather than a linear transformation. In this case, we align coordinates via a permutation that matches empirical second-order structure. Let $\widehat\Sigma_\mu,\widehat\Sigma_\nu\in \S^d$ be the empirical covariance matrices and $\widehat m_\mu,\widehat m_\nu\in \R^d$ be the empirical means. We seek a permutation matrix $P$ solving
\begin{equation}\label{eq:pre-perm}
P^\star \in \arg\min_{P\in \Pi_d}
\big\|P^\top \widehat\Sigma_\mu P-\widehat\Sigma_\nu\big\|_F^2
+ \lambda \big\|P^\top \widehat m_\mu-\widehat m_\nu\big\|_2^2,
\end{equation}
where $\Pi_d$ denotes the set of $d\times d$ permutation matrices and $\lambda\geq 0$ is optional. We then relabel one distribution according to $P^\star$ before constructing the relaxation. This formulation is a quadratic assignment problem and is NP-hard in general \cite{cela2013quadratic}. In practice, approximate solutions can be obtained via local search methods or by relaxing $\Pi_d$ to its convex hull (the Birkhoff polytope).

\subsection{Estimating cluster structure}\label{subsec:precluster}

After alignment, we estimate a dependence graph and use it to define both the clusters in Definition~\ref{itm:first} and the reference graph $\G$ in Definition~\ref{itm:third}. Since our relaxations are most effective when inter-cluster dependence is weak, clustering should reflect statistical dependence rather than geometric proximity.

We construct a weighted graph based on empirical second-order statistics. Let $\widehat\Sigma_\mu$ and $\widehat\Sigma_\nu$ denote the covariance matrices of the aligned variables, and let $\widehat\Theta_\mu \approx \widehat\Sigma_\mu^{-1}$ and $\widehat\Theta_\nu \approx \widehat\Sigma_\nu^{-1}$ be corresponding estimates of the precision matrices. For Gaussian models, zeros in the precision matrix correspond to conditional independence, so sparsity in $\Theta$ naturally encodes the underlying graphical structure \cite{wainwright2008graphical}. We define the edge weights by
\begin{equation}\label{eq:pre-weight}
W_{ij}=\max\big\{|(\widehat\Theta_\mu)_{ij}|,\ |(\widehat\Theta_\nu)_{ij}|\big\}.
\end{equation}
We then connect coordinates $i$ and $j$ whenever $W_{ij}\geq \tau$ for a prescribed threshold $\tau>0$, yielding a coordinate-level interaction graph.

Let $C_1,\ldots,C_K$ denote the connected components (or communities) of this graph. We define the clusters by
\begin{equation}\label{eq:pre-clusters}
x_k=(x_i)_{i\in C_k}, \qquad
y_k=(y_i)_{i\in C_k}, \qquad k\in [K].
\end{equation}
We then define the reference graph $\G$ on $[K]$ by connecting two clusters whenever there exists a sufficiently strong inter-cluster interaction:
\begin{equation}\label{eq:pre-graph}
k\ell \in E(\G)
\quad \Longleftrightarrow \quad
\max_{i\in C_k,\, j\in C_\ell} W_{ij} \geq \tau_{\rm inter}.
\end{equation}

The output of preprocessing is therefore: a common representation of $\mu$ and $\nu$, a partition of the coordinates into clusters, and a sparse graph $\G$ (or its neighborhood enlargement $G^h$) used in the relaxations. In the numerical experiments of Section~\ref{Sec:numer}, the underlying structures are prescribed in advance, so this preprocessing step is not required; nevertheless, the above procedure provides a practical approach when the structure is unknown.}

\section{Numerical Experiments}\label{Sec:numer}

In this section, we present numerical experiments to demonstrate the effectiveness of our convex relaxation approaches for high-dimensional OT. We use \textsc{MOSEK} \cite{aps2019mosek} to solve the SDP problems and LP problems. All algorithms are implemented in {\sc Matlab} R2024a and executed on a MacBook equipped with an M3 Max chip and 64GB of RAM.

\subsection{Toy Gaussian example}\label{subsec:GS}

{In this section, we evaluate our cluster moment relaxation method for \eqref{OT} on Gaussian distributions. Although the $W_2$ distance between Gaussian measures admits a closed-form expression, we use it as a benchmark due to the availability of exact ground truth. This enables a precise assessment of the accuracy of different optimal transport methods.}

We first test the cluster moment relaxation for \eqref{OT} between Gaussian distributions 
$\mathcal{N}(m_1,\Sigma_1)$ and $\mathcal{N}(m_2,\Sigma_2)$ to verify Theorem~\ref{thmexp}. 
The means $m_1,m_2$ are sampled i.i.d.\ from a standard normal distribution, 
and the precision matrices $\Sigma_1^{-1},\Sigma_2^{-1} \in \S^d_{++}$ are chosen to be sparse, 
diagonally dominated matrices with a path sparsity pattern $G$. 
Off-diagonal entries are drawn i.i.d.\ from a standard normal distribution, 
while diagonal entries are set as 
\begin{equation}
\Sigma_{1,ii}^{-1}=0.1+\sum_{k\neq i}|\Sigma_{1,ik}^{-1}|, 
\qquad
\Sigma_{2,ii}^{-1}=0.1+\sum_{k\neq i}|\Sigma_{2,ik}^{-1}|.
\end{equation}
We compute the first and second moments using explicit Gaussian formulas.  
This is generally infeasible for non-Gaussian distributions, but here we use it 
solely to verify Theorem~\ref{thmexp}. We set the tolerance of MOSEK to be $10^{-12}$ to compute a highly accurate SDP solution. 

We set the reference graph $\G=G^h$ (Definition~\ref{itm:third}) for $h\in\N^+$ and solve the relaxation \eqref{GSmom} with $d=100$. 
The exact value ${\rm opt}_{\rm exact}$ is obtained from \eqref{W2GGM} and compared with the lower bound 
${\rm opt}_{\rm SDP}$, using the relative error 
\begin{equation}
\frac{|{\rm opt}_{\rm exact}-{\rm opt}_{\rm SDP}|}{{\rm opt}_{\rm exact}}.
\end{equation}

\begin{figure}[ht]
\centering
\includegraphics[width=0.3\linewidth]{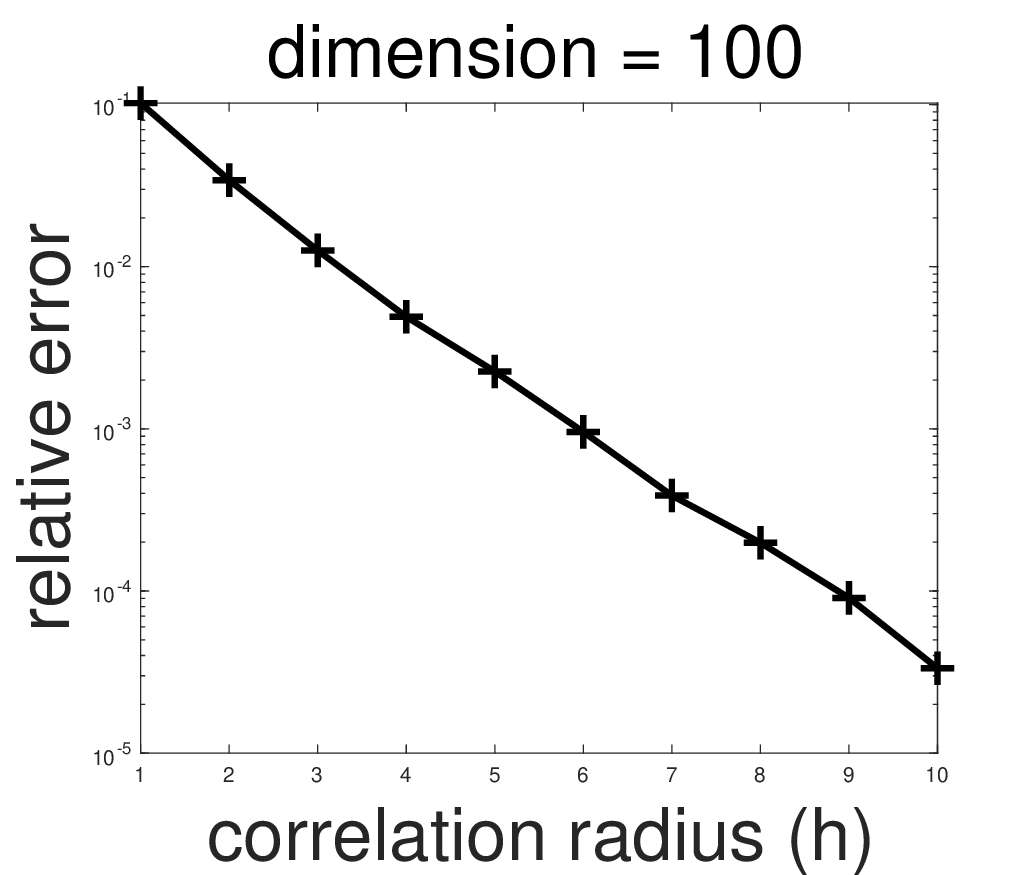}
\caption{\small Cluster moment relaxation \eqref{GSmom} for OT between Gaussian distributions. The correlation sparsity pattern $G$ of Gaussian distributions is a path. The reference graph $\G$ is chosen as $G^h$ for various connectivity radius $h$ (see \eqref{defipower}). } 
\label{fig-GGM}
\end{figure}

Figure~\ref{fig-GGM} shows that the cluster moment relaxation exhibits exponential convergence for sparse Gaussian distributions, consistent with \eqref{thmine} in Theorem~\ref{thmexp}.

{Next, we compare our cluster moment relaxation with the vanilla OT solver based on the sampling approach and with the ICNN method on the same Gaussian benchmark. Specifically, we sample $N$ data points from $\mu$ and $\nu$, replace them with their empirical distributions, and solve the resulting discrete OT problem. This leads to an LP with a decision variable of size $N\times N$. We solve the resulting LP using the Sinkhorn algorithm~\cite{cuturi2013sinkhorn}, a highly efficient method for entropy-regularized optimal transport. Since Sinkhorn computes an entropy-regularized solution, its output differs from that of the unregularized OT problem. To balance accuracy and numerical stability, we set the entropy parameter to $0.01$ times the median of the cost matrix and use a tolerance of $10^{-4}$. These parameters are used throughout all experiments.

For fairness, we also evaluate the moments in our cluster moment relaxation method from the same samples (rather than using Gaussian formulas), and we set the connectivity radius $h=5$ in the reference graph $\G=G^h$ (see \eqref{defipower}) of \eqref{GSmom}. Following Proposition~\ref{propGS}, we apply chordal conversion to transfer the problem \eqref{GSmom} into a multi-block SDP problem with small dimensionality using the {\sc Matlab} code \href{http://www.opt.c.titech.ac.jp/kojima/SparseCoLO/SparseCoLO.htm}{SparseCoLO} \cite{kim2011exploiting}. For the SDP problems, we employ MOSEK with its default tolerance $10^{-7}$.

For the ICNN baseline, we use the Brenier-map formulation implemented in the OT-ICNN code \cite{makkuva2020optimal}. We train a pair of convex neural networks by alternating minimization, using three layers, width $128$ in all Gaussian experiments. 

\begin{figure}[ht]
\centering
\includegraphics[width=\linewidth]{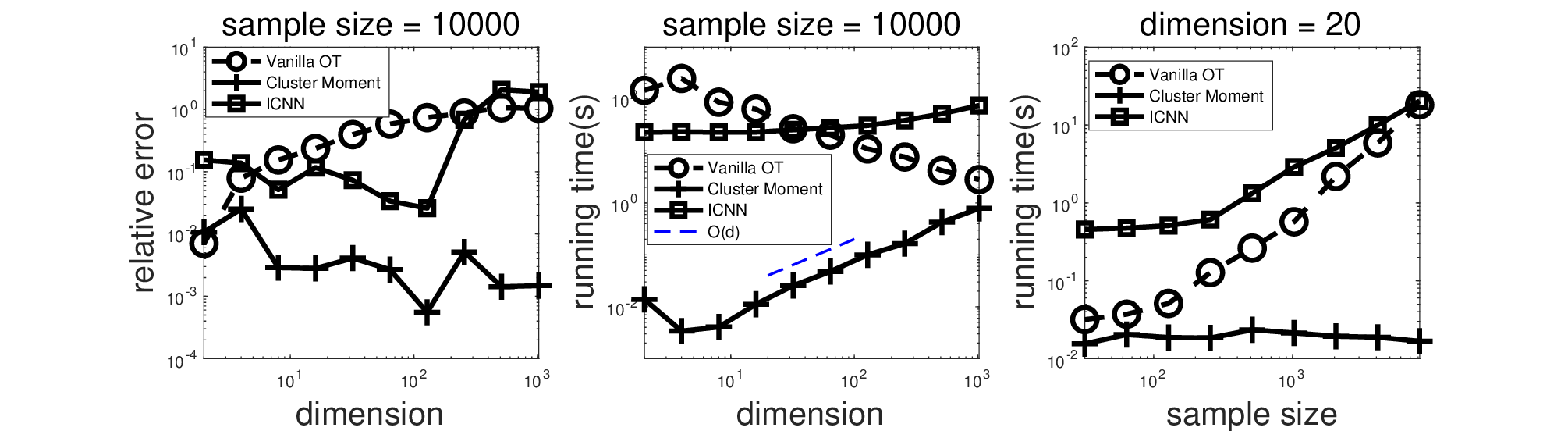}
\caption{\small Comparison of the vanilla OT solver, the cluster moment relaxation \eqref{OTmom-2}, and the ICNN method for \eqref{OT} between Gaussian distributions. The reference graph for \eqref{OTmom-2} is $\G=G^h$ (see \eqref{defipower}), where $G$ is a path graph and the connectivity radius is fixed as $h=5$.}
\label{Fig:GGMS}
\end{figure}

Figure~\ref{Fig:GGMS} illustrates that the cluster moment relaxation achieves both higher accuracy and greater computational efficiency on this Gaussian benchmark.

The first panel shows that the relative error of the vanilla OT solver grows rapidly with the problem dimension, whereas that of the cluster moment relaxation remains around or below $10^{-2}$ throughout the whole sweep. This dimension-stable relative error is consistent with Remark~\ref{rem:sdp-relative-W2}: the plotted quantity is the relative error of the quadratic OT value $W_2^2$, whose scale is $\Theta(d)$ on this additive Gaussian benchmark. The ICNN is more accurate than the vanilla OT solver in several low- and moderate-dimensional cases, but its performance becomes unstable in larger dimensions and is still clearly less reliable than the cluster moment relaxation. This demonstrates that our method exhibits much smaller sample complexity than the vanilla OT solver, which suffers from the curse of dimensionality, while also being more robust than the neural network baseline on this benchmark.

The second panel reports the running times of the three methods for fixed sample size and increasing dimension. The vanilla OT solver becomes faster in higher dimensions because, with fixed $N$, the discrete OT cost matrix becomes more concentrated as $d$ grows, producing a better-conditioned LP on which the Sinkhorn method converges in fewer iterations. Nevertheless, our cluster moment relaxation remains much faster across all tested dimensions. In contrast, the ICNN training time is consistently much larger than the running time of the SDP relaxation.

The third panel shows that, for fixed dimension, the cluster moment relaxation method is substantially faster than both the vanilla OT solver and the ICNN method as the sample size grows. The vanilla OT approach requires solving an entropy-regularized LP with $\Omega(N^2)$ variables and $\Omega(N)$ constraints, and each Sinkhorn iteration costs $\mathcal{O}(N^2)$. In contrast, the SDP~\eqref{GSmom} has a fixed matrix size $2d+1$ and $\mathcal{O}(hd)$ constraints, which do not depend on the sample size. The ICNN method also becomes more expensive as the sample size increases, due to repeated stochastic-gradient updates over the training samples; when $N$ is very small, the effective batch size is reduced accordingly so that the network is still genuinely trained.}

{
\subsection{Product Beta measures}

To further assess the performance of our method beyond the Gaussian setting, we consider \eqref{OT} between two non-Gaussian product measures on $\R^d$. The source distribution $\mu$ is chosen as the product of $d$ identical one-dimensional affine-Beta distributions on $[-2,2]$ with parameters $\mathrm{Beta}(1.4,5.2)$, while the target distribution $\nu$ is chosen as the product of $d$ identical affine-Beta distributions on $[-2,2]$ with parameters $\mathrm{Beta}(5.0,1.8)$. This provides a bounded and genuinely non-Gaussian benchmark for which the exact OT cost is still available.

Recall that the Beta distribution on $[0,1]$ with parameters $\alpha,\beta>0$ has density
\begin{equation}
\rho_{\alpha,\beta}(t)=\frac{1}{B(\alpha,\beta)}\, t^{\alpha-1}(1-t)^{\beta-1}, 
\qquad t\in [0,1],
\end{equation}
where $B(\alpha,\beta)$ is the Beta function. In our experiment, we map this distribution affinely from $[0,1]$ to $[-2,2]$. Thus, if $z\sim \mathrm{Beta}(\alpha,\beta)$, then the corresponding sample on $[-2,2]$ is given by
\begin{equation}
x=-2+4z.
\end{equation}

Since both $\mu$ and $\nu$ are product measures and the cost is quadratic, the exact OT map is coordinatewise and the exact OT cost is the sum of $d$ one-dimensional OT costs. We therefore estimate the ground truth by generating a large number of one-dimensional samples from the source and target marginals, sorting them, and averaging the squared differences. The total reference cost is then obtained by multiplying the resulting one-dimensional cost by $d$.

We compare three methods: the vanilla OT solver, the cluster moment relaxation, and the ICNN approach. For the vanilla OT solver, we draw $N$ samples from $\mu$ and $\nu$, replace them by their empirical distributions, and solve the resulting discrete OT problem using the Sinkhorn algorithm, exactly as in Subsection~\ref{subsec:GS}. We use the same Sinkhorn implementation as before, with entropy parameter equal to $0.01$ times the median of the empirical cost matrix and stopping tolerance $10^{-4}$. For the cluster moment relaxation, we use the mean-field structure with singleton clusters, namely $K=d$ and a reference graph with $E(\G)=\varnothing$, and fix the relaxation degree to be $n=3$. Since the problem is mean-field, this corresponds to choosing cluster size one and no inter-cluster edges. The SDP is solved by {\sc MOSEK} with tolerance $10^{-7}$. 

For the ICNN baseline, we use the Brenier-map formulation implemented in the OT-ICNN code. We train a pair of convex neural networks by alternating minimization, using Adam with learning rate $10^{-4}$, batch size $256$, and inner iteration number $4$ or $6$ depending on the dimension. The ICNNs have three layers throughout, and we use width $128$ in all experiments. 

In the first two panels of Figure~\ref{Fig:BetaProduct}, we fix the sample size at $N=10^4$ and vary the dimension. The dimensions tested are $d=2,4,8,16,32,64,128,256,512.$ The first panel shows the relative error with respect to the exact OT cost. The second panel reports the running time as the dimension increases. In the third panel, we fix the dimension at $d=32$ and vary the sample size $N=512,\,1024,\,2048,\,4096,\,8192,\,10^4.$

\begin{figure}[!htbp]
\centering
\includegraphics[width=0.9\linewidth]{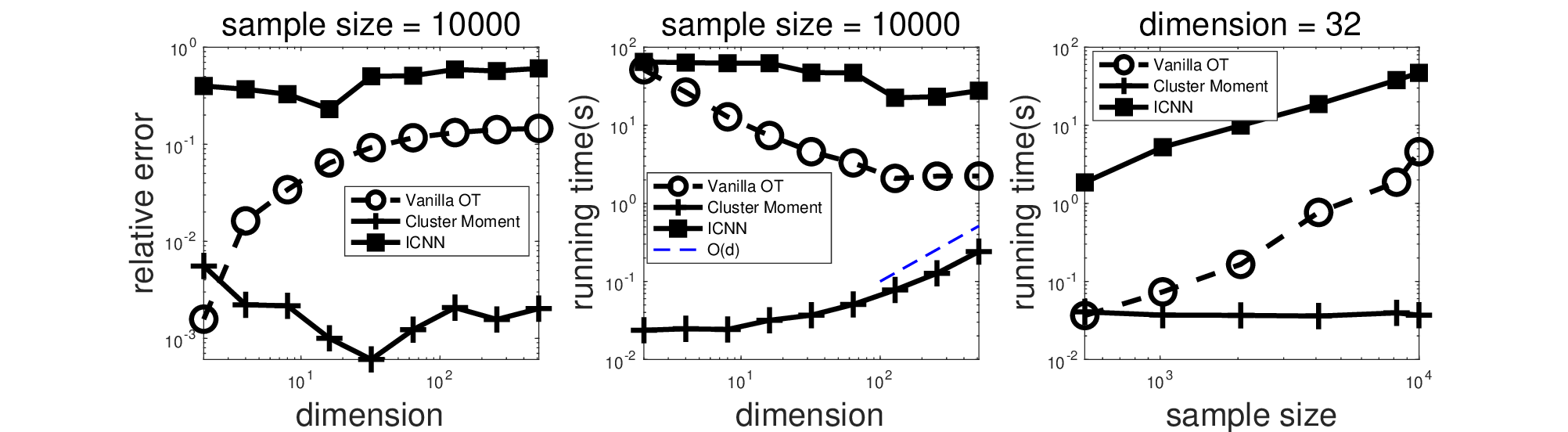}
\caption{\small Comparison of the Sinkhorn-based vanilla OT solver, the cluster moment relaxation, and the ICNN method for \eqref{OT} between two non-Gaussian product measures. The source is the product of $d$ affine-Beta distributions on $[-2,2]$ with parameters $\mathrm{Beta}(1.4,5.2)$, and the target is the product of $d$ affine-Beta distributions on $[-2,2]$ with parameters $\mathrm{Beta}(5.0,1.8)$. For the cluster moment relaxation, we use singleton clusters, reference graph $\G$ to be empty graph, and relaxation degree $n=3$. The vanilla OT solver uses Sinkhorn on the empirical discretization, while the ground truth is computed from the exact one-dimensional decomposition of the product problem.}
\label{Fig:BetaProduct}
\end{figure}

Figure~\ref{Fig:BetaProduct} shows that the cluster moment relaxation remains highly accurate on this non-Gaussian benchmark. Its relative error stays around $10^{-3}$ throughout the entire dimension sweep, even though the distributions are no longer Gaussian. This is the product-measure analogue of Remark~\ref{rem:sdp-relative-W2}: the exact quadratic OT cost is additive across coordinates and therefore grows like $\Theta(d)$, so a fixed level of moment-estimation accuracy leads to a nearly dimension-independent relative error. In contrast, the Sinkhorn-based vanilla OT solver becomes increasingly inaccurate as the dimension grows, with the relative error rising from $1.6\times 10^{-3}$ at $d=2$ to about $1.5\times 10^{-1}$ at $d=512$. The ICNN is still substantially less accurate than the cluster moment relaxation on this benchmark. This suggests that, for this mean-field product benchmark, the convex relaxation captures the structure of the problem much more effectively than the neural network baseline.

The second panel shows that the running time of the cluster moment relaxation scales nearly linearly with the dimension and remains well below one second up to $d=512$. The vanilla OT solver based on Sinkhorn is substantially slower for small and moderate dimensions because it solves a large empirical OT problem with an $N\times N$ transport matrix. The third panel shows the dependence on the sample size for fixed dimension $d=32$. The cluster moment relaxation is essentially insensitive to the sample size, since the SDP size is determined by the relaxation structure rather than by $N$, whereas the Sinkhorn baseline becomes increasingly expensive as the empirical OT problem grows. The ICNN method is again the most time-consuming due to neural network training.

These results complement the Gaussian experiments in Subsection~\ref{subsec:GS}. They show that the strong empirical performance of the cluster moment relaxation is not limited to the Gaussian setting: even for a non-Gaussian product benchmark, the method remains highly accurate and computationally efficient, while the Sinkhorn-based vanilla OT solver deteriorates as the dimension grows.

}

\FloatBarrier

\subsection{Ising Model}

We next test the marginal relaxation for \eqref{OT} between Ising models:
\begin{align}\label{Ising}
&\mu \sim \exp\!\left[ \beta_1\!\left( J_1\sum_{ij\in E(G)}u_iu_j+h_1\sum_{k\in [d]}u_k \right) \right],\notag \\
&\nu \sim \exp\!\left[ \beta_2\!\left( J_2\sum_{ij\in E(G)}v_iv_j+h_2\sum_{k\in [d]}v_k \right) \right],
\end{align}
where $u_k,v_k \in \{-1,1\}$ and $G$ is a path graph.  
For any $\omega\in [d]$, we let $K=\lceil d/\omega\rceil$ and define clusters
\begin{equation}\label{defiomegx}
x_k=\begin{cases}
\{u_{(k-1)\omega+1},\ldots,u_{k\omega}\} & k<K \\
\{u_{(k-1)\omega+1},\ldots,u_{d}\} & k=K
\end{cases}
\end{equation}
\begin{equation}\label{defiomegy}
y_k=\begin{cases}
\{v_{(k-1)\omega+1},\ldots,v_{k\omega}\} & k<K \\
\{v_{(k-1)\omega+1},\ldots,v_{d}\} & k=K
\end{cases}
\end{equation}
We apply the sparse marginal relaxation \eqref{OTmr-1} with the reference graph $\G$ (Definition~\ref{itm:third}) chosen as a path graph 
on $[K]$, i.e., edges $\{i(i+1): i\in [K-1]\}$.  
We set $d=12$, so that the full density vectors $\mu,\nu$ (each with $2^{12}$ entries) 
can be stored explicitly and the exact OT cost computed.  
The marginals in \eqref{marmar} are computed exactly from the distributions $\mu$ and $\nu$. We set the tolerance of MOSEK for solving the LP problem \eqref{OTmr-1} as $10^{-8}.$ 

\begin{center}
\footnotesize
\renewcommand{\arraystretch}{0.92}
\captionof{table}{Marginal relaxation \eqref{OTmr-1} for \eqref{OT} between 1D Ising models. $\omega$ is the size of clusters defined in \eqref{defiomegx} and \eqref{defiomegy}. The reference graph $\G$ (Definition~\ref{itm:third}) is a path graph.}
\label{Test-Ising}
\begin{tabular}{|c|c|c|c|c|}
\hline
$(J_\mu,h_\mu,\beta_\mu)\to(J_\nu,h_\nu,\beta_\nu)$ & Method & OT cost & Relative error & time(s)  \\
\hline
$(1,0.2,0.6)\to(-1,0.2,0.6)$ & $\omega=1$  & 1.3218923e+01 &  0 & 3.90e-03   \\
& $\omega=2$ & 1.3218923e+01 & 0 & 8.38e-03  \\
& $\omega=3$ & 1.3218923e+01 & 0 &  7.03e-02 \\
& $\omega=4$ & 1.3218923e+01 & 0 & 6.09e-01  \\
& exact & 1.3218923e+01 & 0  & 5.41e+01\\
\hline
$(1,0.2,0.6)\to(2,0.2,0.44)$ & $\omega=1$  & 1.9077413e+00 &  3.20e-1 & 4.06e-03  \\
& $\omega=2$ & 2.5413490e+00 & 9.41e-2  & 1.05e-02\\
& $\omega=3$ & 2.6937730e+00 & 3.98e-2 & 1.01e-01\\
& $\omega=4$ & 2.7297483e+00 & 2.70e-2  &7.59e-01 \\
& exact & 2.8054410e+00 & 0 & 7.50e+01  \\
\hline
$(1,0.2,0.6)\to(1,0.2,0.2)$ & $\omega=1$  & 6.5223360e+00 &  6.20e-2 & 6.49e-03 \\
& $\omega=2$ & 6.9073375e+00 & 6.60e-3 & 1.02e-02\\
& $\omega=3$ & 6.9073375e+00 & 6.60e-3 &7.95e-02 \\
& $\omega=4$ & 6.9443953e+00 & 1.30e-3 &5.61e-01 \\
& exact & 6.9535336e+00 & 0 & 6.01e+01\\
\hline
\end{tabular}
\end{center}

Table~\ref{Test-Ising} shows that the lower bound obtained from the marginal relaxation converges to the exact OT cost as $\omega$ increases.  
In the first instance, the relaxation is already tight at $\omega=1$, demonstrating the effectiveness of the proposed approach even with the smallest clusters.  
Moreover, the running time for solving the LP arising from the marginal relaxation is several orders of magnitude smaller than that of the original problem~\eqref{OT}.  
This improvement stems from the fact that our marginal relaxation substantially reduces the number of variables and constraints in the LP formulation.

{Next, we compare our marginal relaxation with the vanilla OT solver and ICNN. To construct a nontrivial Ising benchmark with exact ground truth computable in high dimensions, we consider a block-product Ising model. The coordinates are partitioned into consecutive blocks of size $8$, and on each block we place a 1D path Ising model. The source and target distributions are products over these blocks, while the Ising parameters vary from block to block by cycling through the three parameter pairs in Table~\ref{Test-Ising}:
\begin{equation}
(1,0.2,0.6)\to (-1,0.2,0.6), \qquad
(1,0.2,0.6)\to (2,0.2,0.44), \qquad
(1,0.2,0.6)\to (1,0.2,0.2).
\notag
\end{equation}
Thus, each block remains a genuine interacting path Ising model, but the global law factorizes across blocks. Since the quadratic cost is additive across coordinates, the exact OT cost is the sum of exact blockwise OT costs. Each block OT problem involves at most $2^8$ states and can therefore be solved exactly by the same discrete LP formulation.

For the dimension sweep, we fix the sample size as $N=10^4$ and test $d=8,16,32,64,128,256,512.$ For the sample-size sweep, we fix $d=32$ and vary $N=512,1024,2048,4096,8192,10^4.$ The vanilla OT baseline is the Sinkhorn solver applied to the empirical discretization. For the marginal relaxation, we use path reference graph and cluster size $\omega=2$. For the ICNN baseline, we use three layers, width $128$.

\begin{figure}[ht]
\centering
\includegraphics[width=\linewidth]{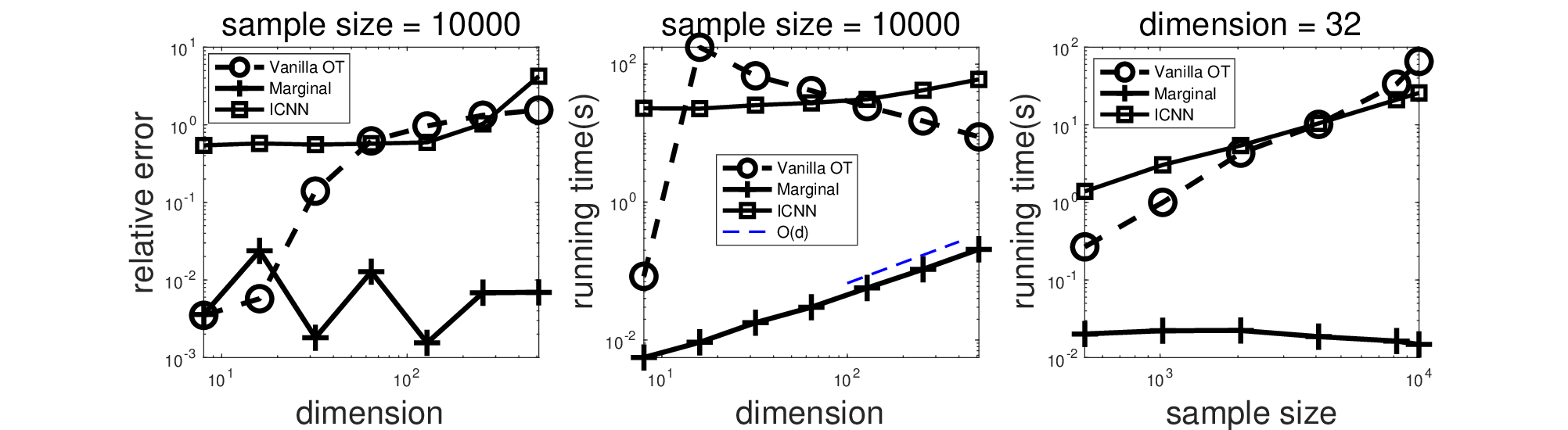}
\caption{\small Comparison of the Sinkhorn-based vanilla OT solver, the marginal relaxation, and the ICNN method for OT between block-product Ising models. Each block is a path Ising model on $8$ spins, and the block parameters cycle through the three exact Ising examples in Table~\ref{Test-Ising}. For the marginal relaxation, we use cluster size $\omega=2$ and a path reference graph on the clusters.}
\label{Fig:block-ising}
\end{figure}

Figure~\ref{Fig:block-ising} shows that the marginal relaxation remains highly accurate on this benchmark. In the dimension sweep, its relative error stays between about $1.5\times 10^{-3}$ and $1.3\times 10^{-2}$ across the entire range from $d=32$ to $d=512$, while the Sinkhorn baseline deteriorates rapidly and exceeds relative error $1$ for sufficiently large dimensions. The ICNN performs substantially worse than the marginal relaxation on every tested dimension, with relative error around $5\times 10^{-1}$ already at $d=8$ and growing further in the largest cases.

The running-time comparison also favors the marginal relaxation. In the dimension sweep, the Sinkhorn solver becomes cheaper as the dimension grows because the sample size is fixed, but it still remains significantly slower than the marginal relaxation over the whole range. The ICNN training time grows with the dimension and is much larger than the marginal relaxation time. In the sample-size sweep at $d=32$, the marginal relaxation remains nearly constant, around a few hundredths of a second, while the Sinkhorn solver grows quickly with the sample size and the ICNN becomes increasingly expensive due to repeated stochastic-gradient updates.

These results underscore the practical advantages of the marginal relaxation over vanilla OT solver.}

\subsection{Ginzburg--Landau Model}

We next test our cluster moment relaxation for the generative modeling task associated with the one-dimensional Ginzburg--Landau model
\begin{equation}\label{GLD}
\nu \sim \exp\Bigg[-\beta\Bigg(\sum_{i=1}^{d+1}\frac{\lambda}{2}\Big(\frac{y_i-y_{i-1}}{h}\Big)^2
+ \sum_{i=1}^{d}\frac{1}{4\lambda}(1-y_i^2)^2\Bigg)\Bigg], \quad y\in [-L,L]^d,
\end{equation}
with boundary conditions $y_0=y_{d+1}=0$, step size $h=1/(1+d)$, and $L=2.5$.

We first generate $N=10^4$ training samples from $\nu$ by tensor-train conditional sampling \cite{dolgov2020approximation,peng2024tensor}. From these samples we compute the mean $m$ and covariance $\Sigma$, define a Gaussian source distribution $\mu\sim\mathcal{N}(m,\Sigma)$, and generate another $N=10^4$ training samples from $\mu$. As before, we use these samples to build the transport map extracted from the cluster moment relaxation \eqref{OTmom-2}, with $K=d$, relaxation degree $n=10$, and a path graph as the reference graph $\G$ (Definition~\ref{itm:third}). After chordal conversion, the resulting multi-block SDP is solved by {\sc MOSEK} with tolerance $10^{-7}$.

{We compare our method with three baselines. The first is the normalizing-flow model \cite{papamakarios2017masked}, using five transforms, each implemented by a three-layer neural network of width~128. The second is an ICNN baseline trained on the same data with three layers, width $128$. The third is a Sinkhorn transport map: we solve the entropy-regularized OT problem on the empirical training clouds by Sinkhorn \cite{cuturi2013sinkhorn}, take the barycentric projection of the resulting transport plan on the source training samples, and extend it to new samples by $8$-nearest-neighbor regression.

Each method is then applied to another $10^5$ reference samples from $\mu$. We compare the resulting mapped samples with an independent reference set of $10^5$ samples from $\nu$ through the pairwise marginals $(y_1,y_2)$ and $(y_1,y_5)$.

We consider two parameter settings. The first is $d=10$, $\beta=1/8$, and $\lambda=0.03$; the second is $d=50$, $\beta=1/20$, and $\lambda=0.01$. Figures~\ref{Fig:GLD_1_2} and~\ref{Fig:GLD_1_20} show that the cluster moment relaxation gives the best visual match to the target marginals in both cases. The normalizing-flow baseline remains reasonably accurate but is more diffuse. The ICNN and the Sinkhorn transport map are substantially smoother, and this effect is especially pronounced in the $50$-dimensional example. These results demonstrate the effectiveness of the proposed cluster moment relaxation for generative modeling.

\begin{figure}[ht]
\centering
\includegraphics[width=\linewidth,height=0.4\linewidth]{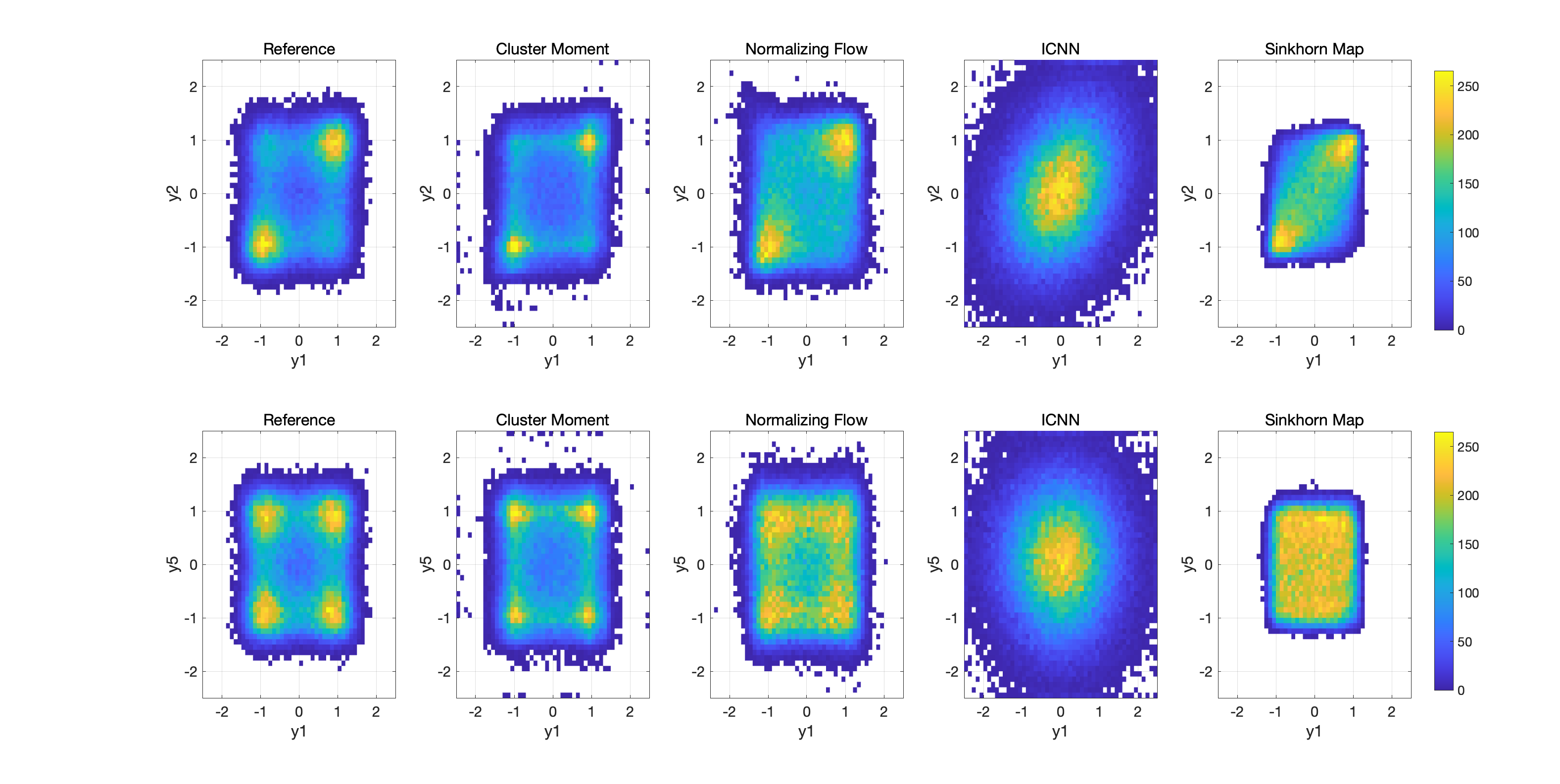}
\caption{\small Comparison of the cluster moment relaxation \eqref{OTmom-2}, the normalizing-flow baseline, the ICNN baseline, and the Sinkhorn transport map for the one-dimensional Ginzburg--Landau model with $d=10$, $\beta=1/8$, $L=2.5$, and $\lambda=0.03$. The displayed marginals are $(y_1,y_2)$ and $(y_1,y_5)$. For the cluster moment relaxation, we set $K=10$, $n=10$, and use a path graph as the reference graph $\G$ (Definition~\ref{itm:third}).}
\label{Fig:GLD_1_2}
\end{figure}

\begin{figure}[ht]
\centering
\includegraphics[width=\linewidth,height=0.4\linewidth]{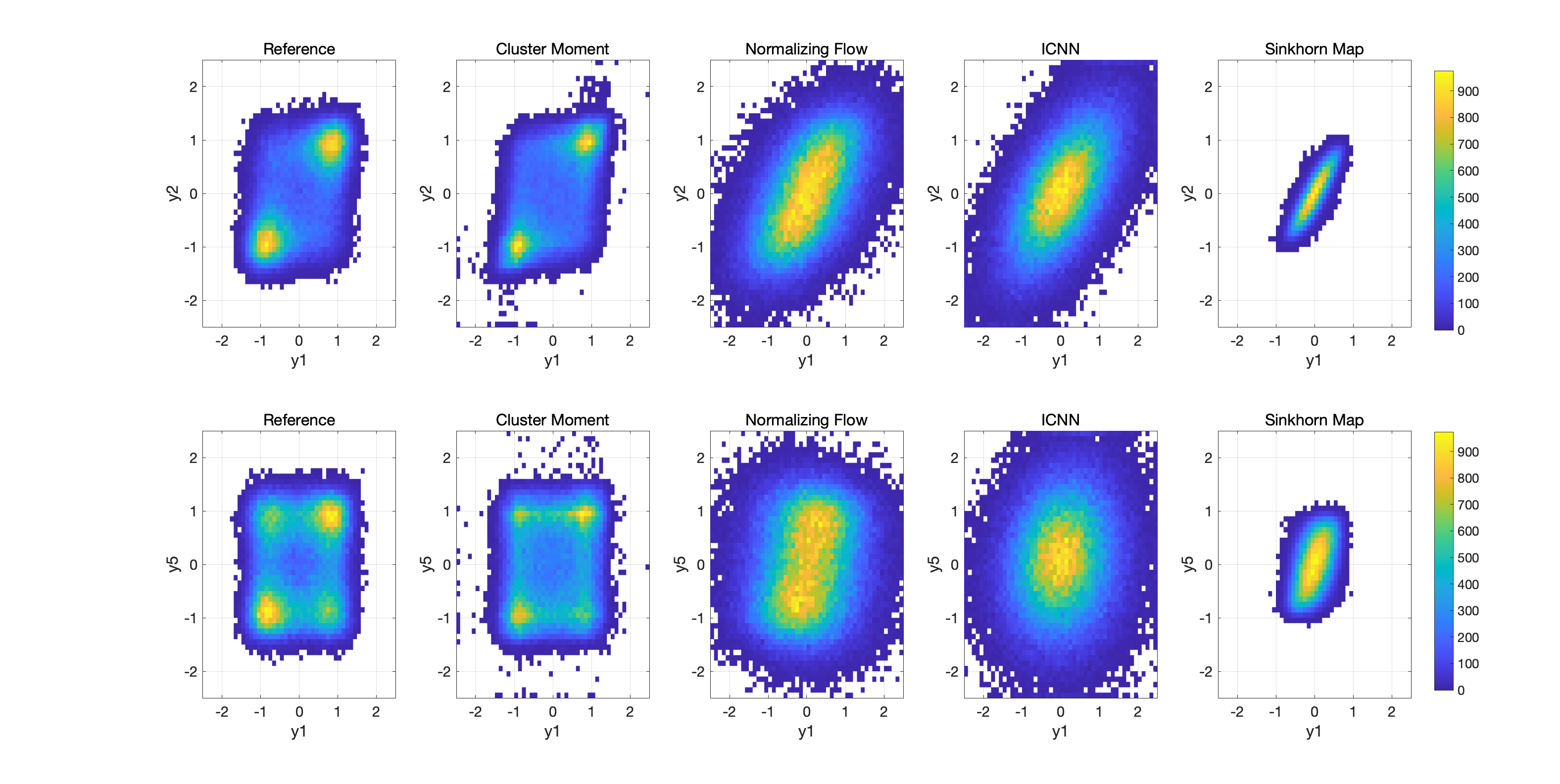}
\caption{\small Comparison of the cluster moment relaxation \eqref{OTmom-2}, the normalizing-flow baseline, the ICNN baseline, and the Sinkhorn transport map for the one-dimensional Ginzburg--Landau model with $d=50$, $\beta=1/20$, $L=2.5$, and $\lambda=0.01$. The displayed marginals are $(y_1,y_2)$ and $(y_1,y_5)$. For the cluster moment relaxation, we set $K=50$, $n=10$, and use a path graph as the reference graph $\G$ (Definition~\ref{itm:third}).}
\label{Fig:GLD_1_20}
\end{figure}

We complement the qualitative marginal plots with a quantitative comparison based on the sliced Wasserstein distance between the transported Gaussian test samples and the target test samples. In each case, we use the same training and test sets as in the generative modeling experiments, and we evaluate the sliced Wasserstein distance using 500 random one-dimensional projections and 5000 samples from each distribution. Smaller values indicate a more accurate transport map.

\begin{table}[!htbp]
\centering
\begingroup
\footnotesize
\setlength{\tabcolsep}{3.6pt}
\renewcommand{\arraystretch}{0.9}
\begin{tabular}{@{}ccccccc@{}}
\hline
$d$ & $\beta$ & $\lambda$ & Cluster moment & Normalizing flow & ICNN & Sinkhorn map\\
\hline
10 & $1/12$ & 0.02 & 0.0274 & 0.1446 & 0.0878 & 0.2284\\
10 & $1/8$ & 0.03 & 0.0342 & 0.1298 & 0.0895 & 0.2083\\
20 & $1/12$ & 0.02 & 0.0297 & 0.0639 & 0.0578 & 0.3205\\
20 & $1/10$ & 0.025 & 0.0270 & 0.0418 & 0.0688 & 0.2738\\
30 & $1/16$ & 0.015 & 0.0278 & 0.0559 & 0.0617 & 0.3725\\
30 & $1/12$ & 0.02 & 0.0339 & 0.0411 & 0.0717 & 0.2799\\
40 & $1/20$ & 0.01 & 0.0311 & 0.0590 & 0.0720 & 0.4326\\
40 & $1/16$ & 0.015 & 0.0298 & 0.0452 & 0.0745 & 0.3279\\
50 & $1/20$ & 0.01 & 0.0289 & 0.0597 & 0.0842 & 0.3987\\
50 & $1/16$ & 0.015 & 0.0311 & 0.0621 & 0.0908 & 0.2818\\
\hline
\end{tabular}
\endgroup
\caption{Sliced Wasserstein distance between the pushforward of the Gaussian test samples and the target samples for ten Gaussian-to-1D-GLD instances.}
\label{tab:GLD-sliced-W2}
\end{table}

\FloatBarrier

The results in Table~\ref{tab:GLD-sliced-W2} are consistent across all ten examples. The cluster moment relaxation attains the smallest sliced Wasserstein distance in every case. The normalizing flow baseline is usually the second best, while ICNN remains competitive on coarse global metrics but is consistently less accurate than the cluster moment relaxation. The Sinkhorn transport map, implemented by barycentric projection of the entropic plan followed by a $k$-nearest-neighbor extension, is the weakest baseline in this study and deteriorates more noticeably as the dimension grows. These quantitative results support the visual comparisons in Figure~\ref{Fig:GLD_1_2}--Figure~\ref{Fig:GLD_1_20} and show that the superior accuracy of the cluster moment relaxation is not limited to a single instance.}

\FloatBarrier

{
\subsection{MNIST digits}

Finally, we include a small real-data generative modeling experiment on MNIST digits.  For each digit class, we compute a six-dimensional PCA representation from the training images and consider the transport from a standard Gaussian source $\mu=\mathcal{N}(0,I_6)$ to the empirical distribution of the PCA coefficients of that digit.  We use $2000$ target training samples, draw the same number of source samples from $\mathcal{N}(0,I_6)$, and extract a transport map from the cluster moment relaxation \eqref{OTmom-2}.  The relaxation degree is set to $n=5,$ and the clusters are selected  with maximal cluster size two.

Figure~\ref{Fig:MNIST-Gauss} compares real test images with samples generated by decoding the transported PCA coefficients.  The learned map produces recognizable samples for several digit classes.  

\begin{figure}[ht]
\centering
\includegraphics[width=0.85\linewidth]{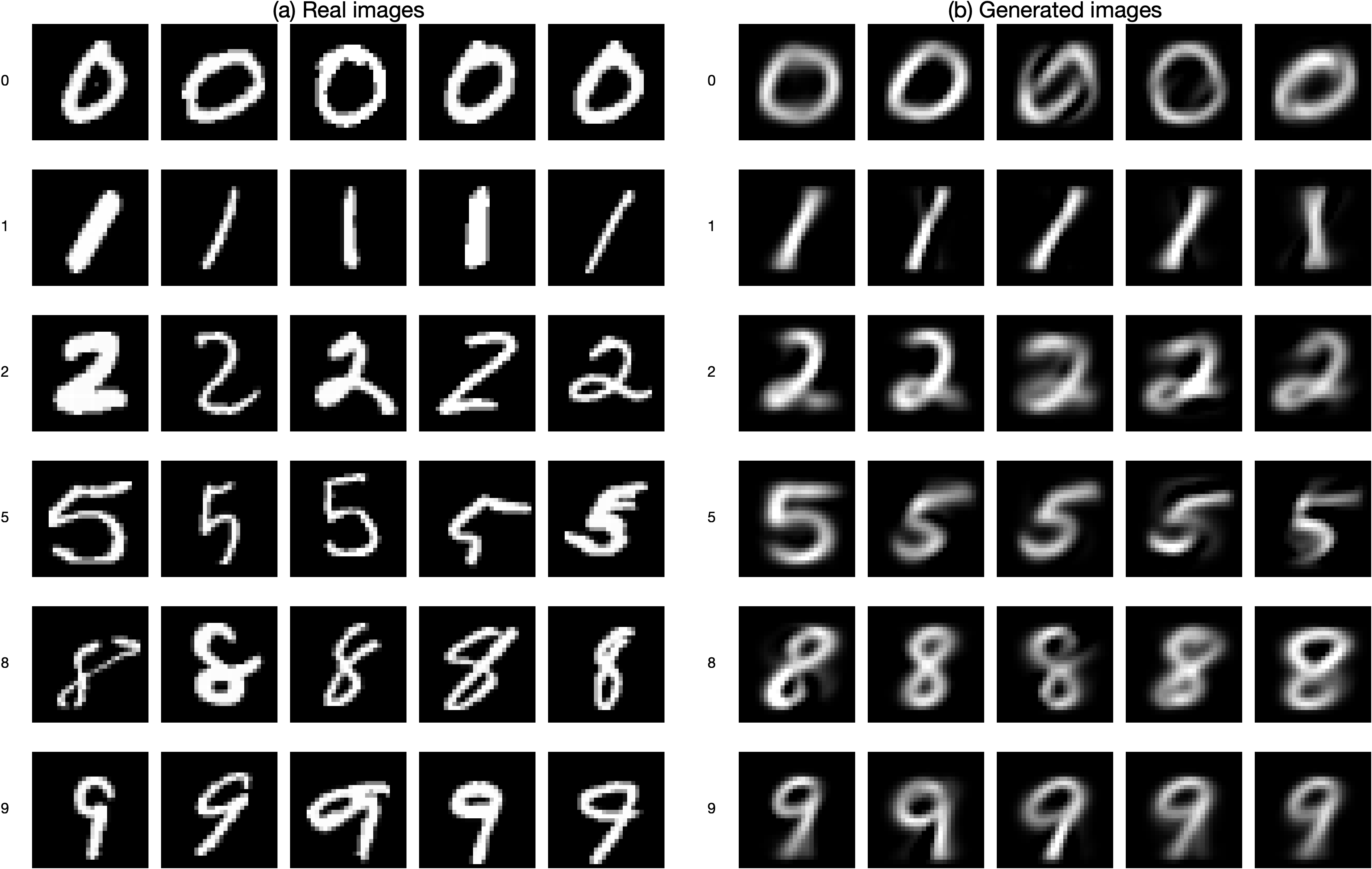}
\caption{\small Gaussian-to-MNIST generative modeling in a six-dimensional PCA latent space.  For each digit, the source is $\mathcal{N}(0,I_6)$ and the target is the empirical distribution of PCA coefficients of the corresponding MNIST digit.  The transport map is extracted from the cluster moment relaxation \eqref{OTmom-2} with relaxation degree $n=5.$}
\label{Fig:MNIST-Gauss}
\end{figure}

}

\section{Conclusion}\label{Sec:conc}

In this paper, we proposed convex relaxation approaches for addressing the high dimensionality of optimal transport. By introducing marginal and cluster moment relaxations, we obtain 
tractable convex programs that provide computable lower bounds and enable the extraction 
of transport maps. Our theoretical analysis gives approximation and sample error bounds
for sparse Gaussian models, and also gives an approximation bound for local perturbations
of a mean-field product measure. Numerical experiments demonstrate that the 
approach extends effectively to non-Gaussian distributions. Furthermore, we illustrated 
the potential of transport maps derived from these relaxations as alternatives to neural 
networks in generative modeling. These results highlight convex relaxation as a 
promising dimension reduction framework for scaling OT to high-dimensional problems.

\appendix

\section{Useful Lemmas}\label{sec:useful-lemmas}

{
We collect the probabilistic and deterministic auxiliary estimates used in
Section~\ref{Sec:analysis}.

\begin{lem}[Hoeffding inequality]\label{lem:hoeffding}
Let $\xi_1,\ldots,\xi_N$ be independent random variables satisfying
$a\le \xi_\ell \le b$ almost surely for all $\ell\in[N]$. Then for any $t>0$,
\begin{equation}\label{hoeffding}
\mathbb{P}\!\left(
\left|
\frac1N\sum_{\ell=1}^N \xi_\ell
-
\mathbb{E}\!\left[\frac1N\sum_{\ell=1}^N \xi_\ell\right]
\right|
\ge t
\right)
\le
2\exp\!\left(
-\frac{2Nt^2}{(b-a)^2}
\right).
\end{equation}
\end{lem}

Lemma~\ref{lem:hoeffding} is classical; see, for example,
Hoeffding~\cite{hoeffding1963probability}.

\begin{lem}[Truncation error for local monomials]\label{lem:tailtrunc}
Assume Assumption~\ref{assump:sampletail}. Let $f$ be any monomial of total degree at most
$2n$ that depends on at most $2r$ coordinates of either $x$ or $y$. Then there exist constants
$C_{\rm tail},c_{\rm tail}>0$, depending only on $n,r,a_{\rm tail},A_{\rm tail}$, such that
for all $R\geq1$,
\begin{equation}\label{tailtrunc}
\mathbb{E}\Bigl[|f(U)|\,\mathbf{1}_{\{\max_{i\in S_f}|U_i|>R\}}\Bigr]
\le
C_{\rm tail}R^{2n}e^{-c_{\rm tail}R^2}.
\end{equation}
Here $U\sim\mu$ if $f$ is an $x$-moment function, $U\sim\nu$ if $f$ is a $y$-moment function,
and $S_f$ denotes the coordinate support of $f$. If $S_f=\varnothing$, the indicator in
\eqref{tailtrunc} is interpreted as zero.
\end{lem}

\begin{proof}
If $S_f=\varnothing$, then the left-hand side of \eqref{tailtrunc} is zero and there is
nothing to prove. Otherwise, set $W:=\max_{i\in S_f}|U_i|$. Since $|S_f|\leq 2r$,
Assumption~\ref{assump:sampletail} and the union bound give
\begin{equation}
\mathbb{P}(W>t)
\leq
2r A_{\rm tail}e^{-a_{\rm tail}t^2},
\qquad t\geq0.
\notag
\end{equation}
Since $|f(U)|\leq W^{2n}$, integration by parts yields
\begin{equation}
\mathbb{E}\bigl[W^{2n}\mathbf{1}_{\{W>R\}}\bigr]
\leq
R^{2n}\mathbb{P}(W>R)
+
2n\int_R^\infty t^{2n-1}\mathbb{P}(W>t)\,\dd t.
\notag
\end{equation}
The right-hand side is bounded by
$C_{\rm tail}R^{2n}e^{-c_{\rm tail}R^2}$ after adjusting constants depending only on
$n,r,a_{\rm tail},A_{\rm tail}$.
\end{proof}

\begin{lem}[Sparse Gaussian concentration]\label{lem:gauss-sample-concentration}
Let $X_i^{(1)},\ldots,X_i^{(N)}$ be independent samples from
$\mathcal N(m_i,\Sigma_i)$, $i=1,2$, where
$aI_d\preceq \Sigma_i\preceq bI_d$ and $\|m_1-m_2\|\le M\sqrt d$.
Let $\widehat m_i$ and $\widehat\Sigma_i$ be the empirical mean and the
sample-centered empirical covariance, with normalization $1/N$, and let $G^h$
be the graph specifying the retained covariance entries.
There is a constant $C>0$, depending only on $a,b,M$, such that with probability
at least $1-\delta$,
\begin{align*}
&\|\widehat m_i-m_i\|^2
\le
C\frac{d+\log(1/\delta)}{N},\qquad i=1,2,\\
&\left|
\left\langle m_1-m_2,
(\widehat m_1-m_1)-(\widehat m_2-m_2)
\right\rangle
\right|
\le
C\sqrt{\frac{d\log(1/\delta)}{N}},\\
&\|[\widehat\Sigma_i-\Sigma_i]_{G^h}\|_F^2
\le
C\left\{
\frac{d+|E(G^h)|+\log(1/\delta)}{N}
+
\left(
\frac{d+|E(G^h)|+\log(1/\delta)}{N}
\right)^2
\right\},\qquad i=1,2.
\end{align*}
If, in addition, $d+|E(G^h)|+\log(1/\delta)\le cN$ for a sufficiently small
absolute constant $c$, then with probability at least $1-\delta$ the
preceding estimates hold and
\begin{align*}
\|[\widehat\Sigma_i-\Sigma_i]_{G^h}\|_\infty
&\le
C\left(
\frac{\log(16(d+|E(G^h)|))+\log(1/\delta)}{N}
\right)^{1/2},
\qquad i=1,2,\\
\|[\widehat\Sigma_i-\Sigma_i]_{G^h}\|_F^2
&\le
C\frac{d+|E(G^h)|+\log(1/\delta)}{N},
\qquad i=1,2.
\end{align*}
\end{lem}

\begin{proof}
For the mean estimates, note first that
$\widehat m_i-m_i\sim \mathcal N(0,\Sigma_i/N)$ and
$\Sigma_i\preceq bI_d$. The upper-tail part of the
Laurent--Massart chi-square inequality~\cite[Lemma~1]{laurent2000adaptive}
states that, for every $t>0$,
\[
\mathbb P\left\{
\frac{N}{b}\|\widehat m_i-m_i\|^2\ge d+2\sqrt{dt}+2t
\right\}
\le e^{-t}.
\]
Using $2\sqrt{dt}\le d+t$, taking $t=\log 16+\log(1/\delta)$, and applying a
union bound over $i=1,2$ gives
\[
\|\widehat m_i-m_i\|^2
\le
C\frac{d+\log(1/\delta)}{N},
\qquad i=1,2,
\]
with probability at least $1-\delta/8$. Similarly,
\[
\left\langle m_1-m_2,
(\widehat m_1-m_1)-(\widehat m_2-m_2)
\right\rangle
\]
is centered Gaussian with variance
\[
\frac1N
(m_1-m_2)^\top(\Sigma_1+\Sigma_2)(m_1-m_2)
\le
\frac{2bM^2d}{N}.
\]
The scalar Gaussian tail bound therefore gives
\[
\left|
\left\langle m_1-m_2,
(\widehat m_1-m_1)-(\widehat m_2-m_2)
\right\rangle
\right|
\le
C\sqrt{\frac{d\log(1/\delta)}{N}}
\]
with probability at least $1-\delta/8$.

It remains to control the covariance terms. We prove the bounds for one fixed
$i\in\{1,2\}$, which indexes the two Gaussian sample clouds rather than a
coordinate, and then take a union bound over $i$. Write
$\Sigma=\Sigma_i$ and $\widehat\Sigma=\widehat\Sigma_i$.

For the entrywise bound under the additional sample-size condition, we use the
standard empirical covariance concentration inequality for sub-Gaussian vectors.
Ravikumar et
al.~\cite[Lemma~1]{ravikumar2011high} prove the per-entry sample
covariance tail bound, which in the Gaussian case gives the small-deviation form
$\mathbb P\{|(\widehat\Sigma-\Sigma)_{jk}|>\varepsilon\}
\le 4\exp(-c_0N\varepsilon^2)$ for each fixed covariance entry, whenever
$0<\varepsilon\le \varepsilon_0$, where $c_0,\varepsilon_0>0$ depend only on
$b$.
Taking
\[
\varepsilon=C
\left(
\frac{\log(16(d+|E(G^h)|))+\log(1/\delta)}{N}
\right)^{1/2}
\]
we first choose $C$ large enough, depending only on $c_0$, so that the desired
tail probability follows once $\varepsilon\le\varepsilon_0$. We then choose the
constant $c$ in the sample-size condition small enough, depending only on this
fixed $C$ and on $\varepsilon_0$, so that
$d+|E(G^h)|+\log(1/\delta)\le cN$ implies $\varepsilon\le\varepsilon_0$; here
we use that the logarithmic numerator in $\varepsilon$ is bounded by an
absolute multiple of $d+|E(G^h)|+\log(1/\delta)$.
Thus
\[
\mathbb P\left\{
|(\widehat\Sigma-\Sigma)_{jk}|
>
C\left(
\frac{\log(16(d+|E(G^h)|))+\log(1/\delta)}{N}
\right)^{1/2}
\right\}
\le
\frac{\delta}{8(d+|E(G^h)|)}.
\]
A union bound over the $d+|E(G^h)|$ retained entries and then over $i=1,2$
proves the displayed $\ell_\infty$ estimate with probability at least
$1-\delta/4$.

For the Frobenius bound, set
$Y^{(\ell)}:=X_i^{(\ell)}-m_i\in\mathbb R^d$, a centered Gaussian vector, and
define the true-mean-centered retained covariance error
\[
Z:=
\left[
\frac1N\sum_{\ell=1}^N
(Y^{(\ell)}Y^{(\ell)\top}-\Sigma)
\right]_{G^h}.
\]
Here $Z$ is not exactly $[\widehat\Sigma-\Sigma]_{G^h}$, because
$\widehat\Sigma$ is centered at the sample mean. Indeed,
\[
[\widehat\Sigma-\Sigma]_{G^h}
=
Z-\left[(\widehat m_i-m_i)(\widehat m_i-m_i)^\top\right]_{G^h}.
\]
Let $S$ be the subspace of symmetric matrices supported on $G^h$, equipped with
the Frobenius norm. Identifying $S$ with a Euclidean space of dimension
$d+|E(G^h)|$ through an orthonormal basis, the standard volumetric estimate for
unit spheres gives a $1/2$-net $\mathcal N$ satisfying
$|\mathcal N|\le (1+2/(1/2))^{d+|E(G^h)|}=5^{d+|E(G^h)|}$. For any fixed
$B\in S$ with $\|B\|_F=1$,
\[
\langle Z,B\rangle
=
\frac1N\sum_{\ell=1}^N
\left\{(Y^{(\ell)})^\top B Y^{(\ell)}-\Tr(\Sigma B)\right\}.
\]
For this fixed scalar quadratic form, write
$Y^{(\ell)}=\Sigma^{1/2}g^{(\ell)}$ with
$g^{(\ell)}\sim\mathcal N(0,I_d)$. The matrix
$\Sigma^{1/2}B\Sigma^{1/2}$ has both operator and Frobenius norms bounded by
$C$ because $\Sigma\preceq bI_d$ and $\|B\|_F=1$. The scalar Hanson--Wright
inequality, or equivalently the standard Gaussian quadratic-form concentration
\cite[Chapter~2]{wainwright2019high}, therefore implies
\[
\mathbb P\left\{|\langle Z,B\rangle|>t\right\}
\le
2\exp\left[-cN\min(t^2,t)\right],
\qquad t>0.
\]
Taking
\[
t=C\left(
\sqrt{
\frac{d+|E(G^h)|+u}{N}
}
+
\frac{d+|E(G^h)|+u}{N}
\right)
\]
with $u\geq 0$ and increasing $C$ if necessary gives
\[
\mathbb P\left\{|\langle Z,B\rangle|>t\right\}
\le
2e^{-u-(d+|E(G^h)|)\log 5}.
\]
A union bound over $B\in\mathcal N$ gives the same bound simultaneously on the
net with probability at least $1-2e^{-u}$. On this event, the net approximation
implies
\[
\|Z\|_F
=
\sup_{\substack{B\in S\\ \|B\|_F=1}}
\langle Z,B\rangle
\le
2\max_{B\in\mathcal N}|\langle Z,B\rangle|
\le
C\left(
\sqrt{
\frac{d+|E(G^h)|+u}{N}
}
+
\frac{d+|E(G^h)|+u}{N}
\right).
\tag{A.1}
\]
The retained Frobenius norm of the sample-centering correction is bounded by
$\|\widehat m_i-m_i\|^2$ and is absorbed by the mean estimate above.

Taking $u=\log 16+\log(1/\delta)$ in (A.1) gives, with probability at least
$1-\delta/8$,
\[
\left\|
[\widehat\Sigma-\Sigma]_{G^h}
\right\|_F^2
\le
C\left\{
\frac{d+|E(G^h)|+\log(1/\delta)}{N}
+
\left(
\frac{d+|E(G^h)|+\log(1/\delta)}{N}
\right)^2
\right\}.
\]
A union bound over $i=1,2$ gives the Frobenius estimate simultaneously for the
two samples.

Finally, taking the intersection of the mean-norm event, the linear-mean event,
and the Frobenius covariance event gives the first three estimates with total
probability at least $1-\delta$ after adjusting constants. Under the additional
condition $d+|E(G^h)|+\log(1/\delta)\le cN$, we also intersect with the
entrywise covariance event. This condition implies the small-deviation condition
needed in the entrywise covariance bound, after decreasing $c$ if necessary, and
it absorbs the quadratic term in the Frobenius estimate into the linear one.
The simplified Frobenius bound and the entrywise estimate follow with total
probability at least $1-\delta$.
\end{proof}

\begin{lem}[Dual perturbation bound for SDP values]\label{lem:sdp-perturb}
Consider the value function $p(\cdot)$ in \eqref{sdpstab-primal} and its dual
\eqref{sdpstab-dual}. Let $b_1$ and $b_2$ be two right-hand sides for the prescribed
sample-estimated constraints. Suppose that strong duality holds and that the dual optimum
is attained at both $b_1$ and $b_2$. For $i=1,2$, let $\Lambda^\star(b_i)$ be the set of
optimal dual multipliers $\lambda$ associated with the constraint $\A(M)=b_i$. If
\begin{equation}
R\ge
\max_{i=1,2}\inf_{\lambda\in\Lambda^\star(b_i)}\|\lambda\|_1,
\notag
\end{equation}
then
\begin{equation}\label{sdp-perturb-bound}
|p(b_1)-p(b_2)|
\le
R\|b_1-b_2\|_\infty .
\end{equation}
\end{lem}

\begin{proof}
Fix $\alpha>0$. Choose $\lambda_1\in\Lambda^\star(b_1)$ with
$\|\lambda_1\|_1\le R+\alpha$. Since the dual feasible set in \eqref{sdpstab-dual} does
not depend on the right-hand side, $\lambda_1$ is also dual feasible for the problem with
right-hand side $b_2$. Strong duality at $b_1$ and weak duality at $b_2$ give
\begin{equation}
p(b_1)-p(b_2)
\le
\langle b_1,\lambda_1\rangle-\langle b_2,\lambda_1\rangle
\le
(R+\alpha)\|b_1-b_2\|_\infty .
\notag
\end{equation}
Letting $\alpha\downarrow0$ yields
$p(b_1)-p(b_2)\le R\|b_1-b_2\|_\infty$. Repeating the same argument with an optimal dual
multiplier at $b_2$ gives the reverse inequality.
\end{proof}

}

\begin{lem}[Gaussian relaxation error]\label{propapp}
Let $\mu=\mathcal{N}(m_1,\Sigma_1)$ and $\nu=\mathcal{N}(m_2,\Sigma_2)$ with
$\Sigma_1,\Sigma_2\succ0$. Suppose the optimal value of \eqref{GSmom} is
${\rm opt}_{\G}$. Let
\begin{equation}\label{dLam}
\begin{aligned}
\Lambda_1^\star
&:=\Sigma_1^{-1/2}
\left[ \Sigma_1^{1/2}\Sigma_2 \Sigma_1^{1/2} \right]^{1/2}
\Sigma_1^{-1/2},\\
\Lambda_2^\star
&:=\Sigma_1^{1/2}
\left[ \Sigma_1^{1/2}\Sigma_2 \Sigma_1^{1/2} \right]^{-1/2}
\Sigma_1^{1/2}.
\end{aligned}
\end{equation}
Then
\begin{equation}\label{ineapp}
W^2_2(\mu,\nu)-\epsilon_{\G}\leq {\rm opt}_{\G}\leq W^2_2(\mu,\nu),
\end{equation}
where
\begin{align}\label{defiep}
\epsilon_{\G}
&:=2\Tr(\Sigma_1)
\inf_{B_1\in\S_{\G}}\|B_1-\Lambda_1^\star\|_2
+2\Tr(\Sigma_2)
\inf_{B_2\in\S_{\G}}\|B_2-\Lambda_2^\star\|_2 .
\end{align}
Here $\S_{\G}$ is defined in \eqref{defspset}.
\end{lem}

\begin{proof}
The problem \eqref{GSmom} can be simplified into the following problem by using the Schur complement of the first entry of $X$ as the decision variable
\begin{align}\label{GSmom1}
\min\ &\ \|m_1-m_2\|^2+\Tr(\Sigma_1)+\Tr(\Sigma_2)-2\Tr(Y) \\
{\rm s.t.}\ &\ [Z_1]_{\G}=[\Sigma_1]_{\G},\quad
              [Z_2]_{\G}=[\Sigma_2]_{\G}, \notag \\
&\ \begin{bmatrix}
Z_1 & Y \\
Y^\top & Z_2
\end{bmatrix} \in \S^{2d}_+, \notag
\end{align}
where the operator $[\cdot]_{\G}$ projects a matrix onto the sparsity pattern $\G$  (Definition~\ref{itm:third}). The dual problem of \eqref{GSmom1} is:
\begin{align}\label{GSmomd1}
\max\ &\ \|m_1-m_2\|^2+\Tr(\Sigma_1)+\Tr(\Sigma_2)-\<\Sigma_1,\Lambda_1\>-\<\Sigma_2,\Lambda_2\> \\
{\rm s.t.}\ &\ \begin{bmatrix} \Lambda_1&-I_d\\-I_d & \Lambda_2 \end{bmatrix}\in \S^{2d}_+,\ \Lambda_1,\Lambda_2\in {\S_{\G}}, \notag
\end{align}
where {$\S_{\G}$} is defined in \eqref{defspset}. When $\G$ is complete graph, \eqref{GSmom1} and \eqref{GSmomd1} become
\begin{equation}\label{GSmomc1}
\min\Big\{ \|m_1-m_2\|^2+\Tr(\Sigma_1)+\Tr(\Sigma_2)-2\Tr(Y):\ \begin{bmatrix}
\Sigma_1 & Y \\
Y^\top & \Sigma_2
\end{bmatrix}\in \S^{2d}_+\Big\},
\end{equation}
and
\begin{equation}\label{GSmomcd1}
\max\Big\{ \|m_1-m_2\|^2+\Tr(\Sigma_1)+\Tr(\Sigma_2)-\<\Sigma_1,\Lambda_1\>-\<\Sigma_2,\Lambda_2\>:\ \begin{bmatrix} \Lambda_1&-I_d\\-I_d & \Lambda_2 \end{bmatrix}\in \S^{2d}_+\Big\},
\end{equation}
respectively. The problem \eqref{GSmomc1} and \eqref{GSmomcd1} have the following closed-form solutions, whose function values are exactly $W_2^2(\mu,\nu).$
\begin{equation}\label{primalf}
\begin{bmatrix}
\Sigma_1& \Sigma_1^{1/2}\left[ \Sigma_1^{1/2}\Sigma_2 \Sigma_1^{1/2} \right]^{1/2}\Sigma_1^{-1/2}\\
\Sigma_1^{-1/2}\left[ \Sigma_1^{1/2}\Sigma_2 \Sigma_1^{1/2} \right]^{1/2}\Sigma_1^{1/2}&\Sigma_2
\end{bmatrix},
\end{equation}
\begin{equation}\label{dSf}
\begin{bmatrix}
\Lambda_1^\star& -I_d\\
-I_d&\Lambda_2^\star
\end{bmatrix}.
\end{equation}
One can easily check that \eqref{primalf} and \eqref{dSf} satisfy the constraints in \eqref{GSmomc1} and \eqref{GSmomcd1}, with objective values equal to the exact OT cost \eqref{W2GGM}.

Now let $B_1,B_2\in\S_{\G}$ and set
\begin{equation}
\eta_i:=\|B_i-\Lambda_i^\star\|_2,
\qquad
\widetilde\Lambda_i:=B_i+\eta_i I_d,
\qquad i=1,2.
\notag
\end{equation}
Since $B_i-\Lambda_i^\star+\eta_i I_d\succeq0$, we have
\begin{equation}
\begin{bmatrix}
\widetilde\Lambda_1&-I_d\\
-I_d&\widetilde\Lambda_2
\end{bmatrix}
=
\begin{bmatrix}
\Lambda_1^\star&-I_d\\
-I_d&\Lambda_2^\star
\end{bmatrix}
+
\begin{bmatrix}
B_1-\Lambda_1^\star+\eta_1I_d&0\\
0&B_2-\Lambda_2^\star+\eta_2I_d
\end{bmatrix}
\succeq0.
\notag
\end{equation}
Thus $(\widetilde\Lambda_1,\widetilde\Lambda_2)$ is feasible for \eqref{GSmomd1}. Substituting this feasible point into \eqref{GSmomd1} gives
\begin{align}
{\rm opt}_{\G}
&\ge
W_2^2(\mu,\nu)
-
\sum_{i=1}^2
\left\langle \Sigma_i,B_i-\Lambda_i^\star+\eta_iI_d\right\rangle 
\notag\\
&\ge
W_2^2(\mu,\nu)-2\Tr(\Sigma_1)\eta_1-2\Tr(\Sigma_2)\eta_2,
\notag
\end{align}
where we used $-\eta_iI_d\preceq B_i-\Lambda_i^\star\preceq \eta_iI_d$ and $\Sigma_i\succeq0$.
Taking the infimum over $B_1,B_2\in\S_{\G}$ gives the left-hand side of \eqref{ineapp}. The right-hand side of \eqref{ineapp} is immediate because the cluster moment relaxation provides a lower bound of the exact OT cost.
\end{proof}

\section{Proof details}\label{sec:proof}

We first prove Theorems~\ref{thmexp} and
\ref{thm:gauss-sample}. For these Gaussian results, we assume that $\mu,$ $\nu$ are
Gaussian distributions $\mathcal{N}(m_1,\Sigma_1),\mathcal{N}(m_2,\Sigma_2)$ for some
$\Sigma_1,\Sigma_2\succ 0.$ We use Lemma~\ref{propapp} from
Appendix~\ref{sec:useful-lemmas}, which provides an error bound between the convex
relaxation \eqref{GSmom} and the exact OT cost \eqref{W2GGM}.

\medskip

\noindent{\bf Proof of Theorem~\ref{thmexp}}

\begin{proof}
When $\G$ is a complete graph, the two infima in \eqref{defiep} are zero. Thus,
Lemma~\ref{propapp} implies the relaxation is exact. This completes the proof of (i).

We next consider $\G=G^h$. Our proof idea is motivated by the Demko-Moss-Smith theorem
\cite{demko1984decay}. {Let $\Lambda_1^\star,\Lambda_2^\star$ denote the dense
Gaussian dual optimizers from Lemma~\ref{propapp}, namely
\begin{equation}
\Lambda_1^\star=\Sigma_1^{-1/2}\left[ \Sigma_1^{1/2}\Sigma_2 \Sigma_1^{1/2} \right]^{1/2}\Sigma_1^{-1/2},
\quad
\Lambda_2^\star=\Sigma_1^{1/2}\left[ \Sigma_1^{1/2}\Sigma_2 \Sigma_1^{1/2} \right]^{-1/2}\Sigma_1^{1/2}.
\notag
\end{equation}
}
Because $b I_d\succeq \Sigma_1,\Sigma_2 \succeq a I_d,$ we have
\begin{equation}\label{specbd}
a^{-2} I_d\succeq \Sigma_1^{-1/2}\Sigma_2^{-1}\Sigma_1^{-1/2}\succeq b^{-2} I_d.
\end{equation}
From the approximation theorem of Chebyshev polynomials \cite{trefethen2019approximation}, there are constants $C_{a,b}>0$ and $\rho_{a,b}>1$ such that, for any $k\in\N^+$, there is a degree $k$ polynomial $p_k$ satisfying
\begin{equation}\label{polyapprox}
\|x^{-1/2}-p_k(x)\|_{\infty,[b^{-2},a^{-2}]}
\leq C_{a,b}\rho_{a,b}^{-k}.
\end{equation}
Thus
\begin{equation}\label{polyapproxS}
\left\|\left[\Sigma_1^{1/2}\Sigma_2\Sigma_1^{1/2}\right]^{1/2}
-p_k\(\Sigma_1^{-1/2}\Sigma_2^{-1}\Sigma_1^{-1/2}\)\right\|_2
\leq C_{a,b}\rho_{a,b}^{-k}.
\end{equation}
Let $p_k(x)=\sum_{i=0}^k\gamma_i x^i$. Then
\begin{align}
&\Sigma_1^{-1/2}p_k\(\Sigma_1^{-1/2}\Sigma_2^{-1}\Sigma_1^{-1/2}\)\Sigma_1^{-1/2}
=\sum_{i=0}^k\gamma_i\Sigma_1^{-1/2}\(\Sigma_1^{-1/2}\Sigma_2^{-1}\Sigma_1^{-1/2}\)^i\Sigma_1^{-1/2}
\notag\\
&\qquad=\sum_{i=0}^k\gamma_i\(\Sigma_1^{-1}\Sigma_2^{-1}\)^i\Sigma_1^{-1}
=p_k\(\Sigma_1^{-1}\Sigma_2^{-1}\)\Sigma_1^{-1}. \label{Aug_20_1}
\end{align}
Since $\Sigma_1^{-1},\Sigma_2^{-1}\in\S_G$, the matrix in \eqref{Aug_20_1} belongs to
$\S_{G^{2k+1}}$. Combining this with \eqref{polyapproxS} gives
\begin{equation}\label{Aug_20_2}
\left\|\Lambda_1^\star-p_k\(\Sigma_1^{-1}\Sigma_2^{-1}\)\Sigma_1^{-1}\right\|_2
\leq C_{a,b}a^{-1}\rho_{a,b}^{-k}.
\end{equation}
Similarly,
\begin{equation}\label{Aug_20_3}
\Sigma_2^{-1}\Sigma_1^{-1/2}p_k\(\Sigma_1^{-1/2}\Sigma_2^{-1}\Sigma_1^{-1/2}\)\Sigma_1^{1/2}
=
\Sigma_2^{-1}p_k\(\Sigma_1^{-1}\Sigma_2^{-1}\)
\in\S_{G^{2k+1}},
\end{equation}
and
\begin{equation}\label{Aug_20_4}
\left\|\Lambda_2^\star-
\Sigma_2^{-1}p_k\(\Sigma_1^{-1}\Sigma_2^{-1}\)\right\|_2
\leq C_{a,b}b^{1/2}a^{-3/2}\rho_{a,b}^{-k}.
\end{equation}
Therefore, whenever $h\geq 2k+1$, Lemma~\ref{propapp} gives
\begin{align}\label{Aug_20_11}
\epsilon_{G^h}
&\leq
2\Tr(\Sigma_1)C_{a,b}a^{-1}\rho_{a,b}^{-k}
+2\Tr(\Sigma_2)C_{a,b}b^{1/2}a^{-3/2}\rho_{a,b}^{-k}
\notag\\
&\leq
2\(C_{a,b}a^{-3/2}b^{3/2}+C_{a,b}a^{-1}b\)d\rho_{a,b}^{-k},
\end{align}
where we used $\Tr(\Sigma_1),\Tr(\Sigma_2)\leq bd$. Taking $k$ such that
$h\geq 2k+1$ and $k\geq (h-3)/2$, we obtain
\begin{equation}\label{Aug_20_13}
\epsilon_{G^h}
\leq C'_{a,b}d\left[\rho_{a,b}^{-1/2}\right]^h,
\qquad h\geq3.
\end{equation}
For $h<3$, the bound follows after increasing the constant, since
$\|\Lambda_1^\star\|_2+\|\Lambda_2^\star\|_2\leq C_{a,b}$ and hence
$\epsilon_{G^h}\leq C_{a,b}d$. This proves (ii), after renaming the constants.
\end{proof}

\medskip

{
\noindent{\bf Proof of Theorem~\ref{thm:gauss-sample}}

\begin{proof}
Let $\Lambda_1^\star,\Lambda_2^\star$ be the dense Gaussian dual matrices in
\eqref{dLam}. The proof of Theorem~\ref{thmexp} shows that, for $\G=G^h$, there exist
matrices $\Lambda_1^{\rm sp},\Lambda_2^{\rm sp}\in\S_{G^h}$ and a constant
$C_0>0$ such that
\begin{equation}\label{gauss-sample-sparse-approx}
\|\Lambda_i^{\rm sp}-\Lambda_i^\star\|_2\le C_0\rho^{-h},
\qquad
\|\Lambda_i^{\rm sp}\|_2\le C_0,
\qquad i=1,2,
\end{equation}
where $C_0$ depends only on $a,b$ and $\rho>1$ is the constant in
Theorem~\ref{thmexp}. We choose $C_0$ large enough once so that it also covers
the intermediate concentration estimates below.

We first record the concentration event used below. By
Lemma~\ref{lem:gauss-sample-concentration}, applied with $\delta/2$ and with the harmless
factor $\log2$ absorbed into the constant, with probability at least $1-\delta/2$ the
following four estimates hold simultaneously:
\begin{align}
\|\widehat m_i-m_i\|^2
&\le
C_0\frac{d+\log(1/\delta)}{N},\qquad i=1,2, \label{gauss-sample-event-mean}\\
\left|
\left\langle m_1-m_2,
(\widehat m_1-m_1)-(\widehat m_2-m_2)
\right\rangle
\right|
&\le
C_0\sqrt{\frac{d\log(1/\delta)}{N}}, \label{gauss-sample-event-mean-linear}\\
\|[\widehat\Sigma_i-\Sigma_i]_{G^h}\|_F^2
&\le
C_0\frac{d+|E(G^h)|+\log(1/\delta)}{N},\qquad i=1,2, \label{gauss-sample-event-frob}\\
\|[\widehat\Sigma_i-\Sigma_i]_{G^h}\|_\infty
&\le
C_0\left(
\frac{\log(d+|E(G^h)|)+\log(1/\delta)}{N}
\right)^{1/2},
\qquad i=1,2, \label{gauss-sample-event-entry}
\end{align}
In addition, by the standard Gaussian concentration for fixed quadratic forms
\cite[Chapter~2]{wainwright2019high}, with probability at least $1-\delta/2$, for
$A=I_d-\Lambda_i^{\rm sp}-C_0\rho^{-h}I_d$ and
$A=[I_d-\Lambda_i^\star]_{G^h}$, $i=1,2$,
\begin{equation}\label{gauss-sample-event-linear}
\left|\left\langle \widehat\Sigma_i-\Sigma_i,A\right\rangle\right|
\le
C_0\left(
\sqrt{\frac{d\log(1/\delta)}{N}}
+
\frac{d+\log(1/\delta)}{N}
\right).
\end{equation}
We work on the intersection of these two events, which has probability at least
$1-\delta$. The matrices appearing in
\eqref{gauss-sample-event-linear} have Frobenius norm $O(\sqrt d)$. The first two
also have bounded operator norm by \eqref{gauss-sample-sparse-approx}; for
$[I_d-\Lambda_i^\star]_{G^h}$ we use the crude operator bound $O(\sqrt d)$, and
the condition $d+|E(G^h)|+\log(1/\delta)\le cN$ absorbs the corresponding
operator-norm contribution into the displayed right-hand side, since
$\sqrt d\,\log(1/\delta)/N\le C_0\sqrt{d\log(1/\delta)/N}$ by the preceding
choice of $C_0$.

The smallness condition \eqref{smcond} in the theorem and \eqref{gauss-sample-event-entry} imply
\begin{equation}\label{gauss-sample-op-small}
\left\|[\widehat\Sigma_i-\Sigma_i]_{G^h}\right\|_2
\le
(1+\Delta_h)
\left\|[\widehat\Sigma_i-\Sigma_i]_{G^h}\right\|_\infty
\le a/2,
\qquad i=1,2.
\end{equation}

We prove the upper bound first. The matrix
\begin{equation}
\Sigma_i+[\widehat\Sigma_i-\Sigma_i]_{G^h}
\notag
\end{equation}
is positive definite by \eqref{gauss-sample-op-small} and matches $\widehat\Sigma_i$ on the
entries of $G^h$. Therefore it is a feasible covariance completion for the empirical SDP.
Hence the covariance part of the empirical SDP is bounded above by
\begin{equation}\label{GSform}
\Tr(\bar\Sigma_1+\bar\Sigma_2)
-2\Tr\left(\bar\Sigma_1^{1/2}\bar\Sigma_2\bar\Sigma_1^{1/2}\right)^{1/2},\qquad
\bar\Sigma_i:=\Sigma_i+[\widehat\Sigma_i-\Sigma_i]_{G^h},\quad i=1,2.
\end{equation}
We consider the above term \eqref{GSform} as a function of the perturbation term $[\widehat\Sigma_i-\Sigma_i]_{G^h}.$ It is twice continuously differentiable in the
spectral neighborhood determined by \eqref{gauss-sample-op-small}. After Taylor expansion, its second-order
remainder is bounded by
\begin{equation}
C_0\sum_{i=1}^2\left\|[\widehat\Sigma_i-\Sigma_i]_{G^h}\right\|_F^2\leq C_0^2\frac{d+|E(G^h)|+\log (1/\delta)}{N}.
\notag
\end{equation}
At $(\Sigma_1,\Sigma_2)$, the first-order term in the dense Gaussian covariance
value is
\begin{align}
\sum_{i=1}^2
\left\langle I_d-\Lambda_i^\star,[\widehat\Sigma_i-\Sigma_i]_{G^h}\right\rangle
&=
\sum_{i=1}^2
\left\langle [I_d-\Lambda_i^\star]_{G^h},\widehat\Sigma_i-\Sigma_i\right\rangle
\notag\\
&\le
C_0\left(
\sqrt{\frac{d\log(1/\delta)}{N}}
+
\frac{d+\log(1/\delta)}{N}
\right),
\notag
\end{align}
where the last inequality uses \eqref{gauss-sample-event-linear}. Therefore
\begin{align}\label{gauss-sample-upper-cov}
\phi_{G^h}(\widehat\Sigma_1,\widehat\Sigma_2)
&\le
\Tr(\Sigma_1+\Sigma_2)-2\Tr\left(\Sigma_1^{1/2}\Sigma_2\Sigma_1^{1/2}\right)^{1/2}
\notag\\
&\quad+
C_0\left(
\sqrt{\frac{d\log(1/\delta)}{N}}
+
\frac{d+|E(G^h)|+\log(1/\delta)}{N}
\right).
\end{align}

We next prove the lower bound. By \eqref{gauss-sample-sparse-approx} and the proof of
Lemma~\ref{propapp}, the matrices
\begin{equation}
\widetilde\Lambda_i:=\Lambda_i^{\rm sp}+C_0\rho^{-h}I_d,
\qquad i=1,2,
\notag
\end{equation}
are feasible for the dual problem \eqref{gauss-cov-dual} with $\G=G^h$, and they satisfy
\begin{align}\label{gauss-sample-cert-true}
&\Tr(\Sigma_1)+\Tr(\Sigma_2)
-\langle \Sigma_1,\widetilde\Lambda_1\rangle
-\langle \Sigma_2,\widetilde\Lambda_2\rangle 
\notag\\
&\qquad\ge
\Tr(\Sigma_1+\Sigma_2)-2\Tr\left(\Sigma_1^{1/2}\Sigma_2\Sigma_1^{1/2}\right)^{1/2}
-Cd\rho^{-h}.
\end{align}
The same certificate is feasible for the empirical dual problem, since the feasible set does
not depend on the covariance input. Hence
\begin{align}\label{gauss-sample-lower-start}
&\phi_{G^h}(\widehat\Sigma_1,\widehat\Sigma_2)
\ge
\Tr(\widehat\Sigma_1)+\Tr(\widehat\Sigma_2)
-\langle \widehat\Sigma_1,\widetilde\Lambda_1\rangle
-\langle \widehat\Sigma_2,\widetilde\Lambda_2\rangle 
\notag\\
&=
\Tr(\Sigma_1)+\Tr(\Sigma_2)
-\langle \Sigma_1,\widetilde\Lambda_1\rangle
-\langle \Sigma_2,\widetilde\Lambda_2\rangle+
\sum_{i=1}^2
\left\langle \widehat\Sigma_i-\Sigma_i,I_d-\widetilde\Lambda_i\right\rangle.
\end{align}
Using \eqref{gauss-sample-event-linear} in \eqref{gauss-sample-lower-start} together with
\eqref{gauss-sample-cert-true}, we obtain
\begin{align}\label{gauss-sample-lower-cov}
\phi_{G^h}(\widehat\Sigma_1,\widehat\Sigma_2)
&\ge
\Tr(\Sigma_1+\Sigma_2)-2\Tr\left(\Sigma_1^{1/2}\Sigma_2\Sigma_1^{1/2}\right)^{1/2}
\notag\\
&\quad-
C_0\left(
 d\rho^{-h}
+
\sqrt{\frac{d\log(1/\delta)}{N}}
+
\frac{d+\log(1/\delta)}{N}
\right).
\end{align}
Combining \eqref{gauss-sample-upper-cov} and \eqref{gauss-sample-lower-cov} gives the
covariance part of the desired estimate.

It remains to control the mean term. Set
$e=(\widehat m_1-m_1)-(\widehat m_2-m_2)$. Then
\begin{align}\label{gauss-sample-mean}
\left|\|\widehat m_1-\widehat m_2\|^2-\|m_1-m_2\|^2\right|
&=
\left|2\langle m_1-m_2,e\rangle+\|e\|^2\right|
\notag\\
&\le
2\left|\langle m_1-m_2,e\rangle\right|+
2\|\widehat m_1-m_1\|^2+2\|\widehat m_2-m_2\|^2
\notag\\
&\le
C_0\left(
\sqrt{\frac{d\log(1/\delta)}{N}}
+
\frac{d+\log(1/\delta)}{N}
\right),
\end{align}
where the last inequality uses \eqref{gauss-sample-event-mean} and
\eqref{gauss-sample-event-mean-linear}.
Adding \eqref{gauss-sample-mean} to the covariance estimate proves
\eqref{gauss-sample-total}.
\end{proof}
}

\medskip

{
\noindent{\bf Proof of Proposition~\ref{prop:samplecomplex}}

\begin{proof}
Let $\mathcal{F}_x$ be the finite collection of all scalar monomials that appear as entries of
${\rm R}_x(\Phi_k\Phi_k^\top)$ and ${\rm R}_x(\Phi_i\Phi_j^\top)$ for $ij\in E(\G)$.
Define $\mathcal{F}_y$ analogously from
${\rm R}_y(\Phi_k\Phi_k^\top)$ and ${\rm R}_y(\Phi_i\Phi_j^\top)$. We use the disjoint union
\begin{equation}\label{samplepf-class}
\mathcal F:=\mathcal F_x\sqcup\mathcal F_y ,
\end{equation}
so an $x$-moment and a $y$-moment are treated as two different elements even if they are
represented by the same monomial. This keeps track of which sample cloud is used to estimate
the moment.
By construction,
\begin{equation}\label{samplepf0}
|\mathcal{F}_x|\le M_{\rm mom},
\qquad
|\mathcal{F}_y|\le M_{\rm mom}.
\end{equation}
Thus
\begin{equation}\label{samplepf-card}
|\mathcal F|\le 2M_{\rm mom}.
\end{equation}
Every element of $\mathcal F$ is a monomial of total degree at most $2n$ and is supported
on either one cluster or one edge of the reference graph. Hence it depends on at most
$2r$ coordinates. For $f\in\mathcal F$, let
\begin{equation}\label{samplepf-support}
S_f:=\{i\in[d]: f(z)\ {\rm depends\ on}\ z_i\}
\end{equation}
denote its coordinate support. Then $|S_f|\le 2r$.

Throughout the proof, $C$ denotes a positive constant depending only on
$n,r,a_{\rm tail},A_{\rm tail}$. It may be enlarged in estimates, but the radius
$R_N$ below is fixed once and for all with one sufficiently large choice. In particular,
we choose the constant in $R_N$ so that
\begin{equation}\label{samplepf-C-choice}
a_{\rm tail}C^2\ge 1,
\qquad
c_{\rm tail}C^2\ge 1,
\qquad
C^2\log 8\ge 1,
\end{equation}
Assumption~\ref{assump:sampletail} at $R=0$ gives
$1=\mu(|x_i|\ge0)\le A_{\rm tail}$, and the same argument applies to $\nu$.
Hence $A_{\rm tail}\ge1$, so
$\log(8A_{\rm tail}dN)+\log(1/\delta)\ge\log 8>0$. Set
\begin{equation}\label{samplepf-radius}
R_N:=
C
\left[\log(8A_{\rm tail}dN)+\log(1/\delta)\right]^{1/2}.
\end{equation}
By \eqref{samplepf-C-choice}, we also have $R_N\ge1$.
With probability at least $1-\delta/4$, all coordinates of all samples
$x^{(\ell)}$ and $y^{(\ell)}$, $\ell\in[N]$, are bounded in absolute value by $R_N$. Indeed,
let
\begin{equation}\label{samplepf-event}
E_R:=
\left\{
\max_{\ell\in[N]}\max_{i\in[d]} |x_i^{(\ell)}|\le R_N
\quad\hbox{and}\quad
\max_{\ell\in[N]}\max_{i\in[d]} |y_i^{(\ell)}|\le R_N
\right\}.
\end{equation}
Then Assumption~\ref{assump:sampletail} gives
\begin{align}\label{samplepf-event-prob}
\mathbb{P}(E_R^c)
&\le
\sum_{\ell=1}^N\sum_{i=1}^d
\mathbb{P}\bigl(|x_i^{(\ell)}|>R_N\bigr)
+
\sum_{\ell=1}^N\sum_{i=1}^d
\mathbb{P}\bigl(|y_i^{(\ell)}|>R_N\bigr)  \\
&\le 2dN A_{\rm tail}\exp(-a_{\rm tail}R_N^2). \notag
\end{align}
Since $R_N$ is defined by \eqref{samplepf-radius} and $C$ is chosen so that
$a_{\rm tail}C^2\ge 1$, we have
\begin{equation}\label{samplepf-event-final}
\mathbb{P}(E_R^c)
\le
2dN A_{\rm tail}
\exp\!\left(-\log\!\left(\frac{8A_{\rm tail}dN}{\delta}\right)\right)
= \delta/4.
\end{equation}
Thus $\mathbb{P}(E_R)\ge 1-\delta/4$.

Next truncate each local monomial on its own support. For $f\in\mathcal F$, define
\begin{equation}\label{samplepf-trunc-def}
f_R(z):=
f(z)\,
\mathbf 1_{\{|z_i|\le R_N\ {\rm for\ every}\ i\in S_f\}},
\qquad z\in\R^d.
\end{equation}
If $S_f=\varnothing$, the condition in the indicator is interpreted as true. Thus
$f_R=f$ whenever all coordinates in $S_f$ are bounded by $R_N$, and $f_R=0$ otherwise.
Since $f$ is a monomial of total degree at most $2n$, we have
\begin{equation}\label{samplepf-trunc-bound}
|f_R(z)|\le R_N^{2n},
\qquad z\in\R^d.
\end{equation}
We now derive the concentration estimate for the truncated moments. For
$f\in\mathcal F$, let $U^{(\ell)}=x^{(\ell)}$ if
$f$ belongs to the $\mathcal F_x$ copy of the disjoint union and
$U^{(\ell)}=y^{(\ell)}$ if $f$ belongs to the $\mathcal F_y$ copy.
For fixed $f$, the variables $\xi_\ell^f:=f_R(U^{(\ell)})$ are independent and
satisfy $-R_N^{2n}\le \xi_\ell^f\le R_N^{2n}$ by \eqref{samplepf-trunc-bound}.
Thus Lemma~\ref{lem:hoeffding} implies
\begin{equation}\label{samplepf-hoeffding}
\mathbb{P}\!\left(
\left|
\frac1N\sum_{\ell=1}^N f_R(U^{(\ell)})-\mathbb{E}f_R(U)
\right|\ge t
\right)
\le
2\exp\!\left(-\frac{Nt^2}{2R_N^{4n}}\right).
\end{equation}
Combining \eqref{samplepf-card} with the union bound gives
\begin{equation}\label{samplepf-union}
\mathbb{P}\!\left(
\sup_{f\in\mathcal F}
\left|
\frac1N\sum_{\ell=1}^N f_R(U^{(\ell)})-\mathbb{E}f_R(U)
\right|\ge t
\right)
\le
4M_{\rm mom}\exp\!\left(
-\frac{Nt^2}{2R_N^{4n}}
\right).
\end{equation}
Define
\begin{equation}\label{samplepf-tN}
t_N:=C R_N^{2n}
\sqrt{\frac{\log M_{\rm mom}+\log(1/\delta)}{N}} .
\end{equation}
With $C$ chosen sufficiently large at the start of the proof, substituting $t=t_N$ in
\eqref{samplepf-union} gives
\begin{align}\label{samplepf-union-final}
4M_{\rm mom}\exp\!\left(
-\frac{Nt_N^2}{2R_N^{4n}}
\right)
&=
4M_{\rm mom}\exp\!\left(
-\frac{C^2}{2}\bigl(\log M_{\rm mom}+\log(1/\delta)\bigr)
\right) \notag\\
&\le \frac{3\delta}{4}.
\end{align}
Therefore, with probability at least $1-3\delta/4$,
\begin{equation}\label{samplepf1}
\sup_{f\in\mathcal F}
\left|
\frac1N\sum_{\ell=1}^N f_R\bigl(U^{(\ell)}\bigr)-\mathbb{E}f_R(U)
\right|
\le
t_N.
\end{equation}

It remains to compare $f_R$ with $f$ in expectation. By Lemma~\ref{lem:tailtrunc},
\begin{equation}\label{samplepf-tailbias}
\sup_{f\in\mathcal F}
\mathbb{E}|f(U)-f_R(U)|
\le
C_{\rm tail}R_N^{2n}e^{-c_{\rm tail}R_N^2}.
\end{equation}
The right-hand side of \eqref{samplepf-tailbias} is of smaller order than the
concentration scale. Indeed, by \eqref{samplepf-C-choice} and
$R_N$ as defined in \eqref{samplepf-radius},
\begin{equation}\label{samplepf-tail-exp}
e^{-c_{\rm tail}R_N^2}
=
\exp\!\left(
-c_{\rm tail}C^2
\log\!\left(\frac{8A_{\rm tail}dN}{\delta}\right)
\right)
\le
\exp\!\left(-\log\!\left(\frac{8A_{\rm tail}dN}{\delta}\right)\right).
\end{equation}
Therefore
\begin{equation}\label{samplepf-tail-N}
\exp\!\left(-\log\!\left(\frac{8A_{\rm tail}dN}{\delta}\right)\right)
=
\frac{\delta}{8A_{\rm tail}dN}
\le
N^{-1},
\end{equation}
where the last inequality uses $\delta\le1$, $d\ge1$, and $A_{\rm tail}\ge1$.
In the nontrivial monomial relaxations considered here, $M_{\rm mom}\ge2$; if the basis
contains only constants, all moment errors are zero and the proposition is immediate.
Thus $\log M_{\rm mom}+\log(1/\delta)\ge\log2$.
Consequently, for $N\ge1$,
\begin{equation}\label{samplepf-N-to-log}
N^{-1}
\le
C\sqrt{\frac{\log M_{\rm mom}+\log(1/\delta)}{N}},
\end{equation}
The numerical constant in this inequality is absorbed into the same $C$ chosen at the
start of the proof. Combining \eqref{samplepf-tailbias}--\eqref{samplepf-N-to-log}
gives
\begin{equation}\label{samplepf-tailbias-final}
C_{\rm tail}R_N^{2n}e^{-c_{\rm tail}R_N^2}
\le
C R_N^{2n}\sqrt{\frac{\log M_{\rm mom}+\log(1/\delta)}{N}}
\le
t_N.
\end{equation}
Now take the intersection of the event $E_R$ in \eqref{samplepf-event} and the concentration
event described before \eqref{samplepf1}. By \eqref{samplepf-event-final} and the
probability estimate preceding \eqref{samplepf1}, this intersection has probability at
least $1-\delta$. On $E_R$, for every sample
$U^{(\ell)}$ and every $f\in\mathcal F$, all coordinates in $S_f$ are bounded by $R_N$.
Therefore, by \eqref{samplepf-trunc-def},
\begin{equation}\label{samplepf-empirical-equality}
f_R(U^{(\ell)})=f(U^{(\ell)}),
\qquad
f\in\mathcal F,\ \ell\in[N].
\end{equation}
Consequently,
\begin{equation}\label{samplepf-average-equality}
\frac1N\sum_{\ell=1}^N f(U^{(\ell)})
=
\frac1N\sum_{\ell=1}^N f_R(U^{(\ell)}).
\end{equation}
For each fixed $f$, we then use the triangle inequality:
\begin{align}\label{samplepf-triangle}
\left|
\frac1N\sum_{\ell=1}^N f(U^{(\ell)})-\mathbb{E}f(U)
\right|
&\le
\left|
\frac1N\sum_{\ell=1}^N f_R(U^{(\ell)})-\mathbb{E}f_R(U)
\right|  \\
&\quad+
\left|\mathbb{E}f_R(U)-\mathbb{E}f(U)\right| .
\end{align}
The first term in \eqref{samplepf-triangle} is controlled by \eqref{samplepf1}, after
using \eqref{samplepf-average-equality}; the second term is bounded by
\eqref{samplepf-tailbias-final}, which is obtained from the tail estimate
\eqref{samplepf-tailbias}. Taking the supremum over $f\in\mathcal F$ gives
\begin{equation}\label{samplepf-scalar}
\sup_{f\in\mathcal F}
\left|
\frac1N\sum_{\ell=1}^N f\bigl(U^{(\ell)}\bigr)-\mathbb{E}f(U)
\right|
\lesssim
\left[\log(8A_{\rm tail}dN)+\log(1/\delta)\right]^n
\sqrt{
\frac{\log M_{\rm mom}+\log(1/\delta)}{N}
},
\end{equation}
where we used
$R_N^{2n}=C^{2n}\left[\log(8A_{\rm tail}dN)+\log(1/\delta)\right]^n$
from \eqref{samplepf-radius}, with the resulting constant absorbed by $\lesssim$.

Finally, each entry of
$\widehat{M}_k^x-M_k^x$, $\widehat{M}_{ij}^x-M_{ij}^x$,
$\widehat{M}_k^y-M_k^y$, and $\widehat{M}_{ij}^y-M_{ij}^y$ is exactly the empirical
minus true expectation of one function in $\mathcal F$. Therefore \eqref{samplepf-scalar}
implies the matrix-entry bound \eqref{samplemain}.
\end{proof}
}

\section*{Acknowledgments.}
The authors thank Siyao Yang for assistance with the sampling procedures and Yifan Peng for providing the neural network code used in our numerical experiments. In particular, the first author, Yuehaw Khoo, would like to thank Gero Friesecke for their valuable discussions on high-dimensional optimal transport.

\bibliographystyle{abbrv}
\bibliography{OT}

\end{document}